\documentclass{amsart}
\usepackage{amsmath,amssymb,amsthm,geometry,graphics,color,mdwlist,mathrsfs}
\usepackage[all]{xy}

\DeclareMathOperator{\Coker}{Coker}
\DeclareMathOperator*{\colim}{colim}
\DeclareMathOperator{\hocolim}{hocolim}

\DeclareMathOperator{\Hom}{Hom}

\DeclareMathOperator{\RHom}{\mathrm{R}\!\Hom}

\DeclareMathOperator{\End}{End}
\DeclareMathOperator{\sEnd}{\underline{\End}}

\DeclareMathOperator{\Ext}{Ext}

\DeclareMathOperator{\pd}{proj.dim}

\DeclareMathOperator{\gd}{gl.dim}

\DeclareMathOperator{\md}{mod}
\DeclareMathOperator{\smd}{\underline{\md}}
\DeclareMathOperator{\injsmd}{\overline{\md}}
\DeclareMathOperator{\Md}{Mod}
\DeclareMathOperator{\proj}{proj}
\DeclareMathOperator{\inj}{inj}

\DeclareMathOperator{\thick}{thick}
\DeclareMathOperator{\per}{per}

\DeclareMathOperator{\ind}{ind}
\DeclareMathOperator{\add}{add}

\DeclareMathOperator{\CM}{CM}
\DeclareMathOperator{\sCM}{\underline{\CM}}

\DeclareMathOperator{\cone}{cone}

\DeclareMathOperator{\SL}{SL}
\DeclareMathOperator{\GL}{GL}
\DeclareMathOperator{\diag}{diag}
\def\lotimes{\otimes^\mathrm{L}}

\def\A{\mathscr{A}}
\def\B{\mathscr{B}}
\def\C{\mathscr{C}}
\def\D{\mathscr{D}}

\def\T{\mathscr{T}}
\def\U{\mathscr{U}}
\def\V{\mathscr{V}}

\def\L{\Lambda}
\def\G{\Gamma}

\def\a{\alpha}

\def\s{\sigma}

\def\Z{\mathbb{Z}}

\def\discoprod{\displaystyle\coprod}

\def\op{\mathrm{op}}

\def\rsimeq{\rotatebox{-90}{$\simeq$}}

\newtheorem{Thm}{Theorem}[section]
\newtheorem{Lem}[Thm]{Lemma}
\newtheorem{Prop}[Thm]{Proposition}
\newtheorem{Cor}[Thm]{Corollary}

\newtheorem{Prop-Def}[Thm]{Proposition-Definition}
\newtheorem{Thm-Def}[Thm]{Theorem-Definition}

\theoremstyle{definition}
\newtheorem{Def}[Thm]{Definition}
\newtheorem{Ex}[Thm]{Example}
\newtheorem{Con}[Thm]{Construction}

\theoremstyle{remark}
\newtheorem{Rem}[Thm]{Remark}

\makeatletter

\@addtoreset{equation}{section}
\makeatother

\title{Morita theorem for hereditary Calabi-Yau categories}
\author{Norihiro Hanihara}
\thanks{This work is supported by JSPS KAKENHI Grant Number JP19J21165}
\address{Graduate School of Mathematics, Nagoya University, Furocho, Chikusa-ku, Nagoya, 464-8602, Japan}
\email{m17034e@math.nagoya-u.ac.jp}

\begin{document}
\begin{abstract}
We give a structure theorem for Calabi-Yau triangulated category with a hereditary cluster tilting object. We prove that an algebraic $d$-Calabi-Yau triangulated category with a $d$-cluster tilting object $T$ such that its shifted sum $T\oplus\cdots\oplus T[-(d-2)]$ has hereditary endomorphism algebra $H$ is triangle equivalent to the orbit category $\D^b(\md H)/\tau^{-1/(d-1)}[1]$ of the derived category of $H$ for a naturally defined $(d-1)$-st root $\tau^{1/(d-1)}$ of the AR translation, provided $H$ is of non-Dynkin type. We also show that hereditaryness of $H$ follows from that of $T$ when $d=3$, that of $T\oplus T[-1]$ when $d=4$, and similarly from a smaller endomorphism algebra for higher dimensions under vanishing of some negative self-extensions of $T$. Our result therefore generalizes the established theorems by Keller--Reiten and Keller--Murfet--Van den Bergh. Furthermore, we show that enhancements of such triangulated categories are unique. Finally we apply our results to Calabi-Yau reductions of a higher cluster category of a finite dimensional algebra and of the singularity category of an invariant subring. 
\end{abstract}

\maketitle
\setcounter{tocdepth}{1}
\tableofcontents
\section{Introduction}
Calabi-Yau (CY) triangulated categories have been of great importance in various areas of mathematics. A triangulated category $\T$ is $d$-CY if it has finite dimensional morphism spaces over a field $k$, and the $d$-th suspension $[d]$ is the Serre functor on $\T$, that is, we have functorial isomorphisms
\[ \T(A,B)\simeq D\T(B,A[d]) \]
for all $A,B\in\T$, where $D=\Hom_k(-,k)$. In representation theory, CY triangulated categories with cluster tilting objects are of particular interest \cite{Ke08}. One of the foundations of recent developments on this subject was given in the context of categorification of Fomin--Zelevinsky's cluster algebras \cite{CA1} by Buan--Marsh--Reineke--Reiten--Todorov \cite{BMRRT}. Generalizing their construction, the {\it $d$-cluster category} of a hereditary algebra $H$ is the orbit category
\[ \D^b(\md H)/\tau^{-1}[d-1]. \]
It is a $d$-CY triangulated category \cite{Ke05} equipped with a $d$-cluster tilting object \cite{IYo}.

The subject of this paper is Morita-type theorem for CY categories. Most classically, Morita theory gives a characterization of module categories among abelian categories in terms of projective generators \cite{Ga}. Similarly, its version for triangulated categories specifies derived categories in terms of tilting objects \cite{Ke94}. We investigate its analogue for CY triangulated categories, which attempts to characterize cluster categories in terms of cluster tilting objects.

The only known results on such Morita-type theorems are the following two, due to Keller--Reiten and Keller--Murfet--Van den Bergh.
\begin{Thm}\label{known}
Let $\T$ be an algebraic $d$-CY triangulated category with a $d$-cluster tilting object $T$.
\begin{enumerate}
\item\cite{KRac} Suppose $\End_\T(T)=kQ$ for some acyclic quiver $Q$, and $\T(T,T[-i])=0$ for $0< i<d-1$. Then there exists a triangle equivalence $\T\simeq\D^b(\md kQ)/\tau^{-1}[d-1]$.
\item\cite{KMV} Suppose $d=2n+1\geq3$, $\End_\T(T)=k$, and $\T(T,T[-i])=0$ for $0<i<n$. Put $\dim_k\Hom_\T(T,T[-n])=m$. Then there exists a triangle equivalence $\T\simeq\D^b(\md kQ_m)/\tau^{-1/2}[n]$ for the $m$-Kronecker quiver $Q_m\colon\xymatrix{\circ\ar@3[r]&\circ}$ with $m$ arrows, and a naturally defined square root $\tau^{1/2}$ of the AR translation.
\end{enumerate}
\end{Thm}

This paper is devoted to prove the following result which encompasses both of the above two cases, providing a general Morita-type theorem for CY categories arising from representation-{\it infinite} hereditary algebras. Recall that a hereditary algebra is {\it $1$-representation infinite} if any of its ring indecomposable summand is representation-infinite. Also we denote by $J_\L$ the Jacobson radical of a ring $\L$.
\begin{Thm}[=\ref{root}]\label{main}
Let $d\geq2$ and let $\T$ be an algebraic $d$-CY triangulated category with a $d$-cluster tilting object $T$. Assume that $H=\End_\T(T\oplus T[-1]\oplus\cdots\oplus T[-(d-2)])$ is $1$-representation infinite and that $H/J_H$ is separable over $k$. Then there exists a triangle equivalence
\[ \T\simeq\D^b(\md H)/\tau^{-1/(d-1)}[1] \]
for a naturally defined $(d-1)$-st root $\tau^{1/(d-1)}$ of the AR translation.
\end{Thm}
We also give a certain converse in Section \ref{comb}. We show that if $H$ has a $(d-1)$-st root of $\tau$, then there exists a projective $H$-module $P$ such that $P\oplus\tau^{-1/(d-1)}P\oplus\cdots\oplus \tau^{-(d-2)/(d-1)}P\simeq H$, justifying our choice of $H$.

Our $(d-1)$-st root $\tau^{-1/(d-1)}$ is defined by $-\lotimes_HU$ for the $(H,H)$-bimodule $U=\T(X,X[-1])$ where $X=T\oplus\cdots\oplus T[-(d-2)]$. We show that the $(d-1)$-fold derived tensor product of $U$ is isomorphic to $\Ext^1_H(DH,H)$ (\ref{AR}), thus one can write $-\lotimes_HU=\tau^{-1/(d-1)}$. A square root of the AR translation for generalized Kronecker quivers appears in \cite{KMV}. More general roots of (higher) AR translations for various algebras are studied in \cite{ha3}.

The main theorem recovers both of the known results in \ref{known} for the infinite type.
When $d=2$ this immediately gives Keller--Reiten's recognition theorem \ref{known}(1) for the case $Q$ is of non-Dynkin type. Also when $d=3$ and $\End_\T(T)=k$, then $H$ is the path algebra of a generalized Kronecker quiver, thus we deduce Keller--Murfet--Van den Bergh's theorem \ref{known}(2) for $m\geq2$. We can also recover the versions for arbitrary $d$ in \ref{known}, despite the seemingly different assumption and conclusion; see the second paragraph after \ref{maincor}.

Even if $H$ is of Dynkin type we have a partial result. For arbitrary hereditary $H$ there is an additive equivalence (\ref{st})
\[ \T/[T[1]\oplus T\oplus\cdots\oplus T[-(d-2)]]\simeq\smd H, \]
which gives a classification $\ind\T\simeq\ind(\md H)\sqcup\ind(\add T[1])$ of objects of $\T$, as \ref{main} implies for non-Dynkin cases. Here $\ind\C$ means the set of isomorphism classes of indecomposable objects in an additive category $\C$. In particular, the triangulated category $\T$ has finitely many indecomposable objects up to isomorphism when $H$ is of Dynkin type. For such a triangulated category its structure is deeply known, see for example \cite{XZ,Am07,Ke18,ha,Mu20}.

Furthermore, we remark that our proof of the main result in fact gives the uniqueness of enhancements for triangulated categories as in \ref{main}, in particular for the (usual) cluster category $\D^b(\md kQ)/\tau^{-1}[d-1]$ of a quiver $Q$. See \ref{enh} for the definition of enhancements and a precise statement. We refer to \cite{LO,Ant,Mu20,CNS}, and references therein, for various established results on uniqueness of enhancements.

\bigskip
Although the assumption in \ref{main} that $H=\End_\T(T\oplus\cdots\oplus T[-(d-2)])$ is hereditary looks strong, it in fact follows from the same property for a smaller algebra, in particular, from that of $T$ when $d=3$. More generally, we have the following sufficient condition for $H$ to be hereditary using vanishing of some negative self-extensions of $T$.
\begin{Thm}\label{hered}
Let $\T$ be a $d$-CY triangulated category with a $d$-cluster tilting object $T$.
\begin{enumerate}
	\item{\rm (=\ref{half})} Suppose $\T(T,T[-i])=0$ for $0<i<d/2$ and $H^\prime=\End_\T(T)$ is hereditary. Then we have $\T(T,T[-i])=0$ for $0< i<d-1$, thus $\End_\T(T\oplus\cdots\oplus T[-(d-2)])$ is the direct product of $(d-1)$ copies of $H^\prime$, hence is hereditary.
	\item{\rm (=\ref{oddcy})} Suppose $d=2n+1\geq3$, $\T(T,T[-i])=0$ for $0<i<n$, and $\End_\T(T)$ is hereditary. Then $\End_\T(T\oplus T[-n])$ is hereditary, and $\End_\T(T\oplus\cdots\oplus T[-(2n-1)])$ is the direct product of $n$ copies of it, hence is hereditary.
	\item{\rm (=\ref{evency})} Suppose $d=2n+2\geq4$, $\T(T,T[-i])=0$ for $0<i<n$, and $\End_\T(T\oplus T[-n])$ is hereditary. Then $H=\End_\T(T\oplus\cdots\oplus T[-2n])$ is also hereditary.
\end{enumerate}
\end{Thm}
Note that we are allowing $\T(T,T[-n])$ to survive in (2) and (3) while it is supposed to vanish in (1). One can thus view (1) as a `degenerate' version of (2) and (3). We also explicitly describe the quiver of $\End_\T(T\oplus\cdots\oplus T[-(d-2)])$ for (2) and (3) in terms of AR sequences in $\add T$, see \ref{quiver} and \ref{quiver2}.

Combining \ref{main} and \ref{hered}, together with an interpretation of $(d-1)$-th root of the AR translation for (1) and (2) which we explain below, we obtain the following version of \ref{main}. %under vanishing of negative self-extensions.
\begin{Cor}\label{maincor}
Let $\T$ be an algebraic $d$-CY triangulated category with a $d$-cluster tilting object $T$.
\begin{enumerate}
	\item{\rm (=\ref{kr})} Suppose $H^\prime=\End_\T(T)$ is hereditary and $\T(T,T[-i])=0$ for $0<i<d/2$. Then there exists a triangle equivalence
	\[ \T\simeq\D^b(\md H^\prime)/\tau^{-1}[d-1] \]
	when $H^\prime$ is $1$-representation infinite and $H^\prime/J_{H^\prime}$ is separable over $k$.
	\item{\rm (=\ref{kmv})} Suppose $d=2n+1\geq3$, $\End_\T(T)$ is hereditary, and $\T(T,T[-i])=0$ for $0<i<n$. Then $H^\prime=\End_\T(T\oplus T[-n])$ is hereditary, and there exists a triangle equivalence
	\[ \T\simeq\D^b(\md H^\prime)/\tau^{-1/2}[n] \]
	when $H^\prime$ is $1$-representation infinite and $H^\prime/J_{H^\prime}$ is separable over $k$.
	\item{\rm (=\ref{new})} Suppose $d=2n+2\geq4$, $\T(T,T[-i])=0$ for $0<i<n$, and $\End_\T(T\oplus T[-n])$ is hereditary. Then $H=\End_\T(T\oplus\cdots\oplus T[-2n])$ is hereditary, and there exists a triangle equivalence
	\[ \T\simeq\D^b(\md H)/\tau^{-1/(2n+1)}[1] \]
	when $H$ is $1$-representation infinite and $H/J_H$ is separable over $k$.
\end{enumerate}
\end{Cor}
The above (1) shows that we in fact need only half of the vanishings of negative extensions in \ref{known}(1). Also (2) reduces to Keller--Murfet--Van den Bergh's theorem \ref{known}(2) when $\End_\T(T)=k$ as well as relaxes the assumption on the vanishing of negative extensions in Keller--Reiten's theorem \ref{known}(1). This \ref{maincor}(2) is thus a common generalization of \ref{known}(1)(2). Let us summarize the implications diagrammatically.
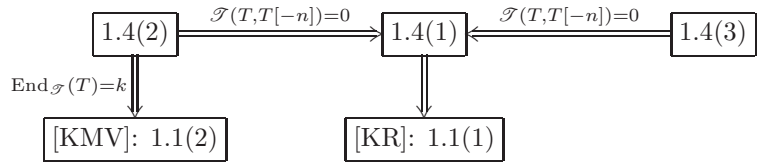
\[ \xymatrix@!C=30mm{
	*+[F]{\text{\ref{maincor}(2)}}\ar@{=>}[d]_{\End_\T(T)=k}\ar@{=>}[r]^-{\T(T,T[-n])=0}&*+[F]{\text{\ref{maincor}(1)}}\ar@{=>}[d]&*+[F]{\text{\ref{maincor}(3)}}\ar@{=>}[l]_-{\T(T,T[-n])=0}\\
	*+[F]{\text{\cite{KMV}: \ref{known}(2)}}&*+[F]{\text{\cite{KRac}: \ref{known}(1)}}& } \]

Now we briefly explain how to deduce the corollaries in \ref{maincor} from \ref{main} and the respective results in \ref{hered}. It depends on an interpretation of the $(d-1)$-st root $\tau^{1/(d-1)}$.
Under the assumption of \ref{maincor}(1), the algebra $H=\End_\T(T\oplus\cdots\oplus T[-(d-2)])$ in \ref{main} is the direct product of $(d-1)$ copies of $\End_\T(T)=:H^\prime$ by \ref{hered}(1). It turns out, under the identification $\D^b(\md H)=\D^b(\md H^\prime)\times\cdots\times\D^b(\md H^\prime)$, that our $(d-1)$-st root is given just by $\tau^{1/(d-1)}\colon(L_1,\ldots,L_{d-1})\mapsto(L_2,\ldots,L_{d-1},\tau^\prime L_1)$ using the AR translation $\tau^\prime$ of $H^\prime$. This yields an equivalence of orbit categories (\ref{adj})
\[ \D^b(\md H)/\tau^{-1/(d-1)}[1]\simeq\D^b(\md H^\prime)/{\tau^\prime}^{-1}[d-1], \]
which gives \ref{maincor}(1). A similar interpretation of $2n$-th root of $H$ in terms of a square root for $H^\prime$ yields \ref{maincor}(2). Note that in \ref{maincor}(3) the $(2n+1)$-st root cannot be made easier in general.

\bigskip
Let us now mention one intermediate result toward proving \ref{main} which is a general consideration on realizing a triangulated category as a cluster category. It is based on a description of cluster categories in terms of differential graded (dg) algebras \cite{Ke06}. Recall that a dg algebra $\Pi$ is {\it homologically smooth} if $\Pi$ is perfect as a bimodule. In this case we have that $\D^b(\Pi)$, the derived category of dg $\Pi$-modules of finite dimensional total cohomology, is contained in the perfect derived category $\per\Pi$. For a homologically smooth dg algebra $\Pi$, we call the Verdier quotient
\[ \C(\Pi):=\per\Pi/\D^b(\Pi) \]
the {\it cluster category} of $\Pi$. The theory on such kind of categories was established by Amiot \cite{Am09} and generalized in \cite{Guo} for CY dg algebras \cite{G,Ke11}, giving rise to CY triangulated categories with cluster tilting objects. We refer to \cite{AMY,ART,AIR,AOim,AOac,AOce,BT,FM,ha3,IQ,IO13,IYa1,KY16,KY18,KY20,Ke11,Ki2,KQ,Pl,Pr,TV} for further studies on this subject. Although the cluster category in the above general sense is not CY nor have cluster tilting objects, those for certain non-CY dg algebras will be important for us. 

Let $\T$ be a triangulated category and $X\in\T$ satisfying the following.
\begin{enumerate}
\renewcommand{\labelenumi}{(\alph{enumi})}
\item $\T$ is $\Hom$-finite over $k$ and $\thick_\T X=\T$.
\item If $\T(X,Y[i])=0$ for $i\ll0$ then $Y=0$.
\item There exists an enhancement $\A$ of $\T$ such that the truncated derived endomorphism algebra $\Pi:=\RHom_\A(X,X)^{\leq0}$ is homologically smooth.
\end{enumerate}
For example if $X\in\T$ is a $d$-cluster tilting object for some $d\geq1$, then the above \ref{ct} indeed holds (see \ref{mot}). In this situation we can realize our triangulated category $\T$ as a cluster category of $\Pi$ in \ref{sm}.
\begin{Thm}[=\ref{str}]
Let $\T$ be a triangulated category satisfying (a), (b), and (c) above. Then $\T$ is triangle equivalent to the cluster category of $\Pi$.
\end{Thm}
Our main result \ref{main} is obtained as an application of this observation together with a description of $\C(\Pi)$ as an orbit category of a derived category (see \ref{cl}) and our method is thus quite different from those in \ref{known}. It should be noted that a similar result on realizing a triangulated category as a cluster category is obtained in \cite{KY20}, in the setting where $\T$ is CY and $X\in\T$ is cluster tilting. It would be interesting to investigate the relationship of these results. We refer also to \cite{Tab} for another result based on a different model.

\subsection*{Acknowledgement}
The author would like to thank his supervisor Osamu Iyama for many helpful discussions. He is also grateful to Bernhard Keller for stimulating comments and discussions.

\section{Hereditaryness of shifted sum of cluster tilting objects}
Let $\T$ be a $d$-CY triangulated category with a $d$-cluster tilting object $T$. While the object $T$, if $\End_\T(T)$ is hereditary, alone can recover the category $\T$ when $d=2$, the same does not hold in higher dimensions. Our generator for larger $d$ is the shifted sum $X=T\oplus\cdots\oplus T[-(d-2)]$ of $T$ which, if $\End_\T(X)$ is hereditary, turns out to be essential to recover the triangulated category $\T$, see \ref{root}.

The aim of this section is to give a sufficient condition for $\End_\T(X)$ to be hereditary in terms of much smaller endomorphism algebra under some vanishing of negative self-extensions of $T$ (\ref{oddcy} and \ref{evency}). Moreover we explicitly describe the quiver of $\End_\T(X)$ in each case (\ref{quiver} and \ref{quiver2}).

\subsection{Rigid objects with hereditary endomorphism algebras}\label{gen}
\newcommand{\Sa}[1]{S_a^{(#1)}}
\newcommand{\Ua}[1]{U_a^{(#1)}}
\newcommand{\Aa}[1]{A_a^{(#1)}}
\newcommand{\Ba}[1]{B_a^{(#1)}}
Some part of our discussion does not depend on the setup of cluster tilting object in a CY triangulated category. Let $\T$ be a triangulated category and $T\in\T$. For each indecomposable summand $T_a$ of $T$, define the objects $\Aa{i}$ and $\Sa{i}$ by the triangles
\[ \xymatrix@R=1mm{
	\Sa{1}\ar[r]&\Aa{0}\ar[r]^-{a_0}&T_a\ar[r]^-{c_0}&\Sa{1}[1] \\
	\Sa{i+1}\ar[r]&\Aa{i}\ar[r]^-{a_i}&\Sa{i}\ar[r]^-{\delta_i}&\Sa{i+1}[1] &\text{for } i\geq1 } \]
with $a_0$ the sink map in $\add T$, and $a_i$ the minimal right $(\add T)$-approximation for each $i\geq1$. We depict these triangles as a complex below.
\[ \xymatrix@R=2mm@!C=5mm{
	\cdots\ar[dr]\ar[rr]&&\Aa{2}\ar[dr]_-{a_2}\ar[rr]^-{f_2}&&\Aa{1}\ar[dr]_-{a_1}\ar[rr]^-{f_1}&&\Aa{0}\ar[dr]^-{a_0}&\\
	&\cdots\ar[ur]&&\Sa{2}\ar[ur]&&\Sa{1}\ar[ur]&&T_a } \]
Let $c_i\colon T_a[-i]\to\Sa{i+1}[1]$ be the composite of the connecting morphisms, precisely it is given by
\[ \xymatrix@C=12mm{c_i\colon T_a[-i]\ar[r]&\Sa{1}[-(i-1)]\ar[r]^-{\delta_1[-(i-1)]}&\Sa{2}[-(i-2)]\ar[r]&\cdots\ar[r]&\Sa{i}\ar[r]^-{\delta_i}&\Sa{i+1}[1] }, \]
and thus satisfies $c_i={\delta_i}\circ (c_{i-1}[-1])$.

An important observation is that for each $m\geq1$ we can determine the quiver of the endomorphism algebra of the shifted sum $T\oplus T[-1]\oplus\cdots\oplus T[-m]$ from the complex above when a smaller algebra $\End_\T(T\oplus\cdots\oplus T[-(m-1)])$ is hereditary. %We will later use this observation for $m\fallingdotseq d/2$ in the setting where $T\in\T$ is $d$-cluster tilting and $T\oplus T[-1]\oplus\cdots\oplus T[-(d-3)]$ has hereditary endomorphism algebra.
\begin{Prop}\label{induction}
Let $m\geq1$ and $T\in\T$ an $(m+1)$-rigid object such that $\End_\T(T\oplus\cdots\oplus T[-(m-1)])$ is hereditary. Then there exists an octahedral
\[ \xymatrix@R=4mm@C=10mm{
\Sa{m}[-1]\ar@{=}[d]\ar[r]^-{-\delta_m[-1]}&\Sa{m+1}\ar[r]\ar[d]&\Aa{m}\ar[d]\ar[r]^-{a_m}&\Sa{m}\ar@{=}[d]\\
\Sa{m}[-1]\ar[r]&\discoprod_{i=0}^{m-1}\Aa{i}[-(m-i)]\ar[d]\ar[r]^-{p_{m}}&T_a[-m]\ar[r]^-{c_{m-1}[-1]}\ar[d]&\Sa{m}\\
&\bullet\ar@{=}[r]&\bullet& } \]
with $p_{m}$ the sink map in $\add(T[-1]\oplus\cdots\oplus T[-m])$. Moreover, the triangle
\[ \xymatrix{ \Sa{m+1}\ar[r]&\discoprod_{i=0}^m\Aa{i}[-(m-i)]\ar[r]^-{q_m}&T_a[-m]\ar[r]^-{c_m}&\Sa{m+1}[1] } \]
obtained from the above homotopy cartesian square gives the sink map $q_m$ at $T_a[-m]$ in $\add(T\oplus\cdots\oplus T[-m])$.
\end{Prop}
\begin{proof}
	Consider the following statements.
	\begin{enumerate}
	\renewcommand{\labelenumi}{(\alph{enumi})$_l$}
	\renewcommand{\theenumi}{\alph{enumi}}
	\item\label{app} The morphism $c_{l-1}[-1]\colon T_a[-l]\to\Sa{l}$ induces a surjection $\T(T,T_a[-l])\to\T(T,\Sa{l})$.
	\item\label{msink} The conclusions of the proposition for $m=l$.
	\end{enumerate}
	\newcommand{\refm}[2]{(\ref{#1})$_{#2}$}
	We prove that \refm{msink}{l-1} and \refm{app}{l} implies \refm{msink}{l} when $l\leq m$, and \refm{msink}{l} gives \refm{app}{l+1} when $l\leq m-1$.	Starting with the sink map
	\[ \xymatrix@C=10mm{\Sa{1}\ar[r]&\Aa{0}\ar[r]^-{q_0:=a_0}&T_a\ar[r]&\Sa{1}[1] } \]
	which can be viewed as \refm{msink}{0}, this will prove our result \refm{msink}{m} by induction.
	
	Suppose first \refm{msink}{l} and $l\leq m-1$. Put $\widetilde{T}=T\oplus\cdots\oplus T[-l]$ and consider the exact sequence
	\[ \xymatrix{ \T(\widetilde{T},T_a[-l-1])\ar[r]^-{c_l[-1]}&\T(\widetilde{T},\Sa{l+1})\ar[r]&\T(\widetilde{T},\coprod_{i=0}^l\Aa{i}[-(l-i)])\ar[r]&\T(\widetilde{T},T_a[-l]) } \]
	obtained from the triangle in \refm{msink}{l}. Since $\End_\T(\widetilde{T})$ is hereditary the sink map in $\add\widetilde{T}$ is a monomorphism, which is to say that the last map in the above exact sequence is injective.	Therefore the first map, in particular its direct summand $\T(T,T_a[-l-1])\to\T(T,\Sa{l+1})$, is surjective. This shows \refm{app}{l+1}.
	
	Suppose next \refm{msink}{l-1} and \refm{app}{l}. Shifting the triangle in \refm{msink}{l-1} by $[-1]$ we have the triangle in the second row in the diagram below. By \refm{msink}{l-1} the map $q_l[-1]$ is the sink map in $\add(T[-1]\oplus\cdots\oplus T[-l])$. We compare it with the triangle in the first row. By \refm{app}{l}, we can lift $a_l\colon \Aa{l}\to\Sa{l}$ to $T_a[-l]$, which can be completed to a desired octahedral. This gives the first part of \refm{msink}{l}.
	\[ \xymatrix@R=4mm@C=10mm{
		\Sa{l}[-1]\ar@{=}[d]\ar[r]^-{-\delta_l[-1]}&\Sa{l+1}\ar[r]\ar@{-->}[d]&\Aa{l}\ar@{-->}[d]\ar[r]^-{a_l}&\Sa{l}\ar@{=}[d]\\
		\Sa{l}[-1]\ar[r]&\discoprod_{i=0}^{l-1}\Aa{i}[-(l-i)]\ar[d]\ar[r]^-{q_{l-1}[-1]}&T_a[-l]\ar[d]\ar[r]^-{c_{l-1}[-1]}&\Sa{l}\\
		&\bullet\ar@{=}[r]&\bullet& }\]
	We have to prove that in the triangle
	\[ \xymatrix{ \Sa{l+1}\ar[r]&\discoprod_{i=0}^l\Aa{i}[-(l-i)]\ar[r]^-{q_l}&T_a[-l]\ar[r]^-{c_l}&\Sa{l+1}[1] } \]
	given by the above octahedral, the middle map is the sink map at $T_a[-l]$ in $\add(T\oplus\cdots\oplus T[-l])$.
	
	We first show that the map is right almost split. Since the map $q_l[-1]$ in second row of the octahedral is the sink map at $T_a[-l]$ in $\add(T[-1]\oplus\cdots\oplus T[-l])$, any radical map $T[-i]\to T_a[-l]$ with $0<i\leq l$ factors through $\coprod_{i=0}^{l-1}\Aa{i}[-(l-i)]$. It remains to consider $\varphi\colon T\to T_a[-l]$.
	\[ \xymatrix@C=8mm@R=0.2mm{
		&T\ar@/^10pt/[drrr]\ar[rddd]_-\varphi\ar@{-->}[dr]\ar@{-->}[dddl]&&&\\
		&&\Aa{l}\ar[dd]\ar[rr]_-{a_l}&&\Sa{l}\ar@{=}[dd]\\ \\
		\discoprod_{i=0}^{l-1}\Aa{i}[-(l-i)]\ar[rr]&&T_a[-l]\ar[rr]&&\Sa{l} } \]
	Since $a_l\colon\Aa{l}\to\Sa{l}$ is a right $(\add T)$-approximation, the morphism $T\xrightarrow{\varphi}T_a[-l]\to\Sa{l}$ can be lifted to $\Aa{l}$. Then the difference of the two maps in the triangle formed by $T$, $\Aa{l}$, and $T_a[-l]$ vanishes under $T_a[-l]\to\Sa{l}$, thus it factors through $\coprod_{i=0}^{l-1}\Aa{i}[-(l-i)]$. We conclude that $\varphi$ factors through $\coprod_{i=0}^{l}\Aa{i}[-(l-i)]$. This proves that $q_l$ is right almost split.
	
	We next show $q_l$ is right minimal. For this we prove that $\Sa{l+1}\xrightarrow{\left( \begin{smallmatrix}r_1\\r_2\end{smallmatrix}\right) }\coprod_{i=0}^l\Aa{i}[-(l-i)]=\Aa{l}\oplus\coprod_{i=0}^{l-1}\Aa{i}[-(l-i)]$ is a radical map. Since $a_l$ is right minimal, the summand $r_1$ is certainly a radical map. It remains to consider $r_2\colon\Sa{l+1}\to\coprod_{i=0}^{l-1}\Aa{i}[-(l-i)]$. If it is not a radical map, then there is a non-zero direct summand $M$ of $\coprod_{i=0}^{l-1}\Aa{i}[-(l-i)]$ which is mapped to $0$ under the vertical map $\coprod_{i=0}^{l-1}\Aa{i}[-(l-i)]\to\bullet$ in the octahedral. Then the composite $M\subset\coprod_{i=0}^{l-1}\Aa{i}[-(l-i)]\xrightarrow{q_{l-1}[-1]}T_a[-l]\to\bullet$ is $0$, so $M\to T_a[-l]$ factors through $\Aa{l}$. Now since $T$ is $(m+1)$-rigid we have $\T(M,\Aa{l})=0$, thus the restriction of $q_{l-1}[-1]\colon\coprod_{i=0}^{l-1}\Aa{i}[-(l-i)]\to T_a[-l]$ to $M$ is $0$. This contradicts right minimality of $q_{l-1}$. Therefore $r_2$ is a radical map.
\end{proof}
Let us note some consequences of this inductive construction of sink maps.
\begin{Lem}\label{tech}
Let $m\geq1$ and $T\in\T$ an $m$-rigid object such that $\End_\T(T\oplus\cdots\oplus T[-(m-1)])$ is hereditary.
\begin{enumerate}
	\item\label{msproj} The map $a_i\colon\Aa{i}\to\Sa{i}$ induces an injection $\T(T,\Aa{i})\hookrightarrow\T(T,\Sa{i})$ for each $0\leq i\leq m-1$, hence an isomorphism for $1\leq i\leq m-1$. %In particular, the $\End_\T(T)$-module $\T(T,\Sa{i-1}[-1])$ is projective when $1\leq i\leq m-1$.
	\item\label{mf0} The maps $f_i\colon\Aa{i}\to\Aa{i-1}$ are $0$ for all $1\leq i\leq m$.
\end{enumerate}
\end{Lem}
\begin{proof}
	Since the assumptions on $T$ for $m$ implies the same for any smaller $m$, it is enough to prove the assertions for the largest possible $i$, that is, $i=m-1$ for (\ref{msproj}) and $i=m$ for (\ref{mf0}). We can apply \ref{induction} for $m-1$, so there is a triangle
	\[ \xymatrix{ \Sa{m}\ar[r]&\discoprod_{i=0}^{m-1}\Aa{i}[-(m-1-i)]\ar[r]&T_a[-(m-1)]\ar[r]&\Sa{m}[1] }, \]
	in which the middle map is the sink map in $\add\widetilde{T}$, where $\widetilde{T}=T\oplus\cdots\oplus T[-(m-1)]$. Applying $\T(\widetilde{T},-)$ we have an exact sequence
	\[ \xymatrix{ \T(\widetilde{T},T_a[-m])\ar[r]&\T(\widetilde{T},\Sa{m})\ar[r]&\T(\widetilde{T},\coprod_{i=0}^{m-1}\Aa{i}[-(m-1-i)])\ar[r]&\T(\widetilde{T},T_a[-(m-1)]) }. \]
	Since $\End_\T(\widetilde{T})$ is hereditary, any sink map in $\add\widetilde{T}$ is a monomorphism, thus the last map in the above exact sequence is injective. Then the middle map, in particular its direct summand
	\[ \xymatrix{ \T(T,\Sa{m})\ar[r]&\T(T,\Aa{m-1}) } \]
	is $0$. We conclude that $\T(T,\Aa{m-1})\to\T(T,\Sa{m-1})$ is injective, and also $f_m\colon\Aa{m}\to\Aa{m-1}$ is $0$ by substituting $T=\Aa{m}$ and considering the image of $a_m\in\T(\Aa{m},\Sa{m})$. 
\end{proof}

Dually we define the objects $\Ba{i}\in\add T$ and $\Ua{i}\in\T$ by the sequence of triangles below.
\[ \xymatrix@R=2mm@!C=5mm{
	&\Ba{0}\ar[dr]\ar[rr]^-{g_1}&&\Ba{1}\ar[dr]\ar[rr]^-{g_2}&&\Ba{2}\ar[dr]\ar[rr]&&\cdots\\
	T_a\ar[ur]^-{b_0}&&\Ua{1}\ar[ur]_-{b_1}&&\Ua{2}\ar[ur]_-{b_2}&&\cdots\ar[ur]&\quad,} \]
where $b_0$ the source map in $\add T$, and $b_i$ the minimal left $(\add T)$-approximation for each $i\geq1$. We state without proof the following dual results.
\begin{Prop}\label{induction2}
Let $m\geq1$ and $T\in\T$ an $(m+1)$-rigid object such that $\End_\T(T\oplus\cdots\oplus T[m-1])$ is hereditary. Then there exists an octahedral
\[ \xymatrix@R=4mm@C=10mm{
	&\bullet\ar[d]\ar@{=}[r]&\bullet\ar[d]&\\
	\Ua{m}\ar@{=}[d]\ar[r]^-{}&T_a[m]\ar[r]^-{p_m}\ar[d]&\discoprod_{i=0}^{m-1}\Ba{i}[m-i]\ar[d]\ar[r]^-{}&\Ua{m}[1]\ar@{=}[d]\\
	\Ua{m}\ar[r]&\Ba{m}\ar[r]^-{}&\Ua{m+1}\ar[r]&\Ua{m}[1] } \]
with $p_{m}$ the source map in $\add(T[1]\oplus\cdots\oplus T[m])$. Moreover, the triangle
\[ \xymatrix{ \Ua{m+1}[-1]\ar[r]&T_a[m]\ar[r]^-{q_m}&\discoprod_{i=0}^m\Ba{i}[m-i]\ar[r]&\Ua{m+1} } \]
obtained from the above homotopy cartesian square gives the source map $q_m$ at $T_a[m]$ in $\add(T\oplus\cdots\oplus T[m])$.
\end{Prop}
\begin{Lem}\label{tech2}
Let $m\geq1$ and $T\in\T$ an $m$-rigid object such that $\End_\T(T\oplus\cdots\oplus T[m-1])$ is hereditary.
\begin{enumerate}
	\item\label{muinj} The map $b_i\colon\Ua{i}\to\Ba{i}$ induces an injection $\T(\Ba{i},T)\hookrightarrow\T(\Ua{i},T)$ for each $0\leq i\leq m-1$, hence an isomorphism for $1\leq i\leq m-1$. %In particular, the $\End_\T(T)$-module $\T(T,\Sa{i-1}[-1])$ is projective when $1\leq i\leq m-1$.
	\item\label{mg0} The maps $g_i\colon\Ba{i-1}\to\Ba{i}$ are $0$ for all $1\leq i\leq m$.
\end{enumerate}
\end{Lem}

We end this section with a technical lemma.
\begin{Lem}\label{np}
Let $T\in\T$ be a rigid object with $\End_\T(T)$ hereditary. Then the $\End_\T(T)$-module $\T(T,\Ua{1})$ has no non-zero projective summands.%\footnote{In fact, $\tau^{-1}$ of the simple at $a$}
\end{Lem}
\begin{proof}
	Applying $\T(T,-)$ to the defining triangle of $\Ua{1}$ gives an exact sequence
	\[ \xymatrix{ \T(T,T_a)\ar[r]&\T(T,\Ba{0})\ar[r]& \T(T,\Ua{1})\ar[r]&\T(T,T_a[1])=0}. \]
	Since $T_a\to\Ba{0}$ is a source map, the first map in the above sequence is left minimal, hence there cannot be a projective summand in $\T(T,U_a[1])$.
\end{proof}

\subsection{AR sequences}\label{ARseq}
A crucial ingredient of the proof of our main results for this section is AR sequences in a cluster tilting subcategory in a triangulated category. Let us first recall its notion and the fundamental existence theorem.
\begin{Thm}[\cite{IYo}]
	Let $\T$ be a $k$-linear, $\Hom$-finite, idempotent-complete triangulated category with a Serre functor $\nu$, and let $\C$ be a $d$-cluster tilting subcategory. Then for any $C\in\C$ there exists a sequence, unique up to isomorphism of complexes,
	\[ \xymatrix@!C=1mm@!R=2mm{
		&C_1\ar[dr]|-{g_1}\ar[rr]&&C_2\ar[r]&\cdots\ar[r]&C_{i-1}\ar[rr]\ar[dr]|-{g_{i-1}}&&C_i\ar[dr]|-{g_i}\ar[rr]&&C_{i+1}\ar[r]&\cdots\ar[r]&C_{d-1}\ar[dr]|-{g_{d-1}}\ar[rr]&&C_d\ar[dr]|-{g_d}&\\
		C_0\ar[ur]|-{f_0}&&Y_1\ar[ur]|-{f_1}&&&&Y_{i-1}\ar[ur]|-{f_{i-1}}&&Y_i\ar[ur]|-{f_i}&&&&Y_{d-1}\ar[ur]|-{f_{d-1}}&&C_{d+1} } \]
	with $C_0=C$ (resp. $C_{d+1}=C$) and all $C_i\in\C$ such that
	\begin{itemize}
		\item each $Y_{i-1}\xrightarrow{f_{i-1}} C_i\xrightarrow{g_i} Y_i$, $1\leq i\leq d$ is a part of a triangle $Y_{i-1}\to C_i\to Y_i\to Y_{i-1}[1]$ in $\T$, where we understand $Y_0=C_0$ and $Y_d=C_{d+1}$,
		\item $f_0\colon C_0\to C_1$ is a source map and $g_d\colon C_{d}\to C_{d+1}$ is a sink map in $\C$,
		\item each $C_{i}\to C_{i+1}$ is a radical map.
	\end{itemize}
	We call the above sequence an {\rm AR $(d+2)$-angle} in $\C$. Moreover, it satisfies the following.
	\begin{enumerate}
		\item For each $1\leq i\leq d-1$, the morphism $f_i$ is a minimal left $\C$-approximation and $g_i$ is a minimal right $\C$-approximation.
		\item $C_{d+1}=\nu_d^{-1} C_{0}$ for $\nu_d=\nu\circ[-d]$.
	\end{enumerate}
\end{Thm}
In particular if $\T$ is $d$-CY the AR $(d+2)$-angles have the same end terms, say $C$. In this case we call it the {\it AR $(d+2)$-angle at $C$}.

Now assume that $\T$ is $d$-CY and $T\in\T$ is $d$-cluster tilting such that $\End_\T(T)$ is hereditary. We denote by $T_a$ the indecomposable direct summand of $T$ corresponding to the vertex $a$ of the quiver $Q$ of $\End_\T(T)$. For each vertex $a$, define the objects $S_a, U_a\in\T$ by the triangles
\[ \xymatrix@R=1mm{
	\discoprod_{b\to a}T_b\ar[r]&T_a\ar[r]& S_a\ar[r]&\discoprod_{b\to a}T_b \\
	U_a\ar[r]&T_a\ar[r]&\discoprod_{a\to b}T_b\ar[r]&U_a[1], } \]
where the sum $\coprod_{b\to a}$ (resp. $\coprod_{a\to b}$) runs over all the arrows ending (resp. starting) at $a$, giving a sink map (resp. source map) at $T_a$ in $\add T$. Note that $S_a[-1]=\Sa{1}$ and $U_a[1]=\Ua{1}$ in Section \ref{gen}.

We note an easy observation on these sink and source maps.
\begin{Lem}\label{rex}
Suppose $\T(T,T[-i])=0$ for $0<i<n$. Then for each $0<i\leq n$, there are exact sequences
\[ \xymatrix@R=1mm{
	\T(T,\coprod_{b\to a}T_b[-i])\ar[r]&\T(T,T_a[-i])\ar[r]&\T(T,S_a[-i])\ar[r]& 0 \\
	\T(\coprod_{a\to b}T_b[i],T)\ar[r]&\T(T_a[i],T)\ar[r]&\T(U_a[i],T)\ar[r]& 0. } \]
In particular, $\T(T,S_a[-i])=0$ and $\T(U_a[i],T)=0$ for each vertex $a$ and $0<i<n$.
\end{Lem}
\begin{proof}
	We only prove the statement for $S_a$. By the defining triangle for $S_a$ and vanishing of negative self-extensions of $T$, we immediately have the exact sequence for $1<i\leq n$. When $i=1$, consider the exact sequence
	\[ \xymatrix{ \T(T,\coprod_{b\to a}T_b[-1])\ar[r]&\T(T,T_a[-1])\ar[r]&\T(T,S_a[-1])\ar[r]&\T(T,\coprod_{b\to a}T_b)\ar[r]&\T(T,T_a) }, \]
	in which the last map is injective since it is the sink map in $\add T$ and $\End_\T(T)$ is hereditary. This proves our assertion.
\end{proof}
Let us also mention that vanishing of negative extensions up to a half of $d$ automatically yields the vanishing for the other half, which allows us to weaken the vanishing assumption in Keller--Reiten's theorem \ref{known}(1).
\begin{Prop}\label{half}
Suppose $\T(T,T[-i])=0$ for $0<i\leq (d-1)/2$. Then $\T(T,T[-i])=0$ for $0<i<d-1$.
\end{Prop}
\begin{proof}
	Consider the AR $(d+2)$-angle at each indecomposable direct summand $T_a$ of $T$.
	\[ \xymatrix@!C=1mm@!R=2mm{
		&\discoprod_{a\to b}T_b\ar[dr]\ar[rr]&&T_a^{(d-2)}\ar[dr]\ar[rr]&&\cdots\ar[dr]\ar[rr]&&T_a^{(1)}\ar[dr]\ar[rr]&&\discoprod_{b\to a}T_b\ar[dr]&\\
		T_a\ar[ur]&&U_a[1]\ar[ur]&&\ar[ur]&&\ar[ur]&&S_a[-1]\ar[ur]&&T_a } \]
	By the assertion for $S_a$ in \ref{rex}, the middle terms $T_a^{(i)}$ are $0$ for $1\leq i\leq (d-1)/2$. Also from that for $U_a$, we have $T_a^{(d-1-i)}=0$ for $1\leq i\leq(d-1)/2$. Therefore all the middle terms $T_a^{(1)},\ldots,T_a^{(d-2)}$ are $0$, hence $\T(T,S_a[-i])=0$ for $0<i<d-1$.
	Now we prove $\T(T,T_a[-i])=0$ for $0<i<d-1$ by induction on the vertices of $Q$, precisely, induction on the maximal length of path ending at the vertex $a$. If $a$ is a source, then $T_a=S_a$ and we are done. Applying $\T(T,-)$ to the rightmost triangle, we have an exact sequence
	\[ \xymatrix{ \T(T,\coprod_{b\to a}T_b[-i])\ar[r]& \T(T,T_a[-i])\ar[r]& \T(T,S_a[-i])}, \]
	in which the right term is $0$ by the former claim, and the left term is $0$ by the induction hypothesis. We therefore conclude that $\T(T,T_a[-i])=0$.
\end{proof}
\subsection{Odd CY triangulated category with a cluster titling object}
We apply our general observations of Section \ref{gen} to the setting of a CY triangulated category and a cluster tilting object. Let $n\geq1$ and let $\T$ be a $(2n+1)$-CY triangulated category with a $(2n+1)$-cluster tilting object $T$ such that $\End_\T(T)$ is hereditary, and $\Hom_\T(T,T[-i])=0$ for $0<i<n$. 
Our proof shows that we can detect the structure of $H=\End_\T(T\oplus\cdots\oplus T[-(2n-1)])$ from the AR $(2n+3)$-angles in $\add T$. The following observation is therefore fundamental.
\begin{Prop}\label{AR2n+3}
The AR $(2n+3)$-angle at $T_a$ is of the following form for some $A_a\in\add T$.
\[	\xymatrix@!C=1mm@!R=2mm{
	&\discoprod_{a\to b}T_b\ar[dr]\ar[rr]&&0\ar[r]&\cdots\ar[r]&0\ar[rr]\ar[dr]&&A_a\ar[dr]\ar[rr]&&0\ar[r]&\cdots\ar[r]&0\ar[dr]\ar[rr]&&\discoprod_{b\to a}T_b\ar[dr]&\\
	T_a\ar[ur]&&U_a[1]\ar[ur]&&&&U_a[n]\ar[ur]&&S_a[-n]\ar[ur]&&&&S_a[-1]\ar[ur]&&T_a\,, }\]
with all the omitted middle terms $0$.
\end{Prop}
\begin{proof}
	We know that the sink map in $\add T$ at $T_a$ is $\coprod_{b\to a}T_b\to T_a$, so we have the rightmost triangle. Then we have $\T(T,S_a[-i])=0$ for $0<i<n$ by \ref{rex}, which gives the triangles on the right half. We dually get the triangles on the left half.
\end{proof}
We will refer to these AR $(2n+3)$-angles as (AR).

The symmetry of AR sequences has the following consequences. Compare the first one with \ref{tech}(\ref{msproj}).
\begin{Lem}\label{sproj}
The map $A_a\to S_a[-n]$ in the middle triangle in (AR) induces an isomorphism $\T(T,A_a)\simeq\T(T,S_a[-n])$. In particular, the $\End_\T(T)$-module $\T(T,S_a[-n])$ is projective.
\end{Lem}
\begin{proof}
	The map is always surjective since $A_a\to S_a[-n]$ is a right $(\add T)$-approximation. If $n>1$, we see $\T(T,U_a[n])=0$ by the leftmost triangle in (AR), which gives injectivity. Now assume $n=1$. In this case (AR) has the form 
	\[ \xymatrix@!C=3mm@!R=2mm{
		&\discoprod_{a\to b}T_b\ar[dr]\ar[rr]^-g&&A_a\ar[dr]\ar[rr]&&\discoprod_{b\to a}T_a\ar[dr]&\\
		T_a\ar[ur]&&U_a[1]\ar[ur]&&S_a[-1]\ar[ur]&&T_a, } \]
	in which we have $g=0$ by \ref{tech2}(\ref{mg0}) for $m=1$. Then applying $\T(T,-)$ to $g=0$ yields a $0$-map $\T(T,\coprod_{a\to b}T_b)\to\T(T,U_a[1])\to\T(T,A_a)$, with the first map being surjective. Therefore, the second map is $0$, hence applying $\T(T,-)$ to the middle triangle yields an isomorphism $\T(T,A_a)\xrightarrow{\simeq}\T(T,S_a[-1])$.
\end{proof}

\begin{Lem}\label{nex}
We have $\T(T,T[-j])=0$ for $n<j<2n$.
\end{Lem}
\begin{proof}
We first show $\T(T,S_a[-j])=0$ for all vertices $a$ and $n<j<2n$. Put $j=n+i$ so that $0<i<n$. Applying $\T(T,-)$ to the middle triangle in (AR) we have an exact sequence
\[ \xymatrix{ \T(T,A_a[-i])\ar[r]& \T(T,S_a[-n-i])\ar[r]& \T(T,U_a[n-i+1]) }, \]
but $\T(T,A_a[-i])=0$ since $A_a\in\add T$, and the leftmost triangle yields $\T(T,U_a[n-i+1])=0$, hence $\T(T,S_a[-n-i])=0$.

Applying $\T(T,-)$ to the rightmost triangle, we have an exact sequence
\[ \xymatrix{ \T(T,\coprod_{b\to a}T_b[-j])\ar[r]& \T(T,T_a[-j])\ar[r]& \T(T,S_a[-j])}. \]
We therefore see by induction on the vertices of the quiver of $\End_\T(T)$ that $\T(T,T_a[-j])=0$ (cf. proof of \ref{half}).
\end{proof}

Now we are ready to prove the first main result of this section.
\begin{Thm}\label{oddcy}
Let $n\geq1$ and $\T$ be a $(2n+1)$-CY triangulated category with a $(2n+1)$-cluster tilting object $T$. Suppose that $\End_\T(T)$ is hereditary and $\Hom_\T(T,T[-i])=0$ for $0<i<n$.
\begin{enumerate}
	\item The algebra $H^\prime=\End_\T(T\oplus T[-n])$ is hereditary.
	\item The algebra $H=\End_\T(T\oplus\cdots\oplus T[-(2n-1)])$ is a direct product of $n$ copies of $H^\prime$, thus is hereditary.
\end{enumerate}
%In particular, if $T\in\T$ is a $3$-cluster tilting object in a $3$-CY triangulated category such that $\End_\T(T)=kQ$ for some acyclic quiver $Q$, then $\End_\T(T\oplus T[-1])$ is also hereditary.
\end{Thm}
\begin{proof}
	(2) follows from (1) and \ref{nex}, so we prove (1). For this we show that any sink map in $\proj H^\prime=\add(T\oplus T[-n])$ is a monomorphism. This is clear for the sink maps at objects in $\add T$ since there are no non-zero morphisms from $\add T[-n]$ to $\add T$, so sink maps in $\add T$ give sink maps in $\add(T\oplus T[-n])$. We consider the sink map at $T_a[-n]$. Applying \ref{induction} for $m=n$ there exists a commutative diagram of triangles
	\begin{equation}\label{eqsink}
	\xymatrix@R=6mm{
		S_a[-n-1]\ar@{=}[d]\ar[r]&U_a[n]\ar[r]\ar[d]&A_a\ar[r]\ar[d]&S_a[-n]\ar@{=}[d]\\
		S_a[-n-1]\ar[r]&\discoprod_{b\to a}T_b[-n]\ar[r]&T_a[-n]\ar[r]&S_a[-n] }
	\end{equation}
	giving a triangle
	\[ \xymatrix{U_a[n]\ar[r]&A_a\oplus \discoprod_{b\to a}T_b[-n]\ar[r]& T_a[-n]\ar[r]& U_a[n+1] }, \]
	in which the middle map is the sink map in $\add(T\oplus\cdots\oplus T[-n])$. Since it has terms in $\add(T\oplus T[-n])$ it is the one in $\add(T\oplus T[-n])$. Applying $\T(T\oplus T[-n],-)$ gives an exact sequence
	\[ \xymatrix{\T(T\oplus T[-n],U_a[n])\ar[r]&\T(T\oplus T[-n],A_a\oplus\coprod_{b\to a}T_b[-n])\ar[r]& \T(T\oplus T[-n],T_a[-n]) }, \]
	in which we want to show that the first map is $0$. It is clear for the summand $\T(T[-n],-)$ since the above sequence reduces to
	\[ \xymatrix{\T(T[-n],U_a[n])\ar[r]&\T(T[-n],\coprod_{b\to a}T_b[-n])\ar[r]& \T(T[-n],T_a[-n]) }, \]
	and the second map, being the sink map in $\add T[-n]$, is a monomorphism. Also we see that $\T(T,U_a[n])\to\T(T,A_a)$ is $0$ by the sequence
	\[ \xymatrix{\T(T,U_a[n])\ar[r]&\T(T,A_a)\ar[r]& \T(T,S_a[-n]) }, \]
	in which the second map is injective (in fact an isomorphism) by \ref{sproj}. It remains to show that the map
	\[ \xymatrix{\T(T,U_a[n])\ar[r]& \T(T,\coprod_{b\to a}T_b[-n]) } \]
	is $0$.	If $n>1$ then $\T(T,U_a[n])=0$ and we have the assertion.
	Now assume $n=1$. We prove the following statements by induction on the vertices of $Q$, which will complete the proof by (i).
	\begin{enumerate}
		\renewcommand{\labelenumi}{(\roman{enumi})}
		\item The map $\T(T,U_a[1])\to\T(T,\coprod_{b\to a}T_b[-1])$ is $0$.
		\item The $\End_\T(T)$-module $\T(T,T_a[-1])$ is projective.
	\end{enumerate}
	If $a$ is a source, then we have $\coprod_{b\to a}T_b=0$ thus (i), and $T_a=S_a$ thus (ii) by \ref{sproj}. Suppose now that $a$ is a general vertex of $Q$. By \ref{np}, the $\End_\T(T)$-module $\T(T,U_a[1])$ has no projective summands while $\T(T,\coprod_{b\to a}T_b[-1])$ is projective by induction hypothesis. Therefore the map in (i) has to be $0$. Then by the leftmost commutative square in (\ref{eqsink}), we have that $\T(T,S_a[-2])\to\T(T,\coprod_{b\to a}T_b[-1])$ is $0$, hence applying $\T(T,-)$ to the second row gives a short exact sequence below by \ref{rex}.
	\[ \xymatrix{0\ar[r]& \T(T,\coprod_{b\to a}T_b[-1])\ar[r]& \T(T,T_a[-1])\ar[r]& \T(T,S_a[-1])\ar[r]& 0 } \]
	This has projective end terms by induction hypothesis and \ref{sproj} respectively, so we conclude that the middle term $\T(T,T_a[-1])$ is also projective.
\end{proof}
We conclude our discussion by explicitly describing the quiver $\widetilde{Q}$ of $\End_\T(T\oplus T[-n])$. Suppose that $\End_\T(T)=kQ$ for an acyclic quiver $Q$. Note first that $T$ and $T[-n]$ have no common direct summand by $\T(T[-n],T)=0$, thus $\widetilde{Q}$ have two copies of $Q$ as a subquiver. We need to investigate the arrows from the subquiver for $T$ to the one for $T[-n]$. We identify the vertices of $\widetilde{Q}$ and the corresponding indecomposable summands of $T\oplus T[-n]$.
\begin{Prop}\label{quiver}
The following are equal.
\begin{enumerate}
	\renewcommand{\labelenumi}{(\alph{enumi})}
	\renewcommand{\theenumi}{(\alph{enumi})}
	\item\label{1} The number of arrows from $T_b$ to $T_a[-n]$.
	\item\label{2} The number of direct summands $T_b$ in $A_a$.
	\item\label{3} The number of arrows from $T_a$ to $T_b[-n]$.
	\item\label{4} The number of direct summands $T_a$ in $A_b$.
\end{enumerate}
\end{Prop}
\begin{proof}
	By \ref{induction} the map $A_a\oplus\coprod_{b\to a}T_b[-n] \to T_a[-n]$ is the sink map in $\add(T\oplus T[-n])$. This gives \ref{1}=\ref{2}, and similarly \ref{3}=\ref{4}. Dually, \ref{induction2} shows that the map $T_a[n]\to A_a\oplus\coprod_{a\to b}T_b[n]$ is the source map in $\add(T\oplus T[n])$, which shows \ref{3}=\ref{2}.
\end{proof}

\begin{Ex}
Let $Q$ be the quiver of linearly oriented type $A_3$ and $\T=\C_3(kQ)$ the $3$-cluster category of $Q$. We have its AR quiver as below, with a fundamental domain inside the dotted line. It has a $3$-cluster tilting object $T=T_1\oplus T_2\oplus T_3$ which is obtained by mutating the initial cluster tilting object $kQ\in\C_3(kQ)$ at the sink.
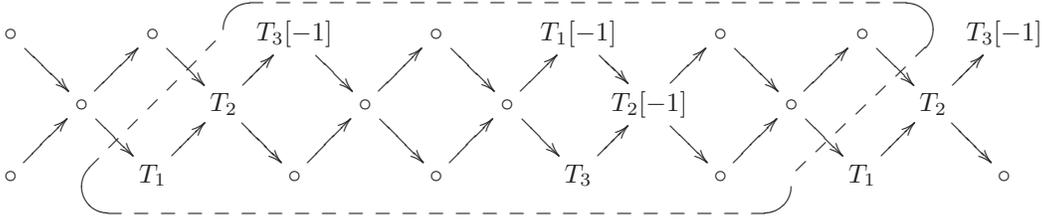
\[ \xymatrix@!R=1mm@!C=1mm{
	\circ\ar[dr]&&\circ\ar[dr]&\ar@{--}`u[r]`[rrrrrrrrrr]`_dl[ddrrrrrrrr][ddrrrrrrrr]&T_3[-1]\ar[dr]&&\circ\ar[dr]&&T_1[-1]\ar[dr]&&\circ\ar[dr]&&\circ\ar[dr]&&T_3[-1]\\
	&\circ\ar[ur]\ar[dr]&&T_2\ar[ur]\ar[dr]&&\circ\ar[ur]\ar[dr]&&\circ\ar[ur]\ar[dr]&&T_2[-1]\ar[ur]\ar[dr]&&\circ\ar[ur]\ar[dr]&&T_2\ar[ur]\ar[dr]&\\
	\circ\ar[ur]&&T_1\ar[ur]&&\circ\ar[ur]&&\circ\ar[ur]&&T_3\ar[ur]&&\circ\ar[ur]&\ar@{--}`d[l]`[llllllllll]`_ur[uullllllll][uullllllll]&T_1\ar[ur]&&\circ } \]
We see that $\End_\T(T)=kA_2\times k$, and $\End_\T(T\oplus T[-1])$ is presented by the quiver below; a direct product of $2$ path algebras of type $A_3$ with different orientations.
\[ \xymatrix@!C=8mm{
	T_1\ar[r]& T_2\ar[dr]& T_3\ar[dl]\\
	T_1[-1]\ar[r]& T_2[-1]& T_3[-1] } \]
\end{Ex}
We end this subsection with the following non-example, which shows that \ref{oddcy} fails for $4$-CY case, that is, even if $T\in\T$ is a $4$-cluster tilting such that $\End_\T(T)$ is hereditary, the shifted sum $T\oplus T[-1]\oplus T[-2]$ does not necessarily have hereditary endomorphism ring.
\begin{Ex}\label{nonex}
Let $Q$ be a quiver of type $A_3$, and $\T=\C_4(kQ)$ the $4$-cluster category of $Q$. Consider the $4$-cluster tilting object $T=T_1\oplus T_2\oplus T_3$ which is obtained by mutating $T_1\oplus T_2\oplus T_3[-1]$ at $T_3[-1]$.
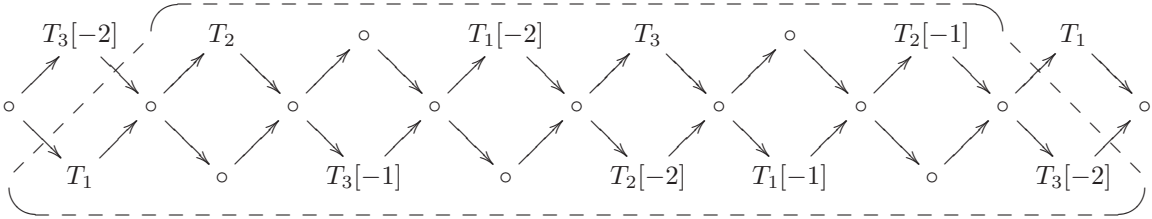
\[ \xymatrix@!R=1mm@!C=1mm{
	&T_3[-2]\ar[dr]&\ar@{--}`u[r]`[rrrrrrrrrrrr][rrrrrrrrrrrr]&T_2\ar[dr]&&\circ\ar[dr]&&T_1[-2]\ar[dr]&&T_3\ar[dr]&&\circ\ar[dr]&&T_2[-1]\ar[dr]&\ar@{--}[ddrr]&T_1\ar[dr]&\\
	\circ\ar[ur]\ar[dr]&&\circ\ar[ur]\ar[dr]&&\circ\ar[ur]\ar[dr]&&\circ\ar[ur]\ar[dr]&&\circ\ar[ur]\ar[dr]&&\circ\ar[ur]\ar[dr]&&\circ\ar[ur]\ar[dr]&&\circ\ar[ur]\ar[dr]&&\circ\\
	\ar@{--}`d[r]`[rrrrrrrrrrrrrrrr][rrrrrrrrrrrrrrrr]\ar@{--}[uurr]&T_1\ar[ur]&&\circ\ar[ur]&&T_3[-1]\ar[ur]&&\circ\ar[ur]&&T_2[-2]\ar[ur]&&T_1[-1]\ar[ur]&&\circ\ar[ur]&&T_3[-2]\ar[ur]& } \]
The algebra $H=\End_\T(T\oplus T[-1]\oplus T[-2])$ is presented by the quiver
\[ \xymatrix@!C=8mm@R=7mm{
	T_1\ar[r]&T_2\ar[dr]&T_3\ar[dll]\\
	T_1[-1]\ar[r]&T_2[-1]\ar[dr]&T_3[-1]\ar[dll]\\
	T_1[-2]\ar[r]&T_2[-2]&T_3[-2] } \]
with relations ``any composite of arrows equals $0$''. It follows that $H$ has global dimension $4$.
\end{Ex}
\subsection{Even CY triangulated category with a cluster tilting object}\label{4cy4ct}
As we have just seen in \ref{nonex} above, hereditaryness of $T$ does not imply that of $T\oplus T[-1]\oplus\cdots\oplus T[-(d-2)]$ when $d\geq4$. Nevertheless we prove in this section that in dimension $4$, the endomorphism algebra of $T\oplus T[-1]\oplus T[-2]$ is hereditary as soon as that of $T\oplus T[-1]$ is (\ref{evency}). As before, we work in a higher dimensional setting with vanishing of some negative extensions.

Let $n\geq1$, and let $\T$ be a $(2n+2)$-CY triangulated category with a $(2n+2)$-cluster tilting object $T$ such that $\T(T,T[-i])=0$ for $0<i<n$, and $\End_\T(T\oplus T[-n])$ is hereditary. In this case $\End_\T(T)$ is also hereditary, whose quiver we denote by $Q$. We denote by $T_a$ the indecomposable direct summand of $T$ corresponding to the vertex $a$ of $Q$. As in the previous subsection, our starting point is a computation of AR sequences. Recall the definition of $S_a$ and $U_a$ from Section \ref{ARseq}.
\begin{Prop}\label{AR2n+4}
The AR $(2n+4)$-angle at $T_a$ is of the form below for some $A_a, B_a\in\add T$ and $Y_a\in\T$.
\[	\xymatrix@!C=1.4mm@R=3mm{
	&\discoprod_{a\to b}T_b\ar[dr]\ar[rr]&&0\ar[r]&\cdots\ar[r]&0\ar[dr]\ar[rr]&&B_a\ar[rr]^-f\ar[dr]&&A_a\ar[dr]\ar[rr]&&0\ar[r]&\cdots\ar[r]&0\ar[dr]\ar[rr]&&\discoprod_{b\to a}T_b\ar[dr]&\\
	T_a\ar[ur]&&U_a[1]\ar[ur]&&&&U_a[n]\ar[ur]\ar@{--}|{{\rm (b)}}[rr]&&Y_a\ar[ur]\ar@{--}|{{\rm (a)}}[rr]&&S_a[-n]\ar[ur]&&&&S_a[-1]\ar[ur]&&T_a}	\]
Moreover the middle map $f$ is $0$.
\end{Prop}
%\[ \xymatrix@R=3mm@C=5mm{
%	&0\ar[dr]\ar[rr]&&T^{m}\ar[dr]\ar[rr]^-f&&T^{m}\ar[dr]\ar[rr]&&0\ar[dr]& \\
%	T\ar[ur]&&T[1]\ar[ur]\ar@{--}|{(1)}[rr]&&Y\ar[ur]\ar@{--}|{(2)}[rr]&&T[-1]\ar[ur]&&T } \]
\begin{proof}
	We draw the AR sequence from the right. Clearly we have the rightmost triangle. By \ref{rex} we get the triangles on the right half up to $S_a[-n]$. Dually we can draw the left triangles to obtain the AR $(2n+2)$-angle above. Now, since $\End_\T(T\oplus T[-n])$ is hereditary we can apply \ref{tech}(\ref{mf0}) for $m=n+1$, which shows $f=0$.
\end{proof}
As before we will refer to these AR $(2n+4)$-angles as (AR). Similarly to the previous subsection, the symmetry of (AR) gives the following.
\begin{Lem}\label{nex2}
We have $\T(T,T_a[-j])=0$ for $n+1<j<2n$.
\end{Lem}
\begin{proof}
	Put $j=n+i$ so that $1<i<n$. By the triangle (a) in (AR) we have an exact sequence
	\[ \xymatrix{ 0=\T(T,A_a[-i])\ar[r]& \T(T,S_a[-j])\ar[r]&\T(T,Y_a[-i+1])\ar[r]& \T(T,A_a[-i+1])=0 }, \]
	in which the two end terms are $0$ by vanishing of small negative self-extensions of $T$.
	Also by the triangle (b) in (AR) we have
	\[ \xymatrix{ 0=\T(T,B_a[-i+1])\ar[r]&\T(T,Y_a[-i+1])\ar[r]&\T(T,U_a[n+2-i])=0 }, \]
	where the leftmost term is $0$ by vanishing of small negative self-extensions of $T$, and so is the rightmost term by the leftmost triangle in (AR). We deduce by these exact sequences that $\T(T,S_a[-j])=\T(T,Y_a[-i+1])=0$.
	
	Now by the exact sequence
	\[ \xymatrix{\T(T,\coprod_{b\to a}T_b[-j])\ar[r]&\T(T,T_a[-j])\ar[r]&\T(T,S_a[-j]) }, \]
	we see inductively that $\T(T,T_a[-j])=0$.
\end{proof}

We need some more technical observations.
\begin{Lem}\label{YB}
We have $\T(T,Y_a[-n+1])=\begin{cases}\T(T,B_a) & (n=1)\\ 0& (n>1) \end{cases}$.
\end{Lem}
\begin{proof}
This is easily shown by applying $\T(T,-)$ to the triangle (b) in (AR).
\end{proof}

\begin{Lem}\label{ASY}
The triangle (a) in (AR) yields a short exact sequence
\[ \xymatrix{0\ar[r]& \T(T,A_a[-n])\ar[r]& \T(T,S_a[-2n])\ar[r]& \T(T,Y_a[-n+1])\ar[r]& 0}. \]
Consequently the $\End_\T(T)$-module $\T(T,S_a[-2n])$ is projective.
\end{Lem}
\begin{proof}
	We apply $\T(T,-)$ to the triangle (a) in (AR). For any $n\geq1$ it yields an exact sequence
	\[ \xymatrix@R=2.5mm{
		\T(T,Y_a[-n])\ar[r]^-g& \T(T,A_a[-n])\ar[r]& \T(T,S_a[-2n])\ar[r]& \T(T,Y_a[-n+1]) \ar@{-}`r/2.5mm[d]`_l[dl][dl]\\
		&\ar@{-}[rr]&&\\
		&\T(T,A_a[-n+1])\ar[r]^-{g^\prime}\ar@{<-}`l[u]`_r[ur][ur]&\T(T,S_a[-2n+1])& } \]
	
	We first show right exactness. By \ref{YB} we only have to consider $n=1$. Then $g^\prime$ is an isomorphism by \ref{sproj}, which gives the result.
	
	We next show that the map $g$ is $0$, which yields left exactness. Applying $\T(T,-)$ to the middle map $0=f\colon B_a\to Y_a\to A_a$ in (AR), we have a zero map $\T(T,B_a[-n])\xrightarrow{h}\T(T,Y_a[-n])\xrightarrow{g}\T(T,A_a[-n])$. Now, applying $\T(T,-)$ to the triangle (b) in (AR) gives an exact sequence
	\[ \xymatrix{ \T(T,B_a[-n])\ar[r]^-h&\T(T,Y_a[-n])\ar[r]&\T(T,U_a[1])\ar[r]^-k&\T(T,B_a[-n+1])\ar[r]^-{k^\prime}&\T(T,Y_a[-n+1]) }, \]
	in which we have $k=0$. Indeed, if $n>1$ then $\T(T,B_a[-n+1])=0$, and if $n=1$ then $k^\prime$ is an isomorphism by \ref{YB}. It follows that $\T(T,U_a[1])=\Coker h$, so $g$ factors through $\T(T,U_a[1])$. But as $\T(T,U_a[1])$ has no non-zero projective summand by \ref{np} while $\T(T,A_a[-n])$ is projective since $\End_\T(T\oplus T[-n])$ is hereditary, this shows $g=0$ as desired.
	
	Finally, the $\End_\T(T)$-module $\T(T,Y_a[-n+1])$ is projective by \ref{YB} and so is $\T(T,A_a[-n])$ since $\End_\T(T\oplus T[-n])$ is hereditary, thus the remaining assertion follows.
\end{proof}
As a final preparation we need an observation on the behavior of sink maps.
\begin{Lem}\label{shift}
Let $n+1\leq j\leq 2n$. Then the sink map at $T_a[-j]$ in $\add(T\oplus\cdots\oplus T[-2n])$ is obtained from the one at $T[-n-1]$ by shifting by $[-(j-n-1)]$.
\end{Lem}
\begin{proof}
	Applying \ref{induction2} for $m=n+1$, the middle map in the triangle
	\[ \xymatrix{U_a[n]\ar[r]& B_a\oplus A_a[-1]\oplus\discoprod_{b\to a}T_b[-n-1]\ar[r]& T_a[-n-1]\ar[r]& U_a[n+1] } \]
	gives the sink map at $T_a[-n-1]$ in $\add(T\oplus\cdots\oplus T[-n-1])$. In view of rigidity of $T$, this is also the sink map in $\add(T\oplus\cdots\oplus T[-2n])$. We have to show that in the shifted triangle
	\[ \xymatrix{U_a[2n+1-j]\ar[r]& B_a[-(j-n-1)]\oplus A_a[-(j-n)]\oplus\discoprod_{b\to a}T_b[-j]\ar[r]^-v& T_a[-j]\ar[r]& U_a[2n+2-j] }, \]
	the middle map is the sink map in $\add(T\oplus\cdots\oplus T[-2n])$. Since $T$ is rigid and $\T(T,T[-i])=0$ for $i\in\{0,\ldots,2n\}\setminus\{0,n,n+1,2n\}$ by \ref{nex2}, this assertion is clear for $j<2n$. It remains to verify $j=2n$. It is enough to show $\T(T,v)$ is surjective, but this is clear by $\T(T,U_a[2])=0$.
\end{proof}
Now we are ready to prove the $4$-CY version, or more generally even CY version, of our main result of this section.
\begin{Thm}\label{evency}
Let $n\geq1$ and let $\T$ be a $(2n+2)$-CY triangulated category with a $(2n+2)$-cluster tilting object $T$. Suppose that $\T(T,T[-i])=0$ for $0<i<n$, and $\End_\T(T\oplus T[-n])$ is hereditary. Then $H=\End_\T(T\oplus\cdots\oplus T[-2n])$ is hereditary.
\end{Thm}
\begin{proof}
	We show that for each $a\in Q_0$ and $0\leq j\leq 2n$ the sink map at $T_a[-j]$ in $\add(T\oplus\cdots\oplus T[-2n])$ is a monomorphism.
	
	By our assumptions this is clear for $0\leq j\leq n$. 
	
	Let $n+1\leq j<2n$, in particular $n>1$. By \ref{induction} and \ref{shift}, the map $v$ in the triangle
	\[ \xymatrix{U_a[n]\ar[r]^-u& B_a\oplus A_a[-1]\oplus\discoprod_{b\to a}T_b[-n-1]\ar[r]^-v& T_a[-n-1]\ar[r]& U_a[n+1] } \]
	%\[ \xymatrix{U_a[2n+1-j]\ar[r]^-u& B_a[-(j-n-1)]\oplus A_a[-(j-n)]\oplus\discoprod_{b\to a}T_b[-j]\ar[r]^-v& T_a[-j]\ar[r]& U_a[2n+2-j] } \]
	gives the sink map at $T_a[-n-1]$ in $\add(T\oplus\cdots\oplus T[-2n])$, and the same holds for the sequence shifted by $[-(j-n-1)]$. We want to show $\T(T\oplus\cdots\oplus T[-2n],v[-(j-n-1)])$ is injective, or equivalently $\T(T\oplus\cdots\oplus T[-2n],u[-(j-n-1)])=0$. By vanishing of some negative and positive self-extension of $T$, it is enough to consider $j=n+1$, thus is reduced to showing $\T(T\oplus T[-1]\oplus T[-n-1],u)=0$. This follows from $\T(T\oplus T[-1]\oplus T[-n-1],U_a[n])=0$, where we used $n>1$ for $\T(T,U_a[n])=0$.
	%We want to show $\T(T\oplus\cdots\oplus T[-2n],v)$ is injective, or equivalently $\T(T\oplus\cdots\oplus T[-2n],u)=0$. By vanishing of self-extension of $T$, it easily reduces to showing $\T(T[-j+n+1]\oplus T[-j+n]\oplus T[-j],u[-j+n+1])=0$, or $\T(T\oplus T[-1]\oplus T[-n-1],u)=0$.
	
	Finally consider $j=2n$. By \ref{induction} and \ref{shift} we have the commutative diagram
	\begin{equation}\label{eq2n}
	\xymatrix@R=6mm{
		Y_a[-n]\ar[r]\ar@{=}[d]&U_a[1]\ar[r]\ar[d]& B_a[-n+1]\ar[r]\ar[d]& Y_a[-n+1]\ar@{=}[d]\\
		Y_a[-n]\ar[r]& A_a[-n]\oplus\discoprod_{b\to a}T_b[-2n]\ar[r]& T_a[-2n]\ar[r]& Y_a[-n+1]}
	\end{equation}
	giving the sink map $v$ at $T_a[-2n]$
	\[ \xymatrix{U_a[1]\ar[r]^-u& B_a[-n+1]\oplus A_a[-n]\oplus\discoprod_{b\to a}T_b[-2n]\ar[r]^-v& T_a[-2n]\ar[r]& U_a[2] }, \]
	in which we want to show $\T(T\oplus\cdots\oplus T[-2n],v)$ is injective. Since $\T(T[-1]\oplus\cdots\oplus T[-2n],U_a[1])=0$ by the leftmost triangle in (AR), it remains to prove that $\T(T,u)=0$. We prove by induction on the vertices of $Q$ the following statements, which will complete the proof.
	\begin{enumerate}
	\renewcommand{\labelenumi}{(\roman{enumi})}
	\item The map $\T(T,U_a[1])\to\T(T,B_a[-n+1]\oplus A_a[-n]\oplus\coprod_{b\to a}T_b[-2n])$ is $0$.
	\item The $\End_\T(T)$-module $\T(T,T_a[-2n])$ is projective.
	\end{enumerate}
	If $a$ is a source, then $\T(T,B_a[-n+1]\oplus A_a[-n])$ is projective while $\T(T,U_a[1])$ has no non-zero projective summand by \ref{np}, so we have (i). Also in this case we have $T_a=S_a$, so (ii) by \ref{ASY}. Now suppose that $a$ is a general vertex of $Q$. Then the $\End_\T(T)$-module $\T(T,B_a[-n+1]\oplus A_a[-n]\oplus\coprod_{b\to a}T_b[-2n])$ is projective by induction hypothesis, and $\T(T,U_a[1])$ does not have a non-zero projective summand by \ref{np}, which gives (i). Then the leftmost commutative square in (\ref{eq2n}) shows that the map $\T(T,Y_a[-n])\to\T(T,A_a[-n]\oplus\coprod_{b\to a}T_b[-2n])$ is $0$, hence we obtain an exact sequence
	\[ \xymatrix{0\ar[r]&\T(T,A_a[-n]\oplus\coprod_{b\to a}T_b[-2n])\ar[r]&\T(T,T_a[-2n])\ar[r]&\T(T,Y_a[-n+1]) }.\]
	Since the rightmost term is projective by \ref{YB}, and so is the leftmost term by induction hypothesis, we deduce that the middle term is also projective, which completes the induction step, hence gives our result.
\end{proof}
Finally let us describe the quiver of our hereditary algebra $H=\End_\T(T\oplus\cdots\oplus T[-2n])$. As in \ref{quiver} we can compute the quiver $\widetilde{Q}$ of $H$ from the AR $(2n+4)$-angles. Since $T$ is $(2n+2)$-rigid the quiver $\widetilde{Q}$ has $(2n+1)$ copies of $Q$ as a subquiver. We investigate the arrows from $T_a[-i]$ to $T_b[-j]$. By \ref{nex2} we know that there is an arrow $T_a[-i]\to T_b[-j]$ in $Q$ only if $j-i\in\{0,n,n+1,2n\}$.
\begin{Prop}\label{quiver2}
\begin{enumerate}
	%\item There are no additional arrows $T_a[-i]\to T_b[-i]$ in $\widetilde{Q}$ than in $Q$.
	\item Let $0\leq i\leq n$ and $0\leq i^\prime\leq n-1$. The following are equal.
	\begin{enumerate}
		\renewcommand{\theenumi}{}
		\item\label{-n} The number of arrows from $T_a[-i]$ to $T_b[-i-n]$.
		\item\label{A} The number of summands $T_a$ in $A_b$.
		\item\label{B} The number of summands $T_b$ in $B_a$.
		\item\label{-n-1} The number of arrows from $T_b[-i^\prime]$ to $T_a[-i^\prime-n-1]$.
	\end{enumerate}
	\item There is an arrow $T_a$ to $T_b[-2n]$ only if $n=1$, in which case the number of arrows is given in (\ref{-n-1}) above.%There is an arrow $T_a$ to $T_b[-2n]$ only if $n=1$, in which case is covered by (\ref{-n-1}) above.
\end{enumerate}
\end{Prop}
\begin{proof}
	(1)  By \ref{induction} the values of (\ref{-n}) are equal for $i=0$ and $i=1$, and in view of \ref{shift} they are equal to the ones for all the other $i$. Therefore it is enough to consider the case $i=0$. Similarly by \ref{shift} it suffices to assume $i^\prime=0$. Applying \ref{induction} to $T_b$ and $m=n$ gives (\ref{-n})=(\ref{A}), and \ref{induction2} to $T_a$ and $m=n$ gives (\ref{-n})=(\ref{B}). Also, applying \ref{induction} to $T_a$ and $m=n+1$ gives (\ref{-n-1})=(\ref{B}).\\
	(2)  This follows from \ref{shift}.
\end{proof}
We give several examples of the quiver of $H$ for the case $\End_\T(T)=k$.
\begin{Ex}\label{star}
	Suppose that $\End_\T(T)=k$ and put $m=\dim_k\T(T,T[-n])$. Then the AR $(2n+4)$-angle is of the form
	\[	\xymatrix@!C=2mm@R=3mm{
		&0\ar[r]&\cdots\ar[r]&0\ar[dr]\ar[rr]&&T^m\ar[rr]\ar[dr]&&T^m\ar[dr]\ar[rr]&&0\ar[r]&\cdots\ar[r]&0\ar[dr]&\\
		T\ar[ur]&&&&T[n]\ar[ur]&&Y\ar[ur]&&T[-n]\ar[ur]&&&&T}	\]
	for some $Y\in\T$. We see that for $n=1,2,3$, the algebra $H$ is presented by the quiver below with $m$-fold arrows between the vertices.
	
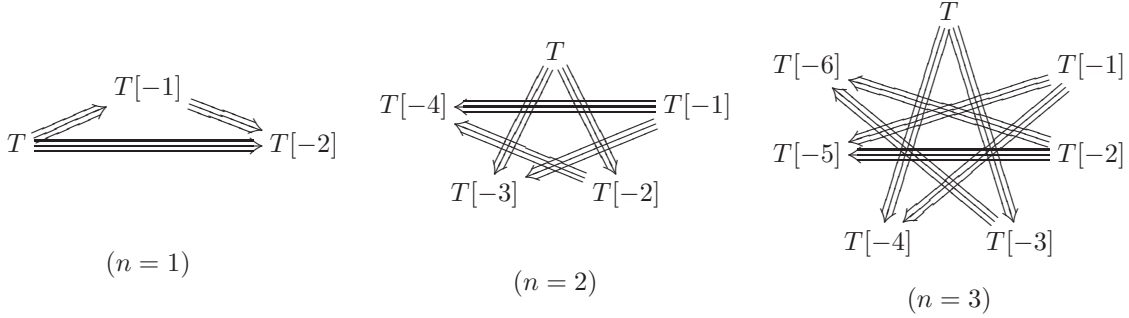
\[
	\xymatrix@R=2mm{
		\\ \\
		&T[-1]\ar@3[dr]\\
		T\ar@3[ur]\ar@3[rr]&&T[-2]
		\\ \\ \\&(n=1)& }\quad
	\xymatrix@R=2mm@!C=1mm{
		\\
		&&T\ar@3[dddl]\ar@3[dddr]&&\\
		T[-4]&&&&T[-1]\ar@3[ddlll]\ar@3[llll]\\
		\\
		&T[-3]&&T[-2]\ar@3[uulll]
		\\ \\ && (n=2)&& }\quad
	\xymatrix@R=2mm@!C=1mm{
		&&T\ar@3[rddddd]\ar@3[lddddd]&&\\
		T[-6]&&&&T[-1]\ar@3[ddddlll]\ar@3[ddllll]\\
		\\
		T[-5]&&&&T[-2]\ar@3[llll]\ar@3[uullll]\\
		\\
		&T[-4]&&T[-3]\ar@3[uuuulll]& \\
		&& (n=3)&& } \]
\end{Ex}
\section{Realizing triangulated categories as cluster categories}
Recall that the {\it cluster category} of a homologically smooth dg algebra $\Pi$ is the Verdier quotient
\[ \C(\Pi):=\per\Pi/\D^b(\Pi) \]
of the perfect derived category by the derived category of dg modules of finite dimensional total cohomology. We think about how a triangulated category can be realized as a cluster category in the above sense, which plays an important role in the proof of main results in this paper.

Let us note a preliminary result on cluster categories of a special class of dg algebras, namely (derived) tensor algebras. This gives a description the cluster category in terms of the triangulated hull of a derived category.
Let $A$ be a finite dimensional algebra of finite global dimension, and $X$ a two-sided tilting complex over $A$, thus $X$ is a complex of $(A,A)$-bimodules such that $F=-\lotimes_AX\colon\D^b(\md A)\to\D^b(\md A)$ gives an autoequivalence. We impose the following two conditions on $X$.
\begin{itemize}
	\item For each $L,M\in\D^b(\md A)$ we have $\Hom_{\D(A)}(L,F^iM)=0$ for all but finitely many $i\in\Z$.
	\item $X$ is concentrated in (cohomological) degree $\leq0$.
\end{itemize}
Replacing $X$ by a projective resolution over $A^e$, we put $\Pi=T_AX$, the tensor algebra of $X$ over $A$.
\begin{Prop}[{\cite[Section 7]{Ke05}, \cite[4.13]{Am09}, see also \cite[Section 8]{ha3}}]\label{cl}
	The dg algebra $\Pi$ is homologically smooth, and its cluster category $\C(\Pi)$ is equivalent to the canonical triangulated hull of $\D^b(\md A)/-\lotimes_AX$.
\end{Prop}

Now we turn to the main subject of this section. Let $\T$ be a $k$-linear, $\Hom$-finite triangulated category with an generator $X$, that is, $\thick X=\T$. Assuming that $\T$ is algebraic and taking a derived endomorphism algebra $\G$ of $X$, we have $\T=\per\G$, and we consider the following condition on the dg algebra $\G$.
\begin{enumerate}
\renewcommand{\labelenumi}{(\alph{enumi})}
\renewcommand{\theenumi}{\labelenumi}
	\item Each cohomology of $\G$ is finite dimensional.
	\item\label{ct} For each $Y\in\per\G$, $H^iY=0$ for $i\ll0$ implies $Y=0$.
	\item\label{sm} The truncation $\Pi:=\G^{\leq0}$ is homologically smooth.
\end{enumerate}
The following example is our principle one for the condition \ref{ct}.
\begin{Lem}\label{mot}
If $e\G\in\per\G$ is $d$-cluster tilting for some idempotent $e\in\G$ and $d\geq1$, then \ref{ct} above hold.
\end{Lem}
\begin{proof}
	Suppose $e\G\in\per\G$ is $d$-cluster tilting and $\Hom_{\D(\G)}(\G,Y[\ll\!0])=0$. Then some $d$-successive extensions from $e\G$ to $Y$ vanish, so there exists $l\in\Z$ such that $\Hom_{\per\G}(e\G,Y[l][i])=0$ for $0\leq i<d$. The vanishing for $0<i<d$ shows $Y[l]\in\add e\G$ since $e\G\in\per\G$ is $d$-cluster tilting, and the vanishing for $i=0$ shows $Y[l]=0$.
\end{proof}
We prove that our triangulated category $\T=\per\G$ arises as the cluster category; in fact it is the cluster category of $\Pi$ which is very simply described as above.
\begin{Thm}\label{str}
Let $\G$ be a dg algebra satisfying the above (a), (b), and (c). Then the functor $-\lotimes_\Pi\G$ induces an equivalence
\[ \xymatrix{\C(\Pi)=\per\Pi/\D^b(\Pi)\ar[r]^-\simeq& \per\G}. \]
\end{Thm}

%\begin{Rem}\label{ky}
%In \cite[4.1]{KY20}, they showed in the setting where $\per\G$ is $d$-CY and $\G\in\per\G$ is $d$-cluster tilting, there exists a ``smooth'' dg algebra $B$ with the same cohomology as $\Pi$ such that $\C(B)\simeq\per\G$.
%\end{Rem}
 
We will temporarily have to work in big triangulated categories. Let $\D$ be a triangulated category with arbitrary (set-indexed) coproducts. Recall that the {\it homotopy colimit} of a sequence
\[ \xymatrix{ M_0\ar[r]^-{f_0}& M_1\ar[r]^-{f_1}&M_2\ar[r]^-{f_2}&\cdots } \]
in $\D$ is an object $M$ which fits into a triangle
\[ \xymatrix{ \discoprod_{i\geq0} M_i\ar[r]^-f& \discoprod_{i\geq0}M_i\ar[r]& M\ar[r]&\discoprod_{i\geq0}M_i[1] } \]
with $f$ having components $M_i\xrightarrow{\left( \begin{smallmatrix}1\\-f_i\end{smallmatrix}\right) } M_i\oplus M_{i+1}\hookrightarrow\coprod_{i\geq0}M_i$. Thus $M$ is uniquely determined up to (non-unique) isomorphism, which is denoted by $\hocolim_{i\geq0}M_i$.

We will use some easy computations on homotopy colimits. For a complex $M=(\cdots\to M^{i-1}\to M^i\to M^{i+1}\to\cdots)$ and $p\in\Z$, we denote by $M^{\leq p}$ the truncated complex, that is, the complex $(\cdots\to M^{p-1}\to Z^pM\to 0\to\cdots)$. We then have an ascending sequence $M^{\leq0}\to M^{\leq1}\to M^{\leq2}\to\cdots$.
\begin{Lem}\label{hc}
Let $\L$ be an arbitrary negative dg algebra.
\begin{enumerate}
	\item We have $M=\hocolim_{p\geq0}M^{\leq p}$ for all $M\in\D(\L)$.
	\item For each $L\in\per\L$ and a sequence $M_0\to M_1\to \cdots$ in $\D(\L)$, there exists a natural isomorphism
	\[ \colim_{p\geq0}\Hom_{\D(\L)}(L,M_p)\xrightarrow{\simeq}\Hom_{\D(\L)}(L,\hocolim_{p\geq0}M). \]
	\item For each $L\in\per\L$, $M\in\D(\L)$, and $N\in\D(\L^e)$, there is a natural isomorphism
	\[ \colim_{p\geq0}\Hom_{\D(\L)}(L,M\lotimes_\L N^{\leq p})\xrightarrow{\simeq}\Hom_{\D(\L)}(L,M\lotimes_\L N). \]
\end{enumerate}
\end{Lem}
\begin{proof}
	(1) is easy and (2) is well-known \cite{Ne92}. Then (3) follows since $M\lotimes_\L N=\hocolim_{p\geq0}(M\lotimes_\L N^{\leq p})$.
\end{proof}

The following observation is crucial.
\begin{Lem}\label{below}
We have $L\lotimes_\Pi\G=0$ for any $L\in\D^b(\Pi)$.
\end{Lem}
\begin{proof}
	Consider the triangle
	\[ \xymatrix{ \Pi\ar[r]& \G\ar[r]& W \ar[r]& \Pi[1] } \]
	in $\D(\Pi^e)$ obtained from the inclusion $\Pi\to\G$. Since $\Pi$ is homologically smooth we have $L\in\per\Pi$, hence $M:=L\lotimes_\Pi\G\in\per\G$. We show that $M$ is bounded below, and the claim follows from our assumption \ref{ct}. Applying $L\lotimes_\Pi-$ to the above triangle yields a triangle $L\to M\to L\lotimes_\Pi W\to L[1]$, in which $L\in\D^b(\Pi)$ is bounded below, and $L\lotimes_\Pi W\in\thick_{\D(\Pi)}W$ is also bounded below. Therefore the middle term $M$ is bounded below.
\end{proof}

Now we are ready to prove our general structure theorem.
\begin{proof}[Proof of \ref{str}]
	First note that \ref{below} gives rise to a functor we want: The tensor product $-\lotimes_\Pi\G$ indeed induces a functor $\C(\Pi)\to\per\G$.	
	We have to show that the map $\Hom_{\C(\Pi)}(L,M)\to\Hom_{\D(\G)}(L\lotimes_\Pi\G,M\lotimes_\Pi\G)$ is bijective for each $L,M\in\per\Pi$. Note that the right-hand-side is isomorphic to $\Hom_{\D(\Pi)}(L,M\lotimes_\Pi\G)$ by adjunction, and then the map in question is induced by applying $\Hom_{\D(\Pi)}(L,-)$ to the natural map $M=M\lotimes_\Pi\Pi\to M\lotimes_\Pi\G$, that is, we have a commutative diagram
	\[ \xymatrix@R=6mm{
		\Hom_{\D(\Pi)}(L,M)\ar[r]\ar[d]&\Hom_{\D(\Pi)}(L,M\lotimes_\Pi\G)\ar@{=}[d]\\
		\Hom_{\C(\Pi)}(L,M)\ar[r]&\Hom_{\D(\G)}(L\lotimes_\Pi\G,M\lotimes_\Pi\G). } \]
	
	We first show the injectivity. Let $L\to M$ be a morphism in $\per\Pi$ such that the composite $L\to M\to M\lotimes_\Pi\G$ is $0$. Put $W=\cone(\Pi\to\G)$ as in \ref{below}. Then it factors through $M\lotimes_\Pi W[-1]$ as in the diagram below, thus through $M\lotimes_\Pi (W[-1])^{\leq p}$ for some $p\geq0$ by \ref{hc}(3). Now this is in $\D^b(\Pi)$ since $M\in\per\Pi$ and $(W[-1])^{\leq p}\in\D^b(\Pi)$. Therefore $L\to M$ is $0$ in $\C(\Pi)$.
	\[ \xymatrix@R=5mm{
		&L\ar[d]\ar@/_/@{-->}[dl]&&\\
		M\lotimes_\Pi W[-1]\ar[r]& M\ar[r]& M\lotimes_\Pi\G\ar[r]& M\lotimes_\Pi W } \]
	
	We next show the surjectivity. Let $f\colon L\to M\lotimes_\Pi\G$ be a morphism in $\D(\Pi)$. It factors through $M\lotimes_\Pi\G^{\leq p}$ for some $p\geq0$ by \ref{hc}(3). Then the canonical map $s\colon M\to M\lotimes_\Pi\G^{\leq p}$ has finite dimensional mapping cone $M\lotimes_\Pi W^{\leq p}$, so we have the diagram $L\xrightarrow{g} M\lotimes_\Pi\G^{\leq p}\xleftarrow{s} M$ in $\per\Pi$ which presents a morphism in $\C(\Pi)$. We verify that this is mapped to the original morphism in $\D(\G)$.
	\[ \xymatrix@R=6mm{
		L\ar[r]^-g\ar[d]&M\lotimes_\Pi\G^{\leq p}\ar[d]\ar@{-->}[dr]^-t&M\ar[d]\ar[l]_-s\\
		L\lotimes_\Pi\G\ar[r]&M\lotimes_\Pi\G^{\leq p}\lotimes_\Pi\G&M\lotimes_\Pi\G\ar[l]_-\simeq } \]
	Applying $-\lotimes_\Pi\G$ to $s^{-1}g\colon L\to M$ in $\C(\Pi)$ yields the second row in the diagram above, which is the morphism $s^{-1}g\otimes1\colon L\lotimes_\Pi\G\to M\lotimes_\Pi\G$ in $\per\G$. Under the adjunction it becomes $L\xrightarrow{g}M\lotimes_\Pi\G^{\leq p}\to M\lotimes_\Pi\G^{\leq p}\lotimes_\Pi\G\xleftarrow{\simeq}M\lotimes_\Pi\G$ in $\D(\Pi)$. On the other hand, the original morphism $f\colon L\to M\lotimes_\Pi\G$ is $L\xrightarrow{g}M\lotimes_\Pi\G^{\geq p}\xrightarrow{t}M\lotimes_\Pi\G$. Therefore it remains to show that the lower triangle in the right square in the above diagram is commutative. Clearly the upper triangle is commutative, so the difference of two maps $M\lotimes_\Pi\G^{\geq p}\xrightarrow{}M\lotimes_\Pi\G$ factors through $\cone s=M\lotimes_\Pi W^{\leq p}\in\D^b(\Pi)$. But there is no non-zero map in $\D(\Pi)$ from $X\in\D^b(\Pi)$ to $Y\in\D(\G)$. Indeed, we have $\Hom_{\D(\Pi)}(X,Y)=\Hom_{\D(\G)}(X\lotimes_\Pi\G,Y)$, which is $0$ by \ref{below}. This completes the proof of surjectivity.
		
	Finally we see that the functor $\C(\Pi)\to\per\G$ is clearly dense.
\end{proof}

Let us note some complementary observations which we use later. Recall that a morphism $A\to B$ of dg algebras is a {\it localization} (or a {\it homological epimorphism}) if the restriction functor $\D(B)\to\D(A)$ is fully faithful. This is equivalent to saying that a natural map $B\lotimes_AB\to B$ is a quasi-isomorphism.
\begin{Prop}\label{loc}
In the situation of \ref{str}, we have the following.
\begin{enumerate}
	\item The map $\Pi\to\G$ is a localization.
	\item Under the identification $\D(\G)\subset\D(\Pi)$, we have $\D(\G)=\{ X\in\D(\Pi) \mid \RHom_\Pi(X,\Pi)=0\}$.
\end{enumerate}
\end{Prop}
\begin{proof}
	(1)  We prove that the multiplication map $\G\lotimes_\Pi\G\to\G$ is an isomorphism. Applying $-\lotimes_\Pi\G$ to the triangle $\Pi\to\G\to W\to \Pi[1]$ obtained from $\Pi\to\G$, we have a triangle
	\[ \xymatrix{
		\G\ar[r]&\G\lotimes_\Pi\G\ar[r]&W\lotimes_\Pi\G\ar[r]&\G[1] }. \]
	Since $W$ is a homotopy colimit of the finite dimensional $\Pi$-modules $W^{\leq p}$ by \ref{hc}(1), we have $W\lotimes_\Pi\G=0$ by \ref{below}, thus the first map is an isomorphism. Being its left inverse, the multiplication map is also an isomorphism.
	
	(2)  The inclusion $\subset$ is equivalent to saying $\RHom_\Pi(\G,\Pi)=0$. Note that this is the dual of $\G\lotimes_\Pi D\Pi$, which equals $\hocolim_{p\geq0}\G\lotimes_\Pi(D\Pi)^{\leq p}$ by \ref{hc}(3), thus is $0$ by \ref{below}. Next we show the converse inclusion. Let $X\in\D(\Pi)$ such that $\RHom_\Pi(X,\Pi)=0$. Applying $X\lotimes_\Pi-$ to $\Pi\to\G\to W\to\Pi[1]$, we have a triangle $X\to X\lotimes_\Pi\G\to X\lotimes_\Pi W\to X[1]$ in $\D(\Pi)$. Therefore we have to show $X\lotimes_\Pi W=0$, which is equivalent to the vanishing of its dual $\RHom_\Pi(X,DW)$. Since $\RHom_\Pi(X,\Pi)=0$ and $\Pi$ is smooth, we have $\RHom_\Pi(X,L)=0$ for all $L\in\D^b(\Pi)$, and consequently $\RHom_\Pi(X,L)=0$ for all $L\in\D(\Pi)$ which can be written as a homotopy limit of objects in $\D^b(\Pi)$. On the other hand, $W=\hocolim_{p\geq0}W^{\leq p}$ by \ref{hc}(1), so $DW$ is a homotopy limit of finite dimensional $\Pi$-modules. It follows that $\RHom_\Pi(X,DW)=0$, as desired.
\end{proof}

\section{Proof of main theorems}
We apply the result from previous section to prove our main result \ref{main} of this paper which is an explicit structure theorem for CY categories with cluster tilting objects. It will be proved in \ref{root} in this section.

Let us give a sketch of the proof. Let $\T$ be an algebraic $d$-CY triangulated category with $d\geq2$. Suppose that there exists a $d$-cluster tilting object $T\in\T$. Put
\[ X=T\oplus T[-1]\oplus\cdots\oplus T[-(d-2)]. \]
Since $X$ has a $d$-cluster tilting object $T$ as a direct summand, we can apply \ref{str}. For this we need to know its truncated derived endomorphism algebra $\Pi$. It turns out that its cohomology which is given by
\[ S=\discoprod_{i\geq0}\T(X,X[-i]), \]
is intrinsically formal (\ref{dim2}, \ref{formal}), that is, any dg algebra with cohomology $S$ is formal, thus $\Pi$ is nothing but $S$ regarded as a dg algebra with trivial differentials. We are therefore reduced to study the algebra $S$. For this we consider the functor
\[ \xymatrix{ F=\T(X,-)\colon \T\ar[r]& \md H } \]
to the category of modules over $H=\End_\T(X)$ and compare $\T$ with it. An important observation is that any object in $\T$ has a $2$-term resolution by objects from $\add X$ (\ref{last}), by which we show that this functor is close enough to being an equivalence; in fact it is full (\ref{full}) and induces an equivalence between stable categories (\ref{st}). Using this functor $F$, we prove in \ref{alg} that there is an isomorphism
\[ S=T_H\T(X,X[-1]) \]
of graded algebras, where the right-hand-side is the tensor algebra of an $(H,H)$-bimodule $\T(X,X[-1])$. In view of \ref{cl}, this description of $S$ as a tensor algebra allows us to write the cluster category $\C(\Pi)$ as an orbit category of $\D^b(\md H)$.

\subsection{The structure theorem}
%Given a triangulated category $\T$ with a $d$-cluster titling object $T$, any object in $\T$ has a $d$-term resolution by objects in $\add T$. This, however, does not behave well under the functor $\T(T,-)$ for $d\geq3$ since a concatenation of triangles is not necessarily mapped to a long exact sequence in $\md\End_\T(T)$.
The first and an important step toward \ref{root} is the observation that any objects in $\T$ has a $2$-term resolution by $X=T\oplus\cdots\oplus T[-(d-2)]$. For subcategories $\U, \V\subset\T$, we denote by $\U\ast\V$ the full subcategory of $\T$ formed by $A\in\T$ such that there exists a triangle $U\to A\to V\to U[1]$ with $U\in\U$ and $V\in\V$. This operation $\ast$ is associative by the octahedral axiom. For objects $U, V\in\T$, we will simply write $U\ast V$ for $\add U\ast\add V$. Also we denote by $\U\vee\V$ the smallest additive subcategory containing $\U$ and $\V$, and similarly $U\vee V$ for $\add U\vee\add V$.
\begin{Prop}\label{last}
	Let $\T$ be a triangulated category and $T\in\T$ a $d$-rigid object. Put
	\begin{equation*}
		\begin{aligned}
			X&=T\oplus\cdots\oplus T[-(d-3)]\oplus T[-(d-2)],\\
			Y&=T\oplus\cdots\oplus T[-(d-3)],
		\end{aligned}
	\end{equation*}
	and suppose that $\End_\T(Y)$ is hereditary. Then we have
	\[ T[-(d-2)]\ast\cdots\ast T\ast T[1]=X\ast X[1]. \]
	In particular, if $T\in\T$ is $d$-cluster tilting, then $\T=X\ast X[1]$.
\end{Prop}
\begin{proof}
	The inclusion $\supset$ is easy. Indeed, we have $X\ast X[1]\subset(T[-(d-2)]\ast\cdots\ast T)\ast(T[-(d-3)]\ast\cdots\ast T[1])$, which equals $T[-(d-2)]\ast\cdots\ast T[1]$ by $d$-rigidity of $T$.	We prove the converse inclusion by induction on $d$. The case $d=2$ is clear (where we understand $Y=0$ for $d=2$), so let $d\geq3$.
	
	Let $A\in T[-(d-2)]\ast\cdots\ast T\ast T[1]$. Then there is a triangle
	\[ \xymatrix{ T_0\ar[r]& B\ar[r]& A\ar[r]& T_0[1]} \]
	in $\T$ with $T_0\in\add T$ and $B\in T[-(d-2)]\ast\cdots\ast T$. By induction hypothesis applied to a $(d-1)$-rigid object $T[-1]$, we see that $B\in Y[-1]\ast Y$, and the same is true for any direct summand of $B$ since $T[-(d-2)]\ast\cdots\ast T$ is closed under direct summands (\cite[2.1]{IYo}).
	
	Now write $B=B^\prime\oplus T_1$ with $T_1\in\add T$ and $\add B^\prime\cap\add T=0$. We know that $B^\prime\in Y[-1]\ast Y$. We claim that there exists a triangle
	\begin{equation*}
		\xymatrix{ Y_0[-1]\ar[r]& B^\prime\ar[r]& Y_1\ar[r]& Y_0 }
	\end{equation*}
	with $Y_0,Y_1\in\add Y$ which induces a surjection $\T(Y,Y_0[-1])\twoheadrightarrow\T(Y,B^\prime)$. Since we can discuss summandwise, it is enough to consider each indecomposable direct summand $B_0$ of $B^\prime$. If $B_0\in\add Y$, then we have $B_0\in\add Y[-1]$ since $B_0\not\in\add T$, so we can take $Y_0[-1]=B_0$ and $Y_1=0$, which gives a desired triangle. If $B_0\not\in\add Y$, we show that any triangle as above has the desired surjectivity. Indeed, since $\End_\T(Y)$ is hereditary, the morphism $Y_1\to Y_0$, which becomes under the equivalence $\T(Y,-)\colon\add Y\to\proj\End_\T(Y)$ the morphism between projective $\End_\T(Y)$-modules, is isomorphic to the direct sum of $Y_1^\prime\to Y_0$ inducing an injection $\T(Y,Y_1^\prime)\hookrightarrow\T(Y,Y_0)$, and $Y_1^{\prime\prime}\to0$. This forces $Y_1^{\prime\prime}\in\add B_0$, so $Y_1^{\prime\prime}$ has to be $0$ since $B_0$ is an indecomposable $\not\in\add Y$. It follows that $\T(Y,Y_1)\to\T(Y,Y_0)$ is injective, hence $\T(Y,Y_0[-1])\to\T(Y,B_0)$ is surjective, which finishes the proof of the claim.
	
	By the triangle $T_0\to B\to A\to T_0[1]$ with $B=B^\prime\oplus T_1$, we can form an octahedral on the left below.
	\[ 
	\xymatrix@!R=7mm@!C=7mm{
		&A^\prime[-1]\ar[d]\ar@{=}[r]&A^\prime[-1]\ar[d]&\\
		\bullet\ar[r]\ar@{=}[d]& T_0\ar[r]\ar[d]&T_1\ar[r]\ar[d]&\bullet\ar@{=}[d] \\
		\bullet\ar[r]& B^\prime\ar[r]\ar[d]&A\ar[r]\ar[d]&\bullet\\
		&A^\prime\ar@{=}[r]&A^\prime& }\qquad\qquad
	\xymatrix@!R=7mm@!C=7mm{
		&T_0\ar[d]\ar@{=}[r]&T_0\ar[d]&\\
		Y_1[-1]\ar[r]\ar@{=}[d]&Y_0[-1]\ar[r]\ar[d]&B^\prime\ar[r]\ar[d]&Y_1\ar@{=}[d]\\
		Y_1[-1]\ar[r]&C\ar[r]\ar[d]&A^\prime\ar[r]\ar[d]&Y_1\\
		&T_0[1]\ar@{=}[r]&T_0[1]& } \]	
	On the other hand, taking a triangle for $B^\prime$ in the claim above, the morphism $T_0\to B^\prime$ can be lifted to $T_0\to Y_0[-1]$ since $T_0\in\add Y$. This gives another octahedral as in the right diagram above.
	
	We can now reach the conclusion using these diagrams. Looking at the second octahedral, we have $C\in Y[-1]\ast T[1]$ by the left vertical triangle, so using the lower horizontal one, we see $A^\prime\in Y[-1]\ast T[1]\ast Y$, which equals $Y[-1]\ast X[1]$ by $T[1]\ast Y=T[1]\vee Y=\add X[1]$. Now we move to the left octahedral. By the right vertical triangle, we see $A\in T\ast Y[-1]\ast X[1]$, which is $X\ast X[1]$ by $T\ast Y[-1]=T\vee Y[-1]=\add X$.
\end{proof}

Now we place ourselves in the setup as in \ref{root}: $\T$ is a $d$-CY triangulated category and $T\in\T$ is $d$-cluster tilting. Put $X=T\oplus\cdots\oplus T[-(d-2)]$ and $H=\End_\T(X)$, which we assume to be hereditary. For the moment we do not need that $H$ is $1$-representation infinite. Let us note a complementary observation on the $2$-term resolution, although it will not be used later in this paper.
\begin{Rem}
	Let $A\in\T$ and let $X_0\to A$ be any right $(\add X)$-approximation. Then its mapping cocone is in $\add X$.
\end{Rem}
\begin{proof}
	%If one right $(\add X)$-approximation has the property that its mapping cone is in $\add X[1]$, then the same holds for any right $(\add X)$-approximation. Indeed, let $f\colon X_0\to A$ be a right $(\add X)$-approximation with $\cone f\in\add X[1]$. Decomposing $f$ as a direct sum of right minimal $f_0\colon X_0^\prime\to A$ and $0\colon X_0^{\prime\prime}\to 0$, we see that the minimal right $(\add X)$-approximation $f_0$ satisfies $\cone f_0\in\add X[1]$. Therefore any approximation $f_1\colon X_1\to A$, being isomorphic to a direct sum of $f_0$ and $0\colon X_1^\prime\to0$, has $\cone f_1=\cone f_0\oplus X_1^\prime[1]\in\add X[1]$.
	We may assume that the approximation is minimal and $A$ is indecomposable. If $A\in\add T[1]$, then the minimal right $(\add X)$-approximation is $0\to A$, thus its mapping cocone is $A[-1]\in\add X$. If $A\in\add X$, then the minimal approximation is the identity, thus its mapping cocone is $0$. Finally suppose that $A\not\in\add(X\oplus T[1])$. Then by \ref{ses}(\ref{F}) below, in any triangle $X_1\to X_0\xrightarrow{f} A\to X_1[1]$ in \ref{last}, the map $f$ is a right $(\add X)$-approximation. In particular, there exists a right $(\add X)$-approximation of $A$ whose mapping cocone is in $\add X$. It follows that the minimal right $(\add X)$-approximation has the same property.
\end{proof}

Now we consider the functor
\[ \xymatrix{ F=\T(X,-)\colon\T\ar[r]& \md H}. \]
Let us give some easy observations which we will often use.
\begin{Lem}\label{easy}
	\begin{enumerate}
		\item\label{proj} $FX$ and $FX[1]$ are projective $H$-modules.
		\item\label{inj} $FX[d]$ and $FX[d-1]$ are injective $H$-modules.
		\item\label{van} For $A\in\T$, we have $FA=0$ if and only if $A\in\add T[1]$. %In particular, the kernel of the induced functor $\overline{F}\colon\T/[T[1]]\to\md H$ lies in the radical of $\T/[T[1]]$, that is, if $\overline{F}(f)=0$ then $f\in J_{\T/[T[1]]}$.
		\item\label{nak} We have a commutative diagram of equivalences
		\[ \xymatrix@!R=3mm{
			\add X\ar[r]^-F\ar[d]_-{[d]}&\proj H\ar[d]^-{-\otimes_HDH}\\
			\add X[d]\ar[r]^-F&\inj H. } \]
	\end{enumerate}
\end{Lem}
\begin{proof}
	We have $FX=H\in\proj H$. Also all the direct summands of $X[1]$ except $T[1]$ are in $\add X$, and $FT[1]=0$ since $T$ is $d$-rigid, so $FX[1]\in\proj H$. This proves (\ref{proj}). Using Serre duality we similarly have (\ref{inj}). We see (\ref{van}) since $T\in\T$ is $d$-cluster tilting. Finally (\ref{nak}) is clear.
\end{proof}

We next discuss more essential properties of the functor $F$.
\begin{Lem}\label{full}
	The functor $F=\T(X,-)\colon \T\to\md H$ is full.
\end{Lem}
\begin{proof}
	We have to show that for each $A,B\in\T$, the functor $F$ induces surjections $F_{A,B}\colon\T(A,B)\to{\Hom}_H(FA,FB)$.
	First $F_{A,B}$ is clearly bijective if $A\in\add X$. Note also that it is surjective if $A\in\add X[1]$. Indeed, all the direct summands of $X[1]$ except $T[1]$ lies in $\add X$, so we may assume $A=T[1]$. But we have $FT[1]=0$ by \ref{easy}(\ref{van}), hence $F_{T[1],B}$ is surjective.
	
	Now we prove $F_{A,B}$ is surjective for all $A\in \T$. By \ref{last}, we have a triangle
	\[ \xymatrix{X_1\ar[r]&X_0\ar[r]&A\ar[r]&X_1[1] }\]
	with $X_0,X_1\in\add X$. Applying $\T(-,B)$ yields the exact sequence in the first row of the diagram below. On the other hand, applying $F=\T(X,-)$ gives an exact sequence $FX_1\to FX_0\to FA\xrightarrow{u} FX_1[1]\xrightarrow{v} FX_0[1]$ in $\md H$. Now since $FX_0[1]$ is projective by \ref{easy}(\ref{proj}) the morphism $v$ is a split epimorphism to its image, hence applying $\Hom_H(-,FB)$ yields a complex in the second row, which is acyclic at the middle term.
	\[ \xymatrix{
		\T(X_0[1],B)\ar[r]\ar@{->>}[d]&\T(X_1[1],B)\ar[r]\ar@{->>}[d]& \T(A,B)\ar[r]\ar[d]& \T(X_0,B)\ar[r]\ar[d]^\rsimeq&\T(X_1,B)\ar[d]^\rsimeq \\ {}_H(FX_0[1],FB)\ar[r]&{}_H(FX_1[1],FB)\ar[r]&{}_H(FA,FB)\ar[r]&{}_H(FX_0,FB)\ar[r]&{}_H(FX_1,FB) } \]
	Now the right two vertical maps are isomorphisms and the left two are surjective by the starting remark. It easily follows that the middle map is surjective.
\end{proof}

Moreover this functor gives equivalences when passing to ideal quotients.
\begin{Prop}\label{st}
	\begin{enumerate}
		\item The functor $F=\T(X,-)\colon\T\to\md H$ induces stable equivalences
		\[ \xymatrix{\underline{F}\colon\T/[X\oplus T[1]]\ar[r]^-\simeq& \smd H } \text{ and } \xymatrix{\overline{F}\colon\T/[T[1]\oplus X[d]]\ar[r]^-\simeq& \injsmd H.} \]
		\item We have an isomorphism of functors $F\circ[d-1]\simeq\tau\circ F\colon \T\to\md H$. Consequently there exists a commutative diagram of equivalences
		\[ \xymatrix{
			\T/[X\oplus T[1]]\ar[d]_-{[d-1]}\ar[r]^-{\underline{F}}&\smd H\ar[d]^-\tau\\
			\T/[T[1]\oplus X[d]]\ar[r]^-{\overline{F}}&\injsmd H. } \]
	\end{enumerate}
\end{Prop}
We first prove (1).
\begin{proof}[Proof of \ref{st}(1)]
	Since $FX=H$, $FT[1]=0$, and $FX[d]=DH$ (see \ref{easy}), the functor $F$ induces the functors $\underline{F}$ and $\overline{F}$ on stable categories.
	
	We only prove that $\underline{F}$ is an equivalence. The statement for $\overline{F}$ is proved dually. We immediately see that this is full by \ref{full}. We show that this is faithful. Let $f\colon A\to B$ be a morphism in $\T$ such that $Ff$ in $\md H$ factors through a projective $H$-module. Taking a right $(\add X)$-approximation $X^\prime\to B$, the morphism $Ff$ factors through a surjection $FX^\prime\twoheadrightarrow FB$ as in the left diagram below. 
	\[ 
	\xymatrix{
		&& FA\ar[d]^-{Ff}\ar[dl]&\\
		&FX^\prime\ar@{->>}[r]& FB, }\qquad
	\xymatrix{
		X_1\ar[r]&X_0\ar[r]& A\ar[r]\ar[d]^-f\ar[dl]& X_1[1]\ar@{-->}[dl]\\
		&X^\prime\ar[r]& B&}
	\]
	Since $F$ is full, the lift $FA\to FX^\prime$ comes from a morphism $A\to X^\prime$ in $\T$. Now take a triangle $X_1\to X_0\to A\to X_1[1]$ in \ref{last} and consider the right diagram above. Since the triangle formed by $A$, $X^\prime$, and $B$ becomes commutative under $F=\T(X,-)$, the maps $f\colon A\to B$ and $A\to X^\prime\to B$ coincide when precomposing $X_0\to A$. Therefore the difference of these two maps factors through $X_1[1]$, and $f$, being the sum of $A\to X^\prime\to B$ and $A\to X_1[1]\to B$, factors through an object in $\add(X\oplus X[1])=\add(X\oplus T[1])$. This proves faithfulness.
	
	Finally we show that $\underline{F}$ is dense. Let $M\in\md H$ and consider the projective resolution $0\to P_1\to P_0\to M\to0$. Let $X_1\to X_0$ be a morphism in $\add X\subset \T$ corresponding to $P_1\to P_0$ under the equivalence $\proj H\simeq\add X$, and complete it to a triangle $X_1\to X_0\to A\to FX_1[1]$. Applying $F$ we have an exact sequence below, with $M$ (resp. $P$) the cokernel of $FX_1\to FX_0$ (resp. $FX_0\to FA$).
	\[ \xymatrix@R=1mm@!C=3mm{
		FX_1\ar@{=}[dd]\ar[rr]&&FX_0\ar@{=}[dd]\ar[rr]\ar[dr]&&FA\ar[rr]\ar[dr]&&FX_1[1]\\
		&&&M\ar[ur]&&P\ar[ur]&\\
		P_1\ar[rr]&&P_0\ar[ur]&&&&} \]
	Note that $P$ is a submodule a projective $H$-module $FX_1[1]$ (\ref{easy}(\ref{proj})), hence is projective. Then the short exact sequence $0\to M\to FA\to P\to 0$ splits, thus $FA\simeq M$ in $\smd H$.
\end{proof}

This yields a certain desirable behavior of objects under $F$.
\begin{Lem}\label{inddet}
	Let $A\in\T$ be an indecomposable object.
	\begin{enumerate}
		\item\label{ind} If $A\not\simeq T[1]$ then $FA$ is indecomposable.
		\item\label{pj} $FA\in\md H$ is projective if and only if $A\in\add(X\oplus T[1])$.
		\item $FA\in\md H$ is injective if and only if $A\in\add(T[1]\oplus X[d])$.
	\end{enumerate}
\end{Lem}
\begin{proof}
	Immediate by \ref{st}(1).%(1)  By \ref{full} we have a surjection $\End_\T(A)\to\End_H(FA)$, and $FA\neq0$ by \ref{easy}(\ref{van}).\\
	%(2)  We have seen `if' part in \ref{easy}(\ref{proj}), so we prove `only if' part. Suppose $FA\in\md H$ is projective. Assuming $A\not\simeq T[1]$, we show $A\in\add X$. Since $F$ gives an equivalence $\add X \to \proj H$, there is indecomposable $X_0\in\add X$ such that $FX_0\simeq FA$. By \ref{full} there are $f\colon X_0\to A$ and $g\colon A\to X_0$ inducing mutually inverse isomorphisms under $F$. It follows by \ref{easy}(\ref{van}) that $1_{X_0}-gf\in\rad\End_\T(X_0)$ and $1_A-fg\in\rad\End_\T(A)$. Therefore $f$ and $g$ are isomorphisms.%Since $F_{X_0,A}\colon \T(X_0,A)\to\Hom_H(FX_0,FA)$ is an isomorphism, we have a morphism $f\colon X_0\to A$ in $\T$ inducing an isomorphism $FX_0\to FA$. Completing $f$ to the triangle $X_0\to A\to B\to$ and applying $F$ we have an exact sequence $FX_0\xrightarrow{\simeq}FA\to FB\to FX_0[1]$
\end{proof}

We next discuss a relationship between triangles in $\T$ and short exact sequences in $\md H$.
\begin{Lem}\label{ses}
	Let $A\in\T$ be an indecomposable object and let
	\[ \xymatrix{X_1\ar[r]^-f& X_0\ar[r]& A\ar[r]&X_1[1] } \]
	be a triangle with $X_0, X_1\in\add X$. Suppose that $A\not\in\add(X\oplus T[1])$.
	\begin{enumerate}
		\item\label{F} The functor $F=\T(X,-)$ induces a short exact sequence
		\[ \xymatrix{0\ar[r]& FX_1\ar[r]& FX_0\ar[r]& FA\ar[r]& 0 }. \]
		\item\label{G} The functor $G=\T(X,-[d])$ induces a short exact sequence
		\[ \xymatrix{0\ar[r]& GA[-1]\ar[r]& GX_1\ar[r]& GX_0\ar[r]& 0 }. \]
	\end{enumerate}
\end{Lem}
\begin{proof}
	We only prove (\ref{F}). We have an exact sequence
	\[ \xymatrix{FA[-1]\ar[r]&FX_1\ar[r]^-u& FX_0\ar[r]& FA\ar[r]^-v&FX_1[1] }. \]
	Since $H$ is hereditary, each of the morphisms $u$ and $v$ is a split epimorphism to its image by \ref{easy}(\ref{proj}). 
	
	We first show that $u$ is injective. Note that $f$ is right minimal. Indeed, if $f$ is not right minimal, then $X_1$ and $A[-1]$ share a direct summand, which is impossible by our assumption $A[-1]\not\in\add X$. It follows that $u=Ff$ is also right minimal, hence injective.
	
	We next show that $v=0$. Since $A\not\in\add(X\oplus T[1])$, we see $FA$ is indecomposable non-projective $H$-module by \ref{inddet}(\ref{ind})(\ref{pj}), thus $v$ has to be $0$.
\end{proof}
Now we can prove the second part of \ref{st}.
\begin{proof}[Proof of \ref{st}(2)]
	Let $A, B\in\T$ without a direct summand in $\add(X\oplus T[1])$, and let $f\colon A\to B$ be a morphism in $\T$. We compute the AR translation of $Ff\colon FA\to FB$ in $\md H$. Let $X_1\to X_0\to A\to X_1[1]$ and $Y_1\to Y_0\to B\to Y_1[1]$ be triangles with $X_0, X_1, Y_0, Y_1\in \add X$ in \ref{last}. By \ref{ses}(\ref{F}), $f$ can be lifted to a morphism of triangles
	\[ \xymatrix{
		A[-1]\ar[d]_-{f[-1]}\ar[r]& X_1\ar[r]\ar@{-->}[d]& X_0\ar[r]\ar@{-->}[d]& A\ar[d]^-f \\
		B[-1]\ar[r]& Y_1\ar[r]& Y_0\ar[r]& B. } \]
	Applying $F=\T(X,-)$ it	induces a commutative diagram of short exact sequence
	\[ \xymatrix{
		0\ar[r]& \T(X,X_1)\ar[r]\ar[d]& \T(X,X_0)\ar[r]\ar[d]& \T(X,A)\ar[r]\ar[d]^-{Ff}& 0\\
		0\ar[r]& \T(X,Y_1)\ar[r]& \T(X,Y_0)\ar[r]& \T(X,B)\ar[r]& 0 } \]
	again by \ref{ses}(\ref{F}). In view of \ref{easy}(\ref{nak}), applying $-\otimes_HDH$ to these sequences gives complexes
	\[ \xymatrix{
		0\ar[r]& \T(X,A[d-1])\ar[r]\ar[d]& \T(X,X_1[d])\ar[r]\ar[d]& \T(X,X_0[d])\ar[r]\ar[d]& 0\\
		0\ar[r]& \T(X,B[d-1])\ar[r]& \T(X,Y_1[d])\ar[r]& \T(X,Y_0[d])\ar[r]& 0,} \]
	which are exact by \ref{ses}(\ref{G}). This shows $\T(X,A[d-1])\simeq\tau \T(X,A)$ and $\T(X,B[d-1])\simeq\tau\T(X,B)$, and $Ff[d-1]=\tau(Ff)$ under these isomorphisms, that is, $F\circ[d-1]\simeq\tau\circ F$ as functors.
\end{proof}

From now on we assume that $H$ is $1$-representation infinite. The reason we (have to) assume this is the following nice behavior of objects under suspension of $\T$, which allows us to use \ref{ses} quite freely.
\begin{Lem}\label{non}
	The objects $\{ T[-l]\mid l\in\Z\}$ do not mutually share a direct summand, that is, $\add T[-i]\cap\add T[-j]=0$ whenever $i\neq j$. In particular, $T[-l]$ does not share a direct summand with $X\oplus T[1]$ whenever $l\geq d-1$.
\end{Lem}
\begin{proof}
	By \ref{st}(2), the functor $F$ takes $T[\leq\!0]$ to preprojective $H$-modules, $T[1]$ to $0$, and $T[\geq\!2]$ to preinjective $H$-modules. Therefore the $H$-modules $FT[-l]$ mutually do not share a summand for all $l\in\Z$ since $H$ is $1$-representation infinite. Then the statements follow from \ref{inddet}.%thus $T[-l]\not\in\add(X\oplus T[1])$ by \ref{inddet}(\ref{pj}).
\end{proof}

Now consider the algebra
\[ S=\discoprod_{i\geq0}\T(X,X[-i]). \]
Also let $U:=\T(X,X[-1])$, which we view as a bimodule over $H$.
\begin{Prop}\label{alg}
	We have an isomorphism $T_HU\simeq S$ of graded algebras.
\end{Prop}
\begin{proof}
	We have to show that $U\otimes_H U\otimes_H\cdots\otimes_HU\simeq\T(X,X[-l])$ for all $l\geq0$. By induction it is enough to show that the natural map
	\[ \xymatrix{\T(X[-1],X[-l])\otimes_H\T(X,X[-1])\ar[r]& \T(X,X[-l]) } \]
	is an isomorphism for all $l\geq1$. One can check that its dual is isomorphic to the composite of the following natural maps.
	\[	\begin{aligned}
		D(\T(X[-1],X[-l])\otimes_H\T(X,X[-1]))&=\Hom_H(\T(X[-1],X[-l]),D\T(X,X[-1]))\\
		&=\Hom_H(\T(X[-1],X[-l]),\T(X[-1],X[d]))\\
		&=\Hom_H(\T(X,X[-l+1]),\T(X,X[d+1]))\\
		&\xleftarrow{F}\T(X[-l+1],X[d+1])=\T(X[-l],X[d])\\
		&=D\T(X,X[-l])
	\end{aligned}	\]
	Therefore it remains to show that $F=\T(X,-)$ induces bijections
	\[ \xymatrix{ \T(T[-l],X[d+1])\ar[r]& \Hom_H(FT[-l],FX[d+1]) } \]
	for all $l\geq0$. By \ref{full}, we know that this is surjective. We claim that this is injective. When $0\leq l\leq d-2$ then $T[-l]\in\add X$ and the assertion is clear. We assume $l\geq d-1$, so that $T[-l]$ do not share a direct summand with $X\oplus T[1]$ by \ref{non} and we can apply \ref{ses}. Let $f\colon T[-l]\to X[d+1]$ be a morphism which is $0$ under $F$. Taking a triangle $X_1\to X_0\to T[-l]\to X_1[1]$ in \ref{last}, we see that $f$ is mapped to $0$ under the last map in the exact sequence below.
	\[ \xymatrix@R=1mm{ 
		\T(X_0,X[d])\ar[r]^-a&\T(X_1,X[d])\ar[r]&\T(T[-l],X[d+1])\ar[r]& \T(X_0,X[d+1]) \\
		&&\hspace{15mm}f\hspace{15mm}\ar@{|->}[r]&\hspace{15mm}0\hspace{15mm} } \]
	Also the first map $a$ is dual to $\T(X,X_1)\to\T(X,X_0)$, which is injective by \ref{ses}(\ref{F}). We conclude that $a$ is surjective, hence $f=0$ as desired. 
\end{proof}

We next give the following property of the bimodule $U$, which shows that the bimodule $U$ gives a $(d-1)$-st root of the AR translation.
\begin{Prop}\label{AR}
	We have isomorphisms
	\[ U\lotimes_H\cdots\lotimes_HU\simeq U\otimes_H\cdots\otimes_HU\simeq\Ext_H^1(DH,H) \]
	in the derived category of $(H,H)$-bimodules, where the tensor factor is $(d-1)$-times. Therefore the functor $-\lotimes_HU$ gives an autoequivalence on $\D^b(\md H)$ whose $(d-1)$-st power is the AR translation.
\end{Prop}
\begin{proof}
	Since $U$ is a preprojective $H$-module by \ref{st}(2) we have the first isomorphism. We prove the second one. By \ref{alg} we have to show $\T(X,X[-d+1])\simeq\Ext_H^1(DH,H)$. It follows from the dual of \ref{st}(2) that $\T(X,X[-d+1])\simeq\tau^{-1}H=\Ext^1_H(DH,H)$ in $\md H$. %Since $H$ is $1$-representation infinite, the right $H$-module $\T(X,X[-d+1])$ does not have a projective summand by \ref{non} and \ref{inddet}(\ref{pj}), and the same holds for $\tau^{-1}H=\Ext^1_H(DH,H)$, so we have $\T(X,X[-d+1])\simeq\Ext_H^1(DH,H)$ in $\md H$.
	Naturality of this isomorphism shows that it is compatible with left $H$-actions, that is, $\T(X,X[-d+1])\simeq\Ext^1_H(DH,H)$ as $(H,H)$-bimodules.
\end{proof}

Now assume that the algebra $H/J_H$ is separable over $k$, which is the case if $H$ is the path algebra of an acyclic quiver.
\begin{Lem}\label{dim2}
The graded algebra $S$ is homologically smooth of dimension $\leq2$, that is, $\pd_{S^e}S\leq2$.
\end{Lem}
\begin{proof}
	Since $H$ is $1$-representation infinite and $U$ is a preprojective module such that $-\lotimes_HU$ gives an autoequivalence on $\D^b(\md H)$, we see that the derived tensor algebra $T^{\mathrm{L}}_HU:=T_H\Psi$, where $\Psi\to U$ is a bimodule projective resolution, has its cohomology concentrated in degree $0$, where it is $S$. This shows that there is a triangle
	\[ \xymatrix{ S\lotimes_HU\lotimes_HS\ar[r]& S\lotimes_HS\ar[r]& S\ar[r]& } \]
	in $\D(\Md S^e)$. Applying $\RHom_{S^e}(-,S^e)$	we obtain a triangle
	\[ \xymatrix{ S^e\lotimes_{H^e}\RHom_{H^e}(H,H^e)[1]\ar[r]& S^e\lotimes_{H^e}\RHom_{H^e}(U,H^e)[1]\ar[r]&\RHom_{S^e}(S,S^e)[2]\ar[r]& }, \]
	cf. \cite[4.8]{Ke11}. Now the first term is concentrated in degree $\leq0$ by $\pd_{H^e}H\leq1$, and so is the second term since $\RHom_{H^e}(U,H^e)[1]=\RHom_{H^\op\otimes H}(U,\RHom_k(DH,H))=\RHom_{H^\op}(U\lotimes_HDH,H)$, which equals $(U\lotimes_HU^{\lotimes_H-(d-1)})^{-1}=U^{\lotimes_H(d-2)}$ by \ref{AR}. We conclude that the third term also lies in degree $\leq0$, hence $\pd_{S^e}S\leq2$.
\end{proof}

Recall that a dg algebra is {\it formal} if it is isomorphic to its cohomology in the homotopy category of dg categories. A graded algebra $\L$ is {\it intrinsically formal} if any dg algebra with cohomology $\L$ is formal. We use the following criterion for intrinsic formality using Hochschild cohomology. We denote by $[1]$ the degree shift of graded vector spaces. For a graded bimodule $M$ over $\L$, the shifted bimodule $M[1]$ has $\L$-actions $a\cdot x\cdot b=(-1)^{\deg a}axb$.
\begin{Prop}[{\cite{Kad}, see also \cite[4.7]{ST}\cite[1.7]{RW}}]\label{formal}
Let $\L$ be a graded $k$-algebra such that $\Ext^i_{\L^e}(\L,\L[2-i])=0$ for all $i>2$. Then $\L$ is intrinsically formal. In particular if $\pd_{\L^e}\L\leq2$ then $\L$ is intrinsically formal.
\end{Prop}

We are now ready to prove our main theorem of this paper which gives a Morita-type result for cluster categories arising from hereditary algebras. 
\begin{Thm}\label{root}
Let $\T$ be an algebraic $d$-CY triangulated category with a $d$-cluster tilting object $T$. We put $X=T\oplus T[-1]\oplus\cdots\oplus T[-(d-2)]$. Suppose that $H=\End_\T(X)$ is $1$-representation infinite and $H/J_H$ is separable over $k$. Set $U=\T(X,X[-1])$, which we view as an $(H,H)$-bimodule.
\begin{enumerate}
	\item\label{U} There exists an isomorphism $U^{\lotimes_H(d-1)}\simeq\Ext^1_H(DH,H)$ of $(H,H)$-bimodules.
	\item\label{T} There exists a triangle equivalence $\T\simeq\D^b(\md H)/-\lotimes_HU[1]$.
\end{enumerate}
Therefore we can write $\T\simeq\D^b(\md H)/\tau^{-1/(d-1)}[1]$ for $\tau^{-1/(d-1)}:=-\lotimes_HU$.
\end{Thm}
\begin{proof}%\footnote{It seems uniqueness of enhancement follow from this.}
	We have seen (\ref{U}) in \ref{AR}, so we prove (\ref{T}). Let $\G$ be a derived endomorphism ring of $X$ in some dg enhancement of $\T$ and set $\Pi=\G^{\leq0}$. By \ref{alg} we know the cohomology of $\Pi$ is isomorphic to the tensor algebra $S=T_HU$. By \ref{dim2} and \ref{formal}, it is intrinsically formal, thus $\Pi$ is quasi-isomorphic to $S$ viewed as a dg algebra with $\deg U=-1$ and zero differential; $\Pi=T_H(U[1])$. It follows from the first triangle in the proof of \ref{dim2} that $\Pi$ is homologically smooth. Now, $X$ has a $d$-cluster tilting object $T$ as a direct summand so we can apply \ref{str}, and therefore there exists a triangle equivalence $\T=\per\G\simeq\C(\Pi)$. By \ref{cl} the cluster category $\C(\Pi)$ is equivalent to $\D^b(\md H)/-\lotimes_HU[1]$.
\end{proof}

\subsection{Corollaries}\label{Cor}
Some of the $(d-1)$-st root $\tau^{1/(d-1)}$ stated \ref{root} arises in a somewhat absurd way, for example, $H$ is a direct product of $(d-1)$ copies of an algebra $H^\prime$, and $\tau^{1/(d-1)}$ on $\D^b(\md H)=\D^b(\md H^\prime)\times\cdots\times\D^b(\md H^\prime)$ is given by $(L_1,\ldots,L_{d-1})\mapsto(L_2,\ldots,L_{d-1},\tau^{\prime}L_1)$ using the AR translation $\tau^{\prime}$ for $H^\prime$.
The proof of corollaries consists of such an interpretation of $\tau^{1/(d-1)}$. Together with an observation in \ref{adj} below allows us to rewrite the orbit category $\D^b(\md H)/\tau^{-1/(d-1)}[1]$ in terms of $\D^b(\md H^\prime)$ and ${\tau^\prime}^{-1}$.

Now we look at some consequences of our main theorem \ref{root}. When $d=2$, this immediately reduces to Keller--Reiten's recognition theorem for non-Dynkin quivers. More generally, using an interpretation of the $(d-1)$-st root $\tau^{-1/(d-1)}$ as above, we can recover its generalization to higher dimension.
\begin{Cor}[cf. \cite{KRac}]\label{kr}
Let $\T$ be an algebraic $d$-CY triangulated category with a $d$-cluster tilting object $T$. Suppose that $H^\prime=\End_\T(T)$ is $1$-representation infinite, $H^\prime/J_{H^\prime}$ is separable over $k$, and that $\Hom_\T(T,T[-i])=0$ for $0<i<d/2$. Then there exists a triangle equivalence $\T\simeq\D^b(\md H^\prime)/\tau^{-1}[d-1]$.
\end{Cor}
\begin{proof}
	By \ref{half} we have $\T(T,T[-i])=0$ for all $1\leq i\leq d-2$. Then we have the following forms of $H=\End_\T(X)$ and the bimodule $U=\T(X,X[-1])$:
	\[ H=H^\prime\times\cdots\times H^\prime,\qquad U=
	\left( 
	\begin{array}{cccc}
		0&H^\prime&\cdots&0\\
		\vdots&\vdots&\ddots&\vdots\\
		0&0&\cdots&H^\prime\\
		V&0&\cdots&0
	\end{array}
	\right), \]
	with $V=\T(T,T[-(d-1)])$. Note that by the above descriptions we have $U^{\otimes_{H}(d-1)}=V\times\cdots\times V$, hence using \ref{AR}, we see $V\simeq\Ext_{H^\prime}^1(DH^\prime,H^\prime)$ as $(H^\prime,H^\prime)$-bimodules.
	Now since $H=H^\prime\times\cdots\times H^\prime$ is $1$-representation infinite and $H/J_H$ is separable over $k$, we can apply \ref{root}, so our triangulated category $\T$ is triangle equivalent to $\D^b(\md H)/-\lotimes_HU[1]$. By \ref{adj} this is triangle equivalent to $\D^b(\md H^\prime)/-\lotimes_{H^\prime}V[d-1]$, and we deduce by $V\simeq\Ext^1_{H^\prime}(DH^\prime,H^\prime)$ that this is precisely the $d$-cluster category of $H^\prime$.
\end{proof}

The next case $d=3$ generalizes a theorem of Keller--Murfet--Van den Bergh as well as giving a $3$-CY version of Keller--Reiten's theorem above. Again with a suitable vanishing conditions and an interpretation of $(d-1)$-st root, we have the following structure theorem for CY categories of dimension $d=2n+1$. In this case the $2n$-th root $\tau^{1/2n}$ can generally be reduced to a square root.%It is notable that, as state in (\ref{H}), the algebra $\End_\T(X)$ is hereditary as soon as $\End_\T(T)$ is.
\begin{Cor}\label{kmv}
Let $\T$ be an algebraic $(2n+1)$-CY triangulated category with a $(2n+1)$-cluster tilting object $T$ such that $\End_\T(T)$ is hereditary, and $\T(T,T[-i])=0$ for $0<i<n$. Then, the algebra $H^\prime=\End_\T(T\oplus T[-n])$ is hereditary, and there exists a triangle equivalence $\T\simeq\D^b(\md H^\prime)/\tau^{-1/2}[n]$ if $H^\prime$ is $1$-representation infinite and $H^\prime/J_{H^\prime}$ is separable over $k$.
\end{Cor}
\begin{proof}
	Recall that $X=T\oplus\cdots\oplus T[-(2n-1)])$ and put $Y=T\oplus T[-n]$. Then by \ref{oddcy} the algebra $H^\prime=\End_\T(Y)$ is hereditary, and by \ref{nex} the algebra $H=\End_\T(X)$ and the bimodule $U=\T(X,X[-1])$ has the following form along the decomposition $X=Y\oplus\cdots\oplus Y[-(n-1)]$.
	\[ H=H^\prime\times\cdots\times H^\prime,\qquad U=
	\left( 
	\begin{array}{cccc}
		0&H^\prime&\cdots&0\\
		\vdots&\vdots&\ddots&\vdots\\
		0&0&\cdots&H^\prime\\
		V&0&\cdots&0
	\end{array}
	\right), \]
	where the product has $n$ factors, the matrix is $n\times n$, and $V=\T(Y,Y[-n])$. Now the above descriptions show $U^{\otimes_H(2n)}=V^{\otimes_{H^\prime}2}\times\cdots\times V^{\otimes_{H^\prime}2}$, hence by \ref{AR} we deduce $V\otimes_{H^\prime}V\simeq\Ext^1_{H^\prime}(DH^\prime,H^\prime)$ as $(H^\prime,H^\prime)$-bimodules so that $-\lotimes_{H^\prime}V$ can be regarded as a square root $\tau^{-1/2}$ of the AR translation on $\D^b(\md H^\prime)$.
	When $H^\prime$ is $1$-representation infinite and $H^\prime/J_{H^\prime}$ is separable over $k$, so is $H$, so we can apply \ref{root} to obtain a triangle equivalence $\T\simeq\D^b(\md H)/-\lotimes_HU[1]$. By \ref{adj} this is equivalent to $\D^b(\md H^\prime)/-\lotimes_{H^\prime}V[n]$.
\end{proof}
The final case is $d=4$. Yet again, we state this more generally for even dimensional CY categories.
\begin{Cor}\label{new}
	Let $\T$ be an algebraic $(2n+2)$-CY triangulated category with a $(2n+2)$-cluster tilting object $T$ such that $\Hom_\T(T,T[-i])=0$ for $0<i<n$ and $\End_\T(T)=k\times\cdots\times k$. Then the algebra $H=\End_\T(T\oplus\cdots\oplus T[-2n])$ is hereditary and there exists a triangle equivalence $\T\simeq\D^b(\md H)/\tau^{-1/(2n+1)}[1]$.
\end{Cor}
We remark that there is no analogue of interpreting the $(d-1)$-st root in this situation; in general $\tau^{1/(d-1)}$ cannot be made easier. This is because $H$ is in general connected in this case, see \ref{star}.

\subsection{Uniqueness of enhancements}
We note that the proof of our main theorem \ref{root} actually gives the uniqueness of enhancements for such triangulated categories (\ref{enh}). Before that we give the following remark on connectivity of our triangulated category. Although it may happen that $H$ is not connected even if the triangulated category $\T$ is connected, we show that every connected component of the (valued) quiver of $H$ as the same underlying graph as long as $\T$ is connected.
\begin{Prop}\label{conn}
Let $\T$ be a $d$-CY triangulated category with a $d$-cluster tilting object $T$ such that $\End_\T(T\oplus T[-1]\oplus\cdots\oplus T[-(d-2)])=H$ is hereditary. Let $Q$ be the valued quiver of $H$.
\begin{enumerate}
	\item There is a free action of the cyclic group $G$ of order $d-1$ on the underlying valued graph $\underline{Q}$ of $Q$.
	\item The triangulated category $\T$ is connected if and only if the orbit graph $\underline{Q}/G$ is connected.
\end{enumerate}
%In particular, every connected component of $Q$ has the same underlying valued graph when $\T$ is connected.
\end{Prop}
\begin{proof}
	(1)  Without loss of generality we assume $T$ is basic, and identify each of the indecomposable direct summands of $X=T\oplus T[-1]\oplus\cdots\oplus T[-(d-2)]$ with the corresponding vertex of $Q$. Clearly $Q$ has $d-1$ copies of the quiver of $\End_\T(T)$ as a full subquiver. We consider the remaining arrows.
	
	We draw the AR $(d+2)$-angles in $\add T$ at each direct summand $T_a$ of $T$ as below.
	\[ \xymatrix@!C=1mm@!R=2mm{
		&T_a^{(d-1)}\ar[dr]\ar[rr]&&T_a^{(d-2)}\ar[dr]\ar[rr]&&\cdots\ar[dr]\ar[rr]&&T_a^{(1)}\ar[dr]\ar[rr]&&T_a^{(0)}\ar[dr]&\\
		T_a\ar[ur]&&\ar[ur]&&\ar[ur]&&\ar[ur]&&\ar[ur]&&T_a } \]
	Put $D_a=\End_\T(T_a)$, which is a division algebra. The valuation of the arrow $T_a\to T_b[-i]$ in $Q$ is given by the $(D_b,D_a)$-bimodule $J_X/J_X^2(T_a,T_b[-i])$, where $J_X$ is the Jacobson radical of the category $\add X$. Applying \ref{induction} to $m=i$, we see that this bimodule is isomorphic as a right $D_a$-vector space to $\T/J_\T(T_a,T_b^{(i)})$, whose dimension is nothing but the number of summands $T_a$ in $T_b^{(i)}$. Dually by \ref{induction2}, we have an isomorphism $J_X/J_X^2(T_a,T_b[-i])\simeq\T/J_\T(T_a^{(d-1-i)},T_b)$ of left $D_b$-vector spaces, whose dimension is the number of summands $T_b$ in $T_a^{(d-1-i)}$. We deduce that the valuation of the arrow ${T_a\rightarrow T_b[-i]}$ is equal to that of ${T_a[-(d-1-i)]\leftarrow T_b}$ for $1\leq i\leq d-2$. It easily follows that the map on the vertices of $Q$ given by $T_a[-i]\mapsto T_a[-(i+1)]$ and $T_a[-(d-2)]\mapsto T_a$ extends to an automorphism of $\underline{Q}$.
	
	(2)  It is clear that $\underline{Q}/G$ is not connected when $\T$ is not connected. Suppose conversely that $\underline{Q}/G$ is not connected. This amounts to saying that there exists a decomposition $T=T_1\oplus T_2$ such that $\T(X_1,X_2)=0=\T(X_2,X_1)$, where $X_i=T_i\oplus T_i[-1]\oplus\cdots\oplus T_i[-(d-2)]$ for $i=1,2$. We claim $\T=\thick T_1\times\thick T_2$. 
	
	{\it Step 1: $\T(T_1,T_2[-(d-1)])=0=\T(T_2,T_1[-(d-1)])$.}
	
	We prove the first equality. Consider the AR $(d+2)$-angle in $\add T$ at $T_2$. Applying \ref{induction} to $m=d-2$, we have that $\coprod_{i=0}^{d-2}T_2^{(i)}[-(d-2-i)]\to T_2[-(d-2)]$ is the sink map in $\add(T\oplus T[-1]\oplus\cdots\oplus T[-(d-2)])$. It follows from $\T(T_1,X_2)=0$ that $T_2^{(i)}\in\add T_2$ for $0\leq i\leq d-2$. On the other hand, the leftmost map $T_2\to T_2^{(d-1)}$ is the source map in $\add T$, so we also have $T_2^{(d-1)}\in\add T_2$. Then applying \ref{induction} to $m=d-1$, we obtain the sink map $\coprod_{i=0}^{d-1}T_2^{(i)}[-(d-1-i)]\to T_2[-(d-1)]$ in $\add(T\oplus T[-1]\oplus\cdots\oplus T[-(d-1)])$. Now there are no non-zero homomorphism from $T_1$ to all the summands in the left-hand-side except $\T(T_1,T_2^{(0)}[-(d-1)])$, so we have a surjection $\T(T_1,T_2^{(0)}[-(d-1)])\to\T(T_1,T_2[-(d-1)])$. It follows by induction on the vertices of the quiver of $T_2$ that $\T(T_1,T_2[-(d-1)])=0$.
	
	{\it Step 2: $\T(T_1,T_2[l])=\T(T_2,T_1[l])=0$ for all $l\in\Z$.}
	
	Since $T_1\oplus T_2\in\T$ is $d$-cluster tilting, we have the assertion for $0<l<d$. By the $d$-CY property of $\T$ it is enough to prove the claim for $l\leq0$. We show $\T(T_1,T_2[-i])=0$ for $i\geq0$ by induction on $i$. By assumption we know this for $0\leq i\leq d-2$, and for $i=d-1$ by the previous step. 
	Now observe that {\it $T_1\oplus T_2[-1]$ is a $d$-cluster tilting object with the same properties}, that is, $\T(X_1,X_2[-1])=0=\T(X_2[-1],X_1)$. Indeed, this is a $d$-cluster tilting object since it is the right mutation of $T_1\oplus T_2$ at $T_2$ (\cite[5.1]{IYo}). Also, the first vanishing amounts to saying $\T(T_1,T_2[i])=0$ for $-(d-1)\leq i\leq d-3$, which follows from Step 1, and the second one to $\T(T_2,T_1[i])=0$ for $-(d-3)\leq i\leq d-1$, which is clear. Then applying Step 1 to the $d$-cluster tilting object $T_1\oplus T_2[-1]$ we deduce $\T(T_1,T_2[-d])=0$. Similarly applying it to $T_1[-1]\oplus T_2$ yields $\T(T_2,T_1[-d])=0$. We will inductively have the assertion.
	
	{\it Step 3: The conclusion.}
	
	By Step 2, we have $\T\supset\thick T_1\times\thick T_2$. It remains to prove that any indecomposable object lies either in $\thick T_1$ or $\thick T_2$. Let $A\in\T$ be an indecomposable object. If $A\in\add(X\oplus T[1])$ then the assertion is clear. If $A\not\in\add(X\oplus T[1])$, then by \ref{ses} there is a triangle $X^{-1}\to X^{0}\to A\to X^{-1}[1]$ in $\T$ with $X^{-1},X^{0}\in\add X$ such that $0\to\T(X,X^{-1})\to\T(X,X^0)\to\T(X,A)\to0$ is a minimal projective resolution in $\md H$. Since $H=\End_\T(X_1)\times\End_\T(X_2)$, both $X^{-1}$ and $X^0$ has to be in $\add X_i$ for the same $i$. We deduce that $A\in X_i\ast X_i[1]\subset \thick T_i$.
\end{proof}
Recall that an {\it enhancement} of a triangulated category $\T$ is a dg category $\A$ together with a triangle equivalence $\per\A\simeq\T$. We say two enhancements $\A$ and $\B$ are {\it equivalent} if there is $X\in\D(\A^\op\otimes\B)$ such that $-\lotimes_\A X\colon\per\A\to\per\B$ is an equivalence. A triangulated category {\it has a unique enhancement} if any of its two enhancements are equivalent.
\begin{Thm}\label{enh}
Let $k$ be a perfect field, $d\geq2$ and $\T$ an algebraic $d$-CY triangulated category with a $d$-cluster tilting object $T$ such that $H=\End_\T(T\oplus T[-1]\oplus\cdots\oplus T[-(d-2)])$ is hereditary. Then $\T$ has a unique enhancement.
\end{Thm}
\begin{proof}
	We may assume that $\T$ is connected as a triangulated category. Then by \ref{conn} above, the hereditary algebra $H$ is either representation-finite or $1$-representation infinite. If $H$ is representation-finite, then $\T$ has finitely many indecomposable objects by \ref{st}. The uniqueness of enhancements for such triangulated categories is established in \cite{Mu20}. Now assume that $H$ is $1$-representation infinite. Let $\A$ be an enhancement of $\T$, and $\G=\A(X,X)$ with $X=T\oplus T[-1]\oplus\cdots\oplus T[-(d-2)]$. Clearly $\A$ and $\G$ are derived Morita equivalent. By the proof of \ref{root} the dg algebra $\Pi=\G^{\leq0}$ is independent of the enhancement $\A$ (by the intrinsic formality of its cohomology), so it is enough to prove that $\Pi$ determines $\G$ up to derived Morita equivalence. But this follows from \ref{loc}(2), as it says $\D(\G)$ is categorically characterized in $\D(\Pi)$.
\end{proof}
	
\section{Adjusting orbits}
We observe that certain interpretation of $(d-1)$-st root of the AR translation as in Section \ref{Cor} gives an equivalence of orbit categories, leading to the proof of corollaries. The aim of this section is to prove such equivalences in the level of dg enhancements. The content here is in this way a refinement of \cite[Appendix A]{ha3}.

Let $\A$ be a dg category and let $F\colon\A\to\A$ be a dg functor. Then we can define the dg orbit category $\A/F$ \cite{Ke05}; it has the same objects as $\A$ and the morphism complex
\[ \A/F(L,M)=\colim\left( \xymatrix{\discoprod_{k\geq0}\A(F^kL,M)\ar[r]&\discoprod_{k\geq0}\A(F^kL,FM)\ar[r]&\cdots} \right).  \]
%and define the dg functor $G\colon\B\to\B$ by $(L_1,\ldots,L_n)\mapsto(FL_n,L_1,\ldots,L_{n-1})$.
\begin{Lem}\label{ff}
	Let $\A$ be a dg category and $\B=\A\times\cdots\times\A$ the $n$-fold product of $\A$. Suppose that $F\colon\A\to\A$ and $G\colon\B\to\B$ are dg functors such that $G^n=F\times\cdots\times F$ on $\B$ and $\B(G^i(L,0,\ldots,0),G^j(M,0,\ldots,0))=0$ unless $n\mid i-j$. Then the functor $\A\to\B$ given by $L\mapsto(L,0,\ldots,0)$ induces a fully faithful functor $\A/F\to\B/G$.
	%$G^n(L_1,\ldots,L_n)=(FL_1,\ldots,FL_n)$
\end{Lem}
\begin{proof}
	Writing $\widetilde{L}=(L,0,\ldots,0)$, we have to show $\B/G(\widetilde{L},\widetilde{M})=\A/F(L,M)$ for each $L,M\in\A$. The left-hand-side is
	\[ \colim\left( \xymatrix{\discoprod_{k\geq0}\B(G^k\widetilde{L},\widetilde{M})\ar[r]^-G&\discoprod_{k\geq0}\B(G^k\widetilde{L},G\widetilde{M})\ar[r]^-G&\discoprod_{k\geq0}\B(G^k\widetilde{L},G^2\widetilde{M})\ar[r]&\cdots} \right), \]
	which equals
	\[ \colim\left( \xymatrix{\discoprod_{k\geq0}\B(G^k\widetilde{L},\widetilde{M})\ar[r]^-{G^n}&\discoprod_{k\geq0}\B(G^k\widetilde{L},G^n\widetilde{M})\ar[r]^-{G^n}&\discoprod_{k\geq0}\B(G^k\widetilde{L},G^{2n}\widetilde{M})\ar[r]&\cdots} \right) \]
	by cofinality. Moreover, by vanishing of $\B(G^i\widetilde{L},G^j\widetilde{M})$ for $n\!\mathrel{\not|} i-j$, this is equal to
	\[ \colim\left( \xymatrix{\discoprod_{k\geq0}\B(G^{nk}\widetilde{L},\widetilde{M})\ar[r]^-{G^n}&\discoprod_{k\geq0}\B(G^{nk}\widetilde{L},G^n\widetilde{M})\ar[r]^-{G^n}&\discoprod_{k\geq0}\B(G^{nk}\widetilde{L},G^{2n}\widetilde{M})\ar[r]&\cdots} \right), \]
	hence to
	\[ \colim\left( \xymatrix{\discoprod_{k\geq0}\A(F^k{L},{M})\ar[r]^-{F}&\discoprod_{k\geq0}\A(F^k{L},F{M})\ar[r]^-{F}&\discoprod_{k\geq0}\A(F^k{L},F^{2}{M})\ar[r]&\cdots} \right), \]
	which is nothing but $\A/F(L,M)$.
\end{proof}

Let us note the following consequence in which form we use.
\begin{Prop}\label{adj}
	Let $\A$ be a pretriangulated dg category and $F\colon\A\to\A$ a dg functor inducing an equivalence on $H^0\A$. Let $\B=\A\times\cdots\times\A$ the $n$-fold product and define $G\colon\B\to\B$ by $(L_1,\ldots,L_n)\mapsto(FL_n,L_1,\ldots,L_{n-1})$. Then $L\mapsto(L,0,\ldots,0)$ gives a quasi-equivalence $\A/F[n]\to\B/G[1]$.
\end{Prop}
\begin{proof}
	Since $(G[1])^n=F[n]\times\cdots\times F[n]$, we can apply \ref{ff} so that we have a fully faithful functor $\A/F[n]\to\B/G[1]$. It remains to show that the induced functor $H^0(\A/F[n])\to H^0(\B/G[1])$ is dense. This follows from the fact that for each $L\in\A$ and $1\leq i\leq n$, the object $L[-i+1]\in H^0(\A/F[n])$ is mapped to $(L[-i+1],0,\ldots,0)$, which is isomorphic in $H^0(\B/G[1])$ to $(G[1])^{i-1}(L[-i+1],0,\ldots,0)=(0,\ldots,L,\ldots,0)$ with $L$ at the $i$-th factor.
\end{proof}

Applying \ref{adj} to finite dimensional algebras, we can realize the $(d+n)$-cluster category of a finite dimensional algebra $A$ of global dimension $\leq d$ as the canonical triangulated hull of the orbit category of $n$-fold product algebra.
\begin{Ex}
Let $A$ be a finite dimensional algebra of global dimension $\leq d$, and $\theta_d\to \RHom_A(DA,A)[d]$ the bimodule projective resolution. Consider the pretriangulated dg category and its dg endofunctor given by
\[ \A=\C^b(\proj A),\quad F=-\otimes_A\theta_d\colon\A\to\A. \]
Then, letting $\mathbf{\Pi}_{d+n+1}(A):=T_A(\theta_d[n])$ the $(d+n+1)$-CY completion of $A$, we have an equivalence
\[ \xymatrix{ \per(\A/F[n])\ar[r]^-\simeq& \C(\mathbf{\Pi}_{d+n+1}(A)) }, \]
cf. \ref{cl}. Let us now consider the dg orbit category $\B/G[1]$ in \ref{adj}. Putting $B=A\times\cdots\times A$, we have
\[ \B=\A\times\cdots\times\A\simeq\C^b(\proj B), \quad G=-\otimes_BU\colon\B\to\B \text{ with } U=\begin{pmatrix}0&A&\cdots&0\\\vdots&\vdots&\ddots&\vdots\\0&0&\cdots&A\\\theta_d&0&\cdots&0\end{pmatrix}. \]
Then, define $\mathbf{S}=T_B(U[1])$ so that \ref{cl} gives an equivalence
\[ \xymatrix{ \per(\B/G[1])\ar[r]^-\simeq& \C(\mathbf{S}). } \]
We conclude by \ref{adj} that there exists an equivalence
\[ \xymatrix{ \C(\mathbf{S})\ar[r]^-\simeq&\C(\mathbf{\Pi}_{d+n+1}).} \]
\end{Ex}

\section{Combinatorial roots of $\tau$}\label{comb}
We give a description of quivers whose derived categories have some roots of the AR translation. Observing that a root of $\tau$ in the derived category of a quiver $Q$ gives rise to a root of $\tau$ as an automorphism of the infinite translation quiver $\Z Q$, we give a necessarily and sufficient condition for $Q$ to have such a root on $\Z Q$.
 
Throughout this section let $Q$ be a finite acyclic quiver. Note that we do not assume it is connected. Suppose that there exists an autoequivalence $F$ of $\D^b(\md kQ)$ such that $F^l\simeq\tau^{-1}$ for some $l\geq1$. %Then it gives an automorphism, again denoted by $F$, of the AR quiver of $\D^b(\md kQ)$ such that $F^l=\tau^{-1}$.
We first translate this categorical autoequivalence into a combinatorial one on the infinite translation quiver $\Z Q$ \cite{ASS,Hap}. We say that two quivers are {\it derived equivalent} if their path algebras are derived equivalent. This is the case if and only if their infinite translation quivers are isomorphic.
\begin{Lem}
Let $F$ be a triangle autoequivalence of $\D^b(\md kQ)$ such that $F^l\simeq\tau^{-1}$. Then $F$ induces an automorphism $F^\prime$ of the translation quiver $\Z Q$ such that $(F^\prime)^l=\tau^{-1}$.
\end{Lem}
\begin{proof}
	The proof depends on the description of AR components of $\D^b(\md kQ)$ (see \cite[I.5.6]{Hap}).
	Divide the components of $Q$ into the derived equivalence classes $Q_1\sqcup\cdots\sqcup Q_n$. Then $F$ acts on each derived category $\D^b(\md kQ_i)$ so we may assume that the components of $Q$ are mutually derived equivalent. Suppose first that each component of $Q$ is Dynkin. Then the AR quiver of $\D^b(\md kQ)$ is $\Z Q$, thus $F$ gives a desired autoequivalence. We now suppose that $Q$ is non-Dynkin and decompose it into the components $Q_1\sqcup\cdots\sqcup Q_n$. Then the AR quiver of $\D^b(\md kQ)$ contains the components $\C_p(Q_i)\simeq\Z Q_i$ each of which is characterized as the component containing $kQ_i[p]\in\D^b(\md kQ)$. In view of the shape of the AR quiver of $\D^b(\md kQ)$, the triangle autoequivalence $F$ induces an automorphism of the translation quiver $\bigcup_{p\in\Z, 1\leq i\leq n}\C_p(Q_i)$. Then there exists a permutation $\sigma$ on $\{1,\ldots,n\}$ and $p_i\in\Z$ such that the component $\C_0(Q_i)$ is mapped to $\C_{p_i}(Q_{\sigma(i)})$. Let $G$ be the triangle automorphism of $\D^b(\md kQ)=\D^b(\md kQ_1)\times\cdots\times\D^b(\md kQ_n)$ given by $(L_1,\ldots,L_n)\mapsto(L_1[-p_{\s^{-1}(1)}],\ldots,L_n[-p_{\s^{-1}(n)}])$. Then the composite $F^\prime:=G\circ F$ preserves $\C_0(Q_1)\sqcup\cdots\sqcup\C_0(Q_n)\simeq\Z Q$, whose $l$-th power equals $\tau^{-1}$.
\end{proof}
 
Recall that a {\it section} of $\Z Q$ \cite[VIII.1.2]{ASS}, (also \cite[7]{ABS2}) is a full subquiver $\Sigma$ of $\Z Q$ such that 
\begin{itemize}
	\item each $\tau$-orbit of $\Z Q$ intersects $\Sigma$ exactly once.
	\item If $x\to y$ is an arrow in $\Z Q$ and $x\in\Sigma$, then $y\in\Sigma$ or $\tau y\in\Sigma$.
\end{itemize}
If $\Sigma$ is a section of $\Z Q$ then there is an isomorphism $\Z\Sigma\simeq\Z Q$ of translation quivers \cite[VIII.1.6]{ASS}. The following main result in this section shows that we can take a section as an orbit under the root of $\tau$.
\begin{Thm}\label{sec}
Let $Q$ be a finite acyclic quiver and $F$ an automorphism of $\Z Q$ such that $F^l=\tau^{-1}$. Then there exists a full subquiver $T$ of $\Z Q$ such that the full subquiver consisting of $T\cup FT\cup\cdots\cup F^{l-1}T$ forms a section of $\Z Q$. Moreover, for any points $s,t\in T$ and $a>0$, there is no arrow from $F^as$ to $t$.
\end{Thm}
Consequently, we can characterize the quivers whose infinite translation quiver $\Z Q$ has an $l$-th root of $\tau$. 
\begin{Cor}\label{class}
	There exists an $l$-th root of $\tau^{-1}$ on $\Z Q$ if and only if $Q$ is derived equivalent to the quiver $Q^\prime$ satisfying the following (a), (b) and (c).
	\begin{enumerate}
		\renewcommand{\labelenumi}{(\alph{enumi})}
		\item $Q^\prime$ has $l$ copies $T^{(0)}, T^{(1)},\ldots,T^{(l-1)}$ of a quiver $T$ as a full subquiver.
		\suspend{enumerate}
		We denote by $F$ the permutation of the vertices of $Q^\prime$ taking $T^{(i)}$ to $T^{(i+1)}$, where $T^{(l)}:=T^{(0)}$.
		\resume{enumerate}
		\renewcommand{\labelenumi}{(\alph{enumi})}
		\item There are additional arrows $x\to y$ in $Q^\prime$ with $x\in T^{(i)}$, $y\in T^{(j)}$ only if $i<j$.
		\item $F$ extends to an automorphism of the underlying graph of $Q^\prime$.
	\end{enumerate}
\end{Cor}
For the proof let us introduce the notion of ``a section with respect to a root of $\tau$''.
\begin{Def}
Let $F$ be an automorphism of $\Z Q$ such that $F^l=\tau^{-1}$. An {\it $F$-section} of $\Z Q$ is a full subquiver $T$ of $\Z Q$ satisfying the following.
\begin{enumerate}
\renewcommand{\labelenumi}{(\alph{enumi})}
	\item Each $F$-orbit of $\Z Q$ intersects $T$ exactly once.
	\item If $t\to x$ is an arrow in $\Z Q$ and $t\in T$ then $x\in T\cup FT\cup\cdots\cup F^{l-1}T$ or $\tau x\in T$.
\end{enumerate}
\end{Def}
Clearly the condition (b) is equivalent to its dual, which have the following form.
\begin{itemize}
	\item[(b')] If $x\to t$ is an arrow in $\Z Q$ and $t\in T$ then $x\in T$ or $\tau^{-1}x\in T\cup FT\cup\cdots\cup F^{l-1}T$.
\end{itemize}

The following observation shows that an $F$-section gives a desired ($\tau$-)section of $\Z Q$.
\begin{Lem}\label{Fsec}
Let $T$ be an $F$-section of $\Z Q$. Then $\Sigma=T\cup FT\cup\cdots\cup F^{l-1}T$ is a section of $\Z Q$.
\end{Lem}
\begin{proof}
	It is easily seen that $\Sigma$ forms a complete set of representatives of the $\tau$-orbits of $\Z Q$.	Let $x\in\Sigma$ and $x\to y$ an arrow in $\Z Q$. We have to show $y\in\Sigma$ or $\tau y\in\Sigma$, that is, $y\in\bigcup_{i=0}^{2l-1}F^iT$. Since $x\in\Sigma$ we have $x\in F^aT$ for some $0\leq a\leq l-1$. Then $y\in F^a\Sigma$ or $\tau y\in F^aT$ since $T$ is an $F$-section. It follows that $y\in\bigcup_{i=a}^{a+l}F^iT\subset\bigcup_{i=0}^{2l-1}F^iT$.
\end{proof}

\begin{Con}\label{constr}
We construct an $F$-section as a subquiver of the standard section $Q$ of $\Z Q$. Let $T$ be the subquiver of $Q$ such that each vertex $x$ of $Q$ can be written uniquely as $x=F^at$ for some $t\in T$ and $a\geq0$. Thus $T$ is obtained by dividing the vertices of $Q$ into $F$-orbits, and taking from each orbit the vertices which are written as the smallest powers of $F$.
\end{Con}

\begin{Ex}\label{upsidedown}
Let $Q$ be the quiver of linearly oriented type $A_4$. Then the infinite translation quiver $\Z Q$ has a square root of $\tau^{-1}$, indeed, consider the automorphism $F$ which ``turns up side down'', moving slightly to the right.
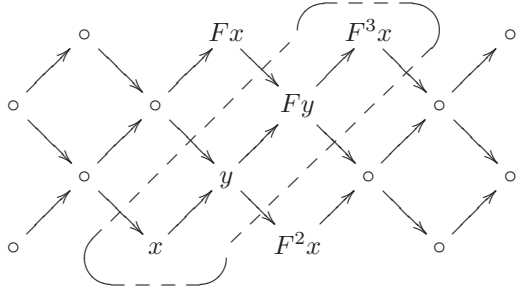
\[ \xymatrix@!R=1mm@!C=1mm{
	&\circ\ar[dr]&&Fx\ar[dr]&\ar@{--}`u[r]`[rr]`_dl[dddl][dddl]&F^3x\ar[dr]&&\circ\\
	\circ\ar[dr]\ar[ur]&&\circ\ar[dr]\ar[ur]&&Fy\ar[dr]\ar[ur]&&\circ\ar[dr]\ar[ur]&\\
	&\circ\ar[dr]\ar[ur]&&y\ar[dr]\ar[ur]&&\circ\ar[dr]\ar[ur]&&\circ\\
	\circ\ar[ur]&&x\ar[ur]&\ar@{--}`d[l]`[ll]`_ur[uuur][uuur]&F^2x\ar[ur]&&\circ\ar[ur]& } \]
The standard section $Q$ which is depicted in the dotted line consists of two $F$-orbits; letting $x$ its source and $y$ the adjacent vertex, the other two vertices are $Fy$ and $F^3x$. Then the construction says $T=\{x,y\}$.
\end{Ex}
\begin{Prop}\label{TF}
The quiver $T$ given in \ref{constr} is an $F$-section of $\Z Q$.
\end{Prop}

To prove this we need some observations on the arrows in $Q$ with sources or targets in $T$.
\begin{Lem}\label{pred}
Let $x\to t$ be an arrow in $Q$ with $t\in T$. Then $x\in T$.
%$T$ is closed under predecessors in $Q$.
\end{Lem}
\begin{proof}
	Since $x\in Q$ we can write $x=F^as$ for some $s\in T$ and $a\geq0$. Then there is an arrow $s\to F^{-a}t$ with $s\in Q$. Since $Q$ is a section we have $F^{-a}t\in Q$ or $F^{-a-l}t=\tau F^{-a}t\in Q$, but by our construction of $T$, $F^pt\in Q$ for some $p\in\Z$ forces $p\geq0$, thus we must have $a=0$, hence $x=s\in T$.
\end{proof}
\begin{Lem}\label{succ}
Let $t\to x$ be an arrow in $Q$ with $t\in T$. Then $x\in T\cup FT\cup\cdots\cup F^{l-1}T$.
\end{Lem}
\begin{proof}
	Since $x\in Q$ we can write $x=F^as$ with $s\in T$ and $a\geq0$. We have to show $a\leq l-1$. Consider the arrow $F^{-a}t\to s$. Since $s\in Q$ and $Q$ is a section, we have $F^{-a}t\in Q$ or $\tau^{-1}F^{-a}t\in Q$. If $F^{-a}t\in Q$ then $F^{-a}t\in T$ by \ref{pred}, hence $a=0$. If $\tau^{-1}F^{-a}t$, which equals $F^{l-a}t$, is in $Q$, we must have $l-a\geq0$ by the construction of $T$. It remains to exclude $a=l$. In this case, we have that both $s$ and $x=F^as=\tau^{-1}s$ lies in $Q$, which is absurd.
\end{proof}

\begin{proof}[Proof of \ref{TF}]
	It is clear from the construction that each $F$-orbits in $\Z Q$ intersects $T$ exactly once. Let $t\to x$ be an arrow in $\Z Q$ with $t\in T$. Since $t\in Q$ and $Q$ is a section, we have $x\in Q$ or $\tau x\in Q$. If $x\in Q$, then $x\in T\cup FT\cup\cdots\cup F^{l-1}T$ by \ref{succ} . If $\tau x\in Q$, then $\tau x\in T$ by \ref{pred}.
\end{proof}
Now we are ready to prove the main results of this section.
\begin{proof}[Proof of \ref{sec}]
	The first assertion follows from \ref{TF} and \ref{Fsec}, so we prove the second one. Let $s,t\in T$ and $F^as\to t$ an arrow in $\Z Q$. We have to show $a\leq0$. Since $T$ is an $F$-section we have $F^as\in T$ or $\tau^{-1}F^as\in T\cup FT\cup\cdots\cup F^{l-1}T$. If $F^as\in T$ then $a=0$ since $T$ intersects each $F$-orbit only once. Similarly if $\tau^{-1}F^as\in T\cup FT\cup\cdots\cup F^{l-1}T$ then comparing the exponent of $F$, we must have $0\leq a+l\leq l-1$, thus $a\leq-1$.
\end{proof}

For a pair $x, y$ in a quiver we denote by $\{x\to y\}$ (resp. $\{x-y\}$) the set of arrows from $x$ to $y$ (resp. unoriented edges between $x$ and $y$).
\begin{proof}[Proof of \ref{class}]
	We first show the ``only if'' part. Suppose $\Z Q$ has an $l$-th root ${F}$ of $\tau^{-1}$. Then by \ref{sec} there exists a subquiver $T$ of $Q$ such that $Q^\prime:=T\cup FT\cup\cdots\cup F^{l-1}T$ is a section of $\Z Q$. We claim that this $Q^\prime$ has the desired properties. Letting $T^{(i)}$ be the full subquiver of $Q^\prime$ consisting of the vertices from $F^iT$ we have (a). Also the second assertion in \ref{sec} shows (b). Now we turn to (c). For each point $t\in T$ we write $t^{(i)}$ the corresponding point in $T^{(i)}$. We have to show $\sharp\{s^{(i)}-t^{(j)}\}=\sharp\{s^{(i+1)}-t^{(j+1)}\}$, where the superscripts are read modulo $l$. We may assume $i<j$, and the assertion is clear if $j<l-1$. If $j=l-1$, then we have $\sharp\{s^{(i)}-t^{(l-1)}\}=\sharp\{F^is-F^{l-1}t\}=\sharp\{F^{i+1}s-F^lt\}=\sharp\{t-F^{i+1}s\}$ since $F^l=\tau^{-1}$, and the last term equals $\sharp\{t-s^{(i+1)}\}$.
	
	We next show the ``if'' part. We may assume $Q=Q^\prime$, satisfying (a), (b), and (c). Define the automorphism on the set of vertices of $\Z Q$ by taking $T^{(i)}$ to $T^{(i+1)}$ for $0\leq i< l-1$, and $T^{(l-1)}$ to $\tau^{-1}T^{(0)}$, via $F$. It is easily seen that this extends to the automorphism of $\Z Q$ whose $l$-th power is $\tau^{-1}$.%It clearly gives a bijection $\{s^{(i)}\to t^{(j)}\}\to\{s^{(i+1)}\to t^{(j+1)}\}$ for each $0\leq i\leq j<l-1$ or $i=j=l-1$. Also when $i<j=l-1$, it gives a bijection $\{s^{(i)}-t^{(l-1)}\}\simeq\{s^{(i+1)}-t\}$ by (c), thus $\{s^{(i)}\to t^{(l-1)}\}\simeq\{s^{(i+1)}-t\}$ and (b), 
\end{proof}

We end this section with a complete list of connected Dynkin quivers whose infinite translation quivers have roots of $\tau$.
\begin{Ex}
Let $Q$ be a connected Dynkin quiver and $l\geq2$. Then {\it $\Z Q$ has an $l$-th root of $\tau$ if and only if $l=2$ and $Q$ is of even type $A$}. Indeed, suppose that $\Z Q$ has an $l$-th root of $\tau$. We use that the underlying graph of $Q$ has a free action of $\Z/l\Z$ (\ref{class}(c)). We can exclude type $D$ and $E$ since it has only one trivalent node. Then $Q$ has to be of type $A$, which has exactly two univalent nodes. This forces $l=2$ and thus the number of vertices to be even. Conversely if $Q$ is of even type $A$, then the automorphism of $\Z Q$ which turns up-side-down as in \ref{upsidedown} gives a square root of $\tau$.
\end{Ex}
%\begin{Ex}
%Let $Q$ be a connected extended Dynkin quiver and $l\geq2$.
%\begin{enumerate}
%	\item First, {\it the type $E$ quivers have no non-trivial roots of $\tau$.} Indeed, they have only one trivalent nodes.
%	\item Suppose next that $Q$ is of type $D_n$. Then {\it $\Z Q$ has an $l$-th root of $\tau$ if and only if $l=2$ and $n$ is odd}. Indeed, the extended type $D$ quivers have at most two trivalent nodes, thus we can only have $l=2$. This forces the number of vertices to be even, thus $n$ is odd. Conversely any of these quivers has a symmetry described in \ref{class}.
%	\item Finally we consider type $A$. Recall that any extended Dynkin quiver of type $\widetilde{A}$ is derived equivalent to the quiver with $p$ clockwise oriented arrows and $q$ counter-clockwise oriented arrows with $p\geq q\geq1$, which we denote by $\widetilde{A}_{p,q}$.
%\end{enumerate}
%\end{Ex}

\section{Application: Calabi-Yau reduction of cluster categories}
Let us start with a reduction process, called {\it Calabi-Yau reduction}, of triangulated categories, which yields a smaller CY triangulated category from a given one.
\begin{Thm}[{\cite[Section 4]{IYo}}]\label{red}
Let $\T$ be a $d$-CY triangulated category and $\mathscr{P}\subset\T$ a functorially finite $d$-rigid subcategory. Put $\mathscr{Z}=\{X\in\T\mid \T(P,X[i])=0 \text{ for all } P\in\mathscr{P} \text{ and } 0<i<d\}$.
\begin{enumerate}
	\item The additive quotient $\U:=\mathscr{Z}/[\mathscr{P}]$ has a natural structure of a triangulated category, which is $d$-CY.
	\item The projection $\mathscr{Z}\to\U$ induces a bijection between the set of $d$-cluster tilting subcategories of $\T$ containing $\mathscr{P}$, and the set of $d$-cluster tilting subcategories of $\U$.%There is a natural bijection $\ct{d}_P{\T}\simeq\ct{d}{\U}$.
	\item The projection $\mathscr{Z}\to\U$ preserves AR $(d+2)$-angles, that is, if $\C\subset\T$ be a $d$-cluster tilting subcategory containing $\mathscr{P}$, then for each indecomposable $C\in\mathscr{Z}\setminus\mathscr{P}$ the image of the AR $(d+2)$-angle at $C$ in $\C$ is the AR $(d+2)$-angle at $C$ in $\U$.
\end{enumerate} 
\end{Thm}
We apply this CY reduction to an important class of CY triangulated categories with cluster tilting objects, namely the cluster categories of finite dimensional algebras. Let $A$ be a finite dimensional algebra which is {\it $\nu_d$-finite} \cite{Iy11,Am09}, that is, we have $\gd A\leq d$ and $\Hom_{\D(A)}(A,\nu_d^{-i}A)=0$ for almost all $i\in\Z$ where $\nu=-\lotimes_A DA$ and $\nu_d=\nu\circ[-d]$. Then the {\it $d$-cluster category $\C_d(A)$} of $A$ is the cluster category of the $(d+1)$-CY completion $\mathbf{\Pi}_{d+1}(A)$ of $A$, which is also the triangulated hull of the orbit category $\D^b(\md A)/\nu_d$;
\[ \D^b(\md A)/\nu_d\hookrightarrow\C_d(A)=\C(\mathbf{\Pi}_{d+1}(A)). \]
It is a $\Hom$-finite $d$-CY triangulated category with a $d$-cluster titling obejct $A\in\C_d(A)$ \cite{Am09,Guo}, whose endomorphism algebra is the $(d+1)$-preprojective algebra $\Pi_{d+1}(A)$ \cite{IO13} of $A$.

Recall that an idempotent $e$ is {\it stratifying} if the restriction functor $\D(\Md A/(e))\to\D(\Md A)$ along $A\to A/(e)$ is fully faithful. It is known \cite{IYa1,Ke11} that the CY reduction of the cluster category $\C_d(A)$ with respect to a stratifying idempotent $e\in A$ is the cluster category $\C_d(A/(e))$. In this section, we apply our main result to give a description of the CY reduction by a {\it non}-stratifying idempotent, and in particular observe that it is not necessarily the cluster category of the quotient algebra. %moreover, not even a cluster category of the form $\C_d(A^\prime)$ for some algebra $A^\prime$.

Applying \ref{kmv} to a CY reduction of the $3$-cluster category gives the following result.
\begin{Prop}\label{nu3}
Let $A$ be a finite dimensional algebra which is $\nu_3$-finite and $\Pi$ its $4$-preprojective algebra. Suppose $e\in A$ is idempotent such that $\Pi/(e)$ is hereditary. Consider the CY reduction $\U$ of $\C_3(A)$ with respect to $eA$.
\begin{enumerate}
	\item The algebra $H=\End_\U(A\oplus A[-1])$ is hereditary.
	\item When $H$ is $1$-representation infinite and $H/J_H$ is separable over $k$, there is triangle equivalence $\T\simeq\D^b(\md H)/\tau^{-1/2}[1]$.
\end{enumerate} 
\end{Prop}
In the rest of this subsection let $A$ be a finite dimensional algebra of global dimension $\leq2$, and $\C$ its $3$-cluster category. Then the $4$-preprojective algebra $\C(A,A)$ of $A$ is just $A$. Let us describe the hereditary algebra $H$ explicitly in this case. For this we give a general description of AR $5$-angles at $A$ in $\C$.
We start with an easy computation.
\begin{Lem}\label{ext}
We have $\C(A,A[-1])=\Ext^2_A(DA,A)$.
\end{Lem}
\begin{proof}
	This is easily seen by $\C(A,A[-1])=\coprod_{i\in\Z}\Hom_{\D(A)}(A,\nu_3^{-i}A[-1])$, in which only $i=1$ survives.
\end{proof}
Let $P$ be an indecomposable direct summand of $A$ and $S$ the corresponding simple $A$-module. We let
\[ \xymatrix@R=1mm{
	0\ar[r]&P_2\ar[r]&P_1\ar[r]&P_0\ar[r]&S\ar[r]&0\\
	0\ar[r]&S\ar[r]&I^0\ar[r]&I^1\ar[r]&I^2\ar[r]&0 } \]
the minimal projective and injective resolutions of $S$, thus $P_0=P$ and $I^0=\nu P_0$. Clearly the AR $5$-angle in $\add A\subset\C$ at $P$ is of the form
\[ \xymatrix@!C=3mm@!R=2mm{
	&\nu^{-1}I^1\ar[dr]\ar[rr]&&Q\ar[dr]\ar[rr]&&P_1\ar[dr]&\\
	\nu^{-1}I^0\ar[ur]&&\bullet\ar[ur]&&X\ar[ur]&&P_0 } \]
for some $Q\in\add A$ and $X\in\C$. In fact we can also determine $Q$.
\begin{Lem}\label{middle}
We have $Q=P_2\oplus \nu^{-1}I^2$.
\end{Lem}
\begin{proof}
	Since $Q\to X$ is a minimal right $(\add A)$-approximation in $\C$, the morphism $\C(A,Q)\to\C(A,X)$ is a projective cover, thus we have to determine the $A$-module $\C(A,X)$. By the first triangle we have an exact sequence
	\[ \xymatrix{
		\C(A,P_1[-1])\ar[r]&\C(A,P_0[-1])\ar[r]&\C(A,X)\ar[r]&\C(A,P_1)\ar[r]&\C(A,P_0) }. \]
	The rightmost map is isomorphic to $P_1\to P_0$, thus has projective kernel $P_2$. On the other hand, the leftmost map is isomorphic by \ref{ext} to $\Ext^2_A(DA,P_1)\to\Ext^2_A(DA,P_0)$, thus its cokernel is $\Ext^2_A(DA,S)$ since $\Ext^2_A(DA,-)$ is right exact by $\gd A\leq2$. Therefore we obtain a split short exact sequence
	\[ \xymatrix{0\ar[r]&\Ext^2_A(DA,S)\ar[r]&\C(A,X)\ar[r]&P_2\ar[r]&0 }. \]
	Now by the minimal injective resolution of $S$ we have a projective presentation
	\[ \xymatrix{ \nu^{-1}I^1\ar[r]&\nu^{-1}I^2\ar[r]&\Ext^2_A(DA,S)\ar[r]&0}, \]
	which gives the projective cover $\nu^{-1}I^2$ of $\Ext^2_A(DA,S)$. We conclude that the projective cover of $\C(A,X)=P_2\oplus\Ext^2_A(DA,S)$ is $P_2\oplus \nu^{-1}I^2$.
\end{proof}

Now we describe the algebra $H$. Suppose $A$ is given by a quiver with relations; $A=kQ/I$ and let $R$ be a minimal set of relations generating $I$. Recall that the minimal projective and injective resolutions of simple modules are controlled by the arrows and the relations (\cite[Section 3]{BIRSm}), precisely, we have exact sequences
\[ \xymatrix{ 0\ar[r]&\discoprod_{\rho\colon c\to a}e_cA\ar[r]&\discoprod_{\a\colon b\to a}e_bA\ar[r]&e_aA\ar[r]&S_a\ar[r]& 0 }, \]
where the first (resp. second) sum runs over the relations $\rho$ in $R$ (resp. paths $\a$) ending at $a$, and dually
\[ \xymatrix{ 0\ar[r]&S_a\ar[r]&D(Ae_a)\ar[r]&\discoprod_{\a\colon a\to b}D(Ae_b)\ar[r]&\discoprod_{\rho\colon a\to c}D(Ae_c)\ar[r]& 0}. \]
We therefore obtain the following explicit description of $H$.
\begin{Thm}
Let $A=kQ/(R)$ be a finite dimensional algebra of global dimension $\leq2$ given by quiver $Q$ and a minimal set of relations $R$. Let $e\in A$ be an idempotent such that $A/(e)$ is hereditary, and consider the CY reduction $\U$ of the $3$-cluster category of $A$ with respect to $eA$.
\begin{enumerate}
	\item The algebra $H=\End_\U(A\oplus A[-1])$ is the path algebra of the quiver $\widetilde{Q}$ obtained as follows.
	\begin{enumerate}
	\renewcommand{\labelenumii}{(\roman{enumii})}
	\item Consider two copies of the quiver $\underline{Q}$ obtained from $Q$ by deleting the vertices corresponding to $e$. For each vertex $a\in\underline{Q}$ we denote by $a^\prime$ the corresponding vertex in the other copy.
	\item For each relation $a\to b$ in $R$ with $a,b\in \underline{Q}$, add arrows $a\to b^\prime$ and $b\to a^\prime$.
	\end{enumerate}
	\item If each connected component of $\widetilde{Q}$ is non-Dynkin, then there exists a triangle equivalence
	\[ \U\simeq\D^b(\md k\widetilde{Q})/\tau^{-1/2}[1]. \]
\end{enumerate}
\end{Thm}
\begin{proof}
	The triangle equivalence is given in \ref{nu3} so we only have to prove (1). Clearly the quiver of $H$ has two copies of $\underline{Q}$ as a subquiver, and by \ref{quiver} the arrows between these copies are given by the middle terms of the AR $5$-angles in $\U$. By \ref{red}(3) they are the image of the ones in $\C$, whose middle terms are as in \ref{middle}. We then deduce the result by the remark on resolutions of simple modules. 
\end{proof}

The following example shows that a CY reduction of a usual $3$-cluster category of a finite dimensional algebra is {\it not} a usual $3$-cluster category; it involves a square root of the AR translation.
\begin{Ex}
Let $A$ be the algebra presented by the following quiver with relations, which has global dimension $2$.
\[ \xymatrix@R=0.1mm{
	&2\ar[dr]^-b&\\
	1\ar[dr]_-c\ar[ur]^-a&&4&dc=0.\\
	&3\ar[ur]_-d&\quad, } \]
Let $e=e_3$ be the idempotent corresponding to the vertex $3$. Then the quotient $A/(e_3)$ is hereditary. Letting $\U$ be the CY reduction of the $3$-cluster category $\C_3(A)$ with respect to $e_3A$, we have a triangle equivalence
\[ \U\simeq \D^b(\md H)/\tau^{-1/2}[1], \]
where $H$ is the path algebra of the following quiver of type $\widetilde{A_5}$.
\[ \xymatrix@R=4mm{
	1\ar[r]\ar[drr]&2\ar[r]&4\ar[dll]\\
	1^\prime\ar[r]&2^\prime\ar[r]&4^\prime } \]
Note that $\D^b(\md H)/\tau^{-1/2}[1]$ has infinitely many indecomposables, so it cannot be equivalent to the cluster category of the quotient algebra $A/(e_3)$.
\end{Ex}
	
\section{Application: Singularity categories of truncated skew group algebras}
%We apply our main result to singularity categories of truncated skew group algebras. This forms a nice class of CY triangulated category endowed with a cluster tilting object whose endomorphism algebra can be described explicitly by the McKay quiver of the given group.
Throughout this section let $k$ be an algebraically closed field of characteristic $0$, $S=k[x_0,\ldots,x_d]$ the polynomial ring, and $G$ a finite subgroup of $\SL_{d+1}(k)$, which acts naturally on $S$. Let $R=S^G$ the invariant ring and $S\ast G$ the {skew group algebra}, that is, the vector space $S\otimes_kkG$ with multiplication $(s\otimes g)(t\otimes h)=sg(t)\otimes gh$ for $s,t\in S$ and $g,h\in G$.
The following result gives examples of CY triangulated categories with cluster tilting objects.
\begin{Thm}[{\cite[10.1]{IYo}, \cite[2.3]{AIR}}]\label{EX}
Let $e$ be an idempotent of $\G=S\ast G$ such that $\G/(e)$ is finite dimensional. Put $\L=e\G e$.
\begin{enumerate}
	\item The algebra $\L$ is a symmetric $R$-order, thus the singularity category $\sCM\L$ is $d$-CY.
	\item The $\L$-module $T=\G e$ is Cohen-Macaulay and is $d$-cluster tilting in $\CM\L$.
	\item We have $\sEnd_\L(T)=\G/(e)$.
\end{enumerate}
\end{Thm}
Recall that AR sequences in $\add\G e\subset\sCM\L$ can be obtained from the Koszul complex \cite{Iy07b}. Let $V$ be the vector space with basis $\{x_0,\ldots,x_d\}$, so that we have the Koszul complex of $S$:
\[ \xymatrix{ 0\ar[r]& S\otimes_k\bigwedge^{d+1}V\ar[r]&S\otimes_k\bigwedge^{d}V\ar[r]&\cdots\ar[r]&S\otimes_kV\ar[r]&S\ar[r]&k\ar[r]& 0}. \]
Let $M=eS\in\md R$ be the direct summand of $S$ corresponding to the idempotent $e\in\G=\End_R(S)$. Applying $\Hom_R(M,-)$ yields an exact sequence
\[ \xymatrix{ 0\ar[r]& \Hom_R(M,S)\otimes_k\bigwedge^{d+1}V\ar[r]&\cdots\ar[r]&\Hom_R(M,S)\otimes_kV\ar[r]&\Hom_R(M,S)}, \]
in $\CM\L$, which is a direct sum of $d$-almost split sequences and $d$-fundamental sequences in $\add\G e$. Passing these to the stable category $\sCM\L$ gives the AR $(d+2)$-angles.

In what follows we use a standard notation for elements in $\GL_n(k)$: for $a, a_1,\ldots,a_n\in\Z$, we write $1/a(a_1,\ldots,a_n)$ for $\diag(\zeta^{a_1},\ldots,\zeta^{a_n})\in\GL_n(k)$ for a fixed primitive $a$-th root of unity $\zeta$.
\begin{Lem}\label{kakunin}
Let $G\subset\SL_{d+1}(k)$ is the cyclic subgroup generated by $1/n(a_0,\ldots,a_d)$. %, and the skew group ring $\G$ is presented by $Q$ with commutativity relations.
\begin{enumerate}
	\item The McKay quiver $Q$ of $G$ has vertices $\Z/n\Z=\{0,1,\ldots,n-1\}$ and arrows $x_i\colon j\to j+a_i$.
	\item The skew group algebra $\G$ is presented by $Q$ with commutativity relations.
\suspend{enumerate}
We view $Q$ as an $\{x_0,\ldots,x_d\}$-colored quiver. Let $e\in\G$ be an idempotent and let $I$ be the corresponding subset of the vertices of $Q$.
\resume{enumerate}
	\item $\G/(e)$ is presented by the quiver $Q\setminus I$ obtained from $Q$ by removing the vertices in $I$.
	\item $\G/(e)$ is finite dimensional hereditary if and only if there are no cycles consisting of arrows of a single color, nor composable arrows of different colors in $Q\setminus I$.
\end{enumerate}
\end{Lem}
\begin{proof}
	(1) is obvious, (2) is well-known (see e.g. \cite[2.8]{CMT}\cite[4.1]{BSW}), and (3) is then clear. We verify (4). If $Q\setminus I$ has no cycles of a single color and no composable arrows of different colors, then there are no relations on $Q\setminus I$ since $\G$ has only commutativity relations. Also the assumption implies that $Q\setminus I$ is acyclic. Therefore $\G/(e)$ is finite dimensional and hereditary. Conversely, if $Q\setminus I$ has a cycle of a single color, then $\G/(e)$ is infinite dimensional. If $Q\setminus I$ has composable arrows of different colors, then it has a relation. Indeed, let $a\xrightarrow{x_i}b\xrightarrow{x_j}c$ be a subquiver of $Q\setminus I$ with $i\neq j$. Then $Q$ has a subquiver of the form
	\[ \xymatrix@R=0.1mm{
		&b\ar[dr]^-{x_j}&\\
		a\ar[dr]_-{x_j}\ar[ur]^-{x_i}&&c\\
		&b^\prime\ar[ur]_-{x_i}&\quad . } \]
	If $b^\prime\in Q\setminus I$ then $Q\setminus I$ has a commutativity relation, and if $b^\prime\not\in Q\setminus I$ then $Q\setminus I$ has a zero relation.
\end{proof}
Let $G=\left\langle 1/n(a_0,\ldots,a_d)\right\rangle \subset\SL_{d+1}(k)$ and its McKay quiver $Q$ as in \ref{kakunin} above. We naturally regard $a_i\in\Z/n\Z$. Let $M_j$ be the indecomposable direct summand of $S$ corresponding to the vertices $j$ of $Q$. Then the construction of almost split sequences and a computation of exterior powers show that it is given for $M_j$ by the exact sequence
\[ \xymatrix{ M_j\ar[r]& \discoprod_{0\leq{i_1}<\ldots<i_{d}\leq d}M_{j-a_{i_1}-\cdots-a_{i_{d}}}\ar[r]&\cdots\ar[r]&\discoprod_{0\leq{i_1}<i_{2}\leq d}M_{j-a_{i_1}-a_{i_{2}}}\ar[r]&\discoprod_{0\leq i\leq d}M_{j-a_i}\ar[r]& M_j }. \]

Now let $e\in\G$ is an idempotent as in \ref{EX}, $I$ the corresponding subset of the vertives of $Q$, and $M$ the corresponding direct summand of $S$. Putting $N_j=\Hom_R(M,M_j)\in\CM\L$, the AR sequence at $N_j$ in $\add\G e\subset\sCM\L$ is
\[ \xymatrix{ N_j\ar[r]& \discoprod_{\substack{0\leq{i_1}<\ldots<i_{d}\leq d\\j-a_{i_1}-\cdots-a_{i_{d}}\not\in I}}N_{j-a_{i_1}-\cdots-a_{i_{d}}}\ar[r]&\cdots\ar[r]&\discoprod_{\substack{0\leq{i_1}<i_{2}\leq d\\j-a_{i_1}-a_{i_{2}}\not\in I}}N_{j-a_{i_1}-a_{i_{2}}}\ar[r]&\discoprod_{\substack{0\leq i\leq d \\ j-a_i\not\in I}}N_{j-a_i}\ar[r]& N_j }. \]

Our first example is $3$-CY category with a $3$-cluster tilting object with hereditary endomorphism ring, so that we can apply \ref{kmv}.
\begin{Ex}
Let $S=k[x,y,z,w]$ and $G$ the cyclic subgroup of $\SL_4(k)$ generated by $1/n(a_0,a_1,a_2,a_3)$ with $0\leq a_i\leq n-1$. Assume that $a_0=1$, and $a_1,a_2,a_3>1$. Moreover pick an integer $l>0$ such that $l-1\leq a_1,a_2,a_3\leq n-l$ and let $e:=e_0+e_{l+1}+\cdots+e_{n-1}=1-(e_1+\cdots+e_l)$. Putting $M=M_0\oplus M_{l+1}\oplus\cdots\oplus M_{n-1}$ we have $\L=e\G e=\End_R(M)$, and by \ref{EX} the stable category $\sCM\L$ is $3$-CY and $N=N_1\oplus\cdots\oplus N_l\in\sCM\L$ with $N_i:=\Hom_R(M,M_i)$ is $3$-cluster tilting.
By \ref{EX}(4) we see that $\sEnd_\L(N)=\G/(e)$ is hereditary whose quiver as below, where the number of arrows $1\to l$ is the number of $a_1,a_2,a_3$ which equals $l-1$.
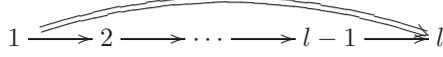
\vspace{15pt}
\[ \xymatrix{ 1\ar[r]\ar@/^15pt/@{-->}@2[rrrr]&2\ar[r]&\cdots\ar[r]&l-1\ar[r]&l} \]
The middle term of the AR $5$-angle at $N_j$ is $\coprod_{0\leq i_1<i_2\leq 3}N_{j-a_{i_1}-a_{i_2}}$, thus equals $N_l^{m_{jl}}$ with $m_{jl}$ the number of pairs $(a_{i_1},a_{i_2})$ such that $a_{i_1}+a_{i_2}=j-l$. Note that $m_{jl}=m_{lj}$ by \ref{quiver}. By \ref{oddcy} the algebra $H=\sEnd_\L(N\oplus N[-1])$ is hereditary, which by \ref{quiver} is presented by the quiver which looks as below, with $m_{jl}$ arrows from $N_j$ to $N_l[-1]$.
\vspace{15pt}
\[ \xymatrix{
	N_1\ar@/^15pt/@{-->}@2[rrrr]\ar[r]\ar@{..>}[dr]\ar@{..>}[drr]&N_2\ar[r]\ar@{..>}[dl]\ar@{..>}[dr]&\cdots\ar@{..>}[drr]\ar@{..>}[dr]\ar@{..>}[dll]\ar@{..>}[dl]\ar[r]&N_{l-1}\ar@{..>}[dr]\ar@{..>}[dl]\ar[r]&N_l\ar@{..>}[dl]\ar@{..>}[dll]\\
	N_1[-1]\ar@/_15pt/@{-->}@2[rrrr]\ar[r]&N_2[-1]\ar[r]&\cdots\ar[r]&N_{l-1}[-1]\ar[r]&N_l[-1]} \vspace{10pt}\]
We obtain by \ref{kmv} a triangle equivalence
\[ \xymatrix{\sCM\L\ar@{-}[r]^-\simeq&\D^b(\md H)/\tau^{-1/2}[1] }. \]
Let us list a specific example of the McKay quiver $Q$, and the hereditary algebra $H=\sEnd_\L(N\oplus N[-1])$ for $G=\left\langle 1/5(1,3,3,3)\right\rangle$ and $l=2$.%, (ii) $G=\left\langle 1/9(1,5,6,6)\right\rangle$ and $l=3$.
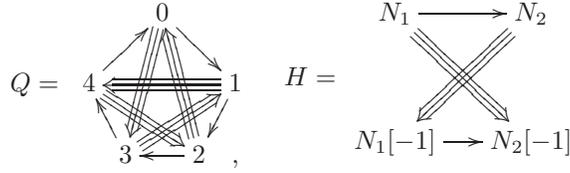
\[
\xymatrix@C=1mm@!R=1mm{
	&&&0\ar[drr]\ar@3[ddl]&&\\
	{Q=}&4\ar[urr]\ar@3[drrr]&&&&1\ar[dl]\ar@3[llll]\\
	&&3\ar[ul]\ar@3[urrr]&&2\ar[ll]\ar@3[uul]&, }
\xymatrix@C=3mm@!R=0.1mm{
	&N_1\ar[rr]\ar@3[ddrr]&&N_2\ar@3[ddll]\\
	\quad H=\\
	&N_1[-1]\ar[rr]&&N_2[-1] } \]
%\[
%\xymatrix@C=1mm@R=1mm{
%	&&&0\ar[rrd]\ar[lddddddd]\ar@2@{<-}[rrrddddd]&&&\\
%	&8\ar[rru]\ar[rrrdddddd]\ar@2@{<-}[rrrrrdd]&&&&1\ar[rdd]\ar[llllldddd]\ar@2@{<-}[ldddddd]&\\
%	\\
%	7\ar[ruu]\ar[rrrrrrdd]\ar@2@{<-}[rrrrruu]&&&&&&2\ar[dd]\ar[llllll]\ar@2@{<-}[lllldddd]\\
%	\\
%	6\ar[uu]\ar[rrrrrruu]\ar@2@{<-}[rrruuuuu]&&&&&&3\ar[lldd]\ar[llllluuuu]\ar@2@{<-}[llllll]\\
%	\\
%	&&5\ar[lluu]\ar[rrruuuuuu]\ar@2@{<-}[luuuuuu]&&4\ar[ll]\ar[luuuuuuu]\ar@2@{<-}[lllluuuu]&& }
%\xymatrix{
%	N_1\ar[r]\ar@2[drr]&N_2\ar[r]&N_3\ar@2[dll]\\
%	N_1[-1]\ar[r]&N_2[-1]\ar[r]&N_3[-1] } \]
%For example if $G=\left\langle 1/5(1,3,3,3)\right\rangle$ and $l=2$, the two AR $5$-angles in $\add(N_1\oplus N_2)\subset\sCM\L$ are
%\[ \xymatrix@R=1mm{
%	&N_2\ar[rr]\ar[dr]&& N_2^{\oplus3}\ar[rr]\ar[dr]&&0\ar[dr] \\
%	N_1\ar[ur]&&\bullet\ar[ur]&&\bullet\ar[ur]&&N_1 \\
%	&0\ar[rr]\ar[dr]&& N_1^{\oplus3}\ar[rr]\ar[dr]&&N_1\ar[dr]& \\
%	N_2\ar[ur]&&\bullet\ar[ur]&&\bullet\ar[ur]&&N_2, } \]
%so by \ref{oddcy} and \ref{quiver} the algebra $H$ is the path algebra of the quiver below, and by \ref{kmv} we deduce a triangle equivalence.	 
\end{Ex}
The next example is a $4$-CY category $\sCM\L$ with a semisimple $4$-cluster tilting object $N$. In this case the algebra $\sEnd_\L(N\oplus N[-1])$ is clearly hereditary, so we can apply \ref{new}.
\begin{Ex}
Let $G\subset\SL_{5}(k)$ be the cyclic subgroup generated by $1/n(a_0,a_1,a_2,a_3,a_4)$ with $1\leq a_i\leq n-1$.

(1)  Let $e_0$ a primitive idempotent of $\G=S\ast G$, $e=1-e_0$, and $M_0$ (resp. $M$) the $R$-module summand of $S$ corresponding to $e_0$ (resp. $e$). Then the almost split sequence at $M_0$ in $\CM R$ is given by
\[ \xymatrix{M_0\ar[r]&\discoprod_{0\leq i\leq4}M_{a_i}\ar[r]&\discoprod_{0\leq i_1<i_2\leq 4}M_{a_{i_1}+a_{i_2}}\ar[r]&\discoprod_{0\leq i_1<i_2\leq 4}M_{-a_{i_1}-a_{i_2}}\ar[r]&\discoprod_{0\leq i\leq4}M_{-a_i}\ar[r]&M_0}. \]
Therefore letting $m$ be the number of pairs $(a_{i_1},a_{i_2})$ with $0\leq i_1<i_2\leq 4$ such that $a_{i_1}+a_{i_2}=n$, the AR $6$-angle in $\sCM\L$ at $N_0=\Hom_R(M,M_0)$, where $\L=e\G e$, is
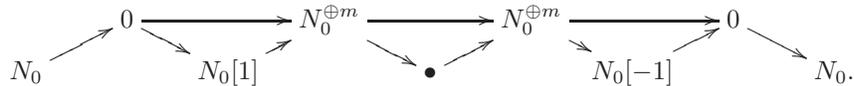
\[ \xymatrix@R=1mm@!C=5mm{
	&0\ar[dr]\ar[rr]&&N_0^{\oplus m}\ar[dr]\ar[rr]&&N_0^{\oplus m}\ar[dr]\ar[rr]&&0\ar[dr]& \\
	N_0\ar[ur]&&N_0[1]\ar[ur]&&\bullet\ar[ur]&&N_0[-1]\ar[ur]&&N_0. } \]
It follows from \ref{quiver2} that $H=\sEnd_\L(N_0\oplus N_0[-1]\oplus N_0[-2])$ is the path algebra of the following quiver of type $\widetilde{A_2}$ with $m$-fold arrows, and from \ref{new} that there exists a triangle equivalence below.
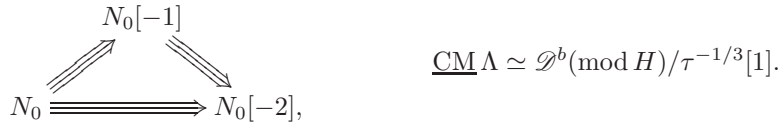
\[ \xymatrix@R=0.1mm@!C=7mm{
	&&N_0[-1]\ar@3[ddr]&\\
	&&&&&&\sCM\L\simeq\D^b(\md H)/\tau^{-1/3}[1].\\
	&N_0\ar@3[uur]\ar@3[rr]&&N_0[-2],} \]

(2)  More generally, let $I_0$ be a subset of $\Z/n\Z$ such that for each $j,j^\prime\in I_0$ we have $j-j^\prime\neq a_i$ for any $0\leq i\leq4$. Then letting $e$ be the idempotent of $\G$ corresponding to $J:=\Z/n\Z\setminus I_0$, the algebra $\G/(e)$ is semisimple. Also the AR $6$-angle in $\add\G e\subset\sCM\L$ is
\[ \xymatrix@R=3mm@!C=5mm{
	&0\ar[dr]\ar[rr]&&\discoprod_{l\in J}N_l^{\oplus m_{jl}}\ar[dr]\ar[rr]&&\discoprod_{l\in J}N_l^{\oplus m_{lj}}\ar[dr]\ar[rr]&&0\ar[dr]& \\
	N_j\ar[ur]&&N_j[1]\ar[ur]&&\bullet\ar[ur]&&N_j[-1]\ar[ur]&&N_j, } \]
where $m_{jl}$ is the number of pairs $(a_{i_1},a_{i_2})$ with $0\leq i_1<i_2\leq 4$ such that $a_{i_1}+a_{i_2}=l-j$ in $\Z/n\Z$. By \ref{quiver2} the algebra $H=\sEnd_\L(N\oplus N[-1]\oplus N[-2])$ is the path algebra of the quiver given as follows.
\begin{itemize}
	\item vertices are $J\times\{0,-1,-2\}$,
	\item $m_{jl}$ arrows from $(j,0)$ to $(l,-1)$ and from $(j,-1)$ to $(l,-2)$, and $m_{lj}$ arrows from $(j,0)$ to $(l,-2)$.
\end{itemize}
The quiver of $H$ looks as in the following picture, where the same kind of arrows indicates that there are same number of arrows. By \ref{new} we have a triangle equivalence below.
\[ \xymatrix@R=5mm@!C=12mm{
	\cdots\ar@{~>}[dr]&N_j\ar[dr]\ar@2[d]\ar@2@/_8mm/[dd]\ar@{-->}[ddr]\ar@{~>}[ddl]&N_l\ar@2[d]\ar@2@/^8mm/[dd]\ar[ddl]\ar@{-->}[dl]\ar@{..>}[dr]&\cdots\ar@{..>}[ddl]\\
	\cdots\ar@{~>}[dr]&N_j[-1]\ar[dr]\ar@2[d]&N_l[-1]\ar@2[d]\ar@{-->}[dl]\ar@{..>}[dr]&\cdots&&\sCM\L\simeq\D^b(\md H)/\tau^{-1/3}[1]\\
	\cdots&N_j[-2]&N_l[-2]&\cdots } \]
We end this paper with some specific examples for $G=\left\langle 1/6(1,1,1,4,5)\right\rangle$ and (i) $I_0=\{0\}$, (ii) $I_0=\{0,3\}$, which describe the McKay quiver of $G$ and the hereditary algebra $H=\sEnd_\L(N\oplus N[-1]\oplus N[-2])$ in each case.
\[ \xymatrix@R=1mm@C=4mm{
	&&0\ar@3[dr]\ar[dddl]\ar@/_8pt/[dl]&\\
	&5\ar@3[ur]\ar[dddr]\ar@/_8pt/[dd]&&1\ar@3[dd]\ar[ll]\ar@/_8pt/[ul]\\
	\\
	&4\ar@3[uu]\ar[rr]\ar@/_8pt/[dr]&&2\ar@3[dl]\ar[uuul]\ar@/_8pt/[uu]\\
	&&3\ar@3[ul]\ar[uuur]\ar@/_8pt/[ur]& }
\hspace{20mm}
{\rm (i)}
\xymatrix@R=1mm@C=3.6mm{
	&N_0\ar@3[dd]\ar@/_8mm/@3[dddd]\\ \\ H=&N_0[-1]\ar@3[dd]\\ \\&N_0[-2] }
\quad
{\rm (ii)}
\xymatrix@R=1mm@C=3.6mm{
	&N_0\ar@3[dd]\ar@/_8mm/@3[dddd]\ar[ddr]\ar[ddddr]&N_3\ar@3[dd]\ar@/^8mm/@3[dddd]\ar[ddddl]\ar[ddl]\\ \\ H=&N_0[-1]\ar@3[dd]\ar[ddr]&N_3[-1]\ar@3[dd]\ar[ddl]\\ \\&N_0[-2]&N_3[-2] } \]
\end{Ex}

\end{document}